\documentclass[notitlepage,reqno,11pt]{amsart}
\usepackage{latexsym,amssymb, epsfig, amsmath,amsfonts, subfigure,amsthm}

\usepackage{rotating}
\usepackage[toc,page]{appendix}
\usepackage{color}
\usepackage{multirow}
\usepackage{relsize}
\usepackage{microtype}

\usepackage[colorlinks, allcolors=blue]{hyperref}

\usepackage[margin=1in]{geometry}

\numberwithin{equation}{section}
 \newtheorem{assumption}{Assumption}[section]
\newtheorem{lemma}{Lemma}[section]
\newtheorem{theorem}{Theorem}[section]
\newtheorem{definition}{Definition}[section]

\newtheorem{prop}{Proposition}[section]
\newtheorem{remark}{Remark}[section]

\newlength{\defbaselineskip}
\newcommand{\setlinespacing}[1]%
           {\setlength{\baselineskip}{#1 \defbaselineskip}}

\newcommand{\RR}{{\mathbb R}}

\newcommand{\NN}{{\mathbb N}}

\def\E{\mathbb{E}}
\def\P{\mathbb{P}}

\newcommand{\sG}{{\mathcal{G}}}
\newcommand{\sF}{{\mathcal{F}}}

\newcommand{\beql}[1]{\begin{equation}\label{#1}}
\newcommand{\eeq}{\end{equation}}

\newcommand{\beqal}[1]{\begin{eqnarray}\label{#1}}
\newcommand{\eeqa}{\end{eqnarray}}
\newcommand{\beq}{\begin{displaymath}}
\newcommand{\eeqno}{\end{displaymath}}
\newcommand{\bali}[1]{\begin{align}\label{#1}}
\newcommand{\eali}{\begin{align}}
\newcommand{\balino}{\begin{align*}}
\newcommand{\ealino}{\begin{align*}}

\newcommand{\ep}{\epsilon}

\newcommand{\Cov}{\text{\rm Cov}}

\newcommand{\bone}{{\mathbf 1}}

\newcommand{\qandq}{\quad\mbox{and}\quad}

\newcommand{\qasq}{\quad\mbox{as}\quad}
\newcommand{\qinq}{\quad\mbox{in}\quad}

\newcommand{\non}{\nonumber}
\newcommand{\RA}{\Rightarrow}

\newcommand{\baa}{\begin{eqnarray*}}
\newcommand{\eaa}{\end{eqnarray*}}

\newcommand{\ttl}{\Large Functional Limit Theorems for  Non-Markovian Epidemic Models}

\newcommand{\ttls}{\large Functional Limit Theorems for Non-Markovian Epidemic Models}

\begin{document}

\title[\ttls]{\ttl}

\author[Guodong \ Pang]{Guodong Pang}
\address{The Harold and Inge Marcus Department of Industrial and
Manufacturing Engineering,
College of Engineering,
Pennsylvania State University,
University Park, PA 16802 USA }
\email{gup3@psu.edu}

\author[{\'E}tienne \ Pardoux]{{\'E}tienne Pardoux}
\address{Aix--Marseille Universit{\'e}, CNRS, Centrale Marseille, I2M, UMR \ 7373 \ 13453\ Marseille, France}
\email{etienne.pardoux@univ.amu.fr}


\begin{abstract} 
We study non-Markovian stochastic epidemic models (SIS, SIR, SIRS, and SEIR),  in which the infectious (and latent/exposing, immune) periods have a general distribution. 
We provide a representation of  the evolution dynamics using the time epochs of infection (and latency/exposure, immunity). 
Taking the limit as the size of the population tends to infinity, we prove both a functional law of large number (FLLN) and a functional central limit theorem (FCLT) for the processes of interest in these models. 
In the FLLN, the limits are a unique solution to a system of deterministic Volterra integral equations, while in the FCLT, the limit processes are multidimensional Gaussian solutions of linear Volterra stochastic integral equations. In the proof of the FCLT, we provide an important Poisson random measures representation of the diffusion-scaled processes converging to  Gaussian components driving the limit process. 
\end{abstract}


\maketitle

\allowdisplaybreaks

\section{Introduction}

There have been extensive studies of Markovian epidemic models, including the SIS, SIR, SIRS and SEIR models, see, e.g., \cite{anderson1992infectious,andersson2012stochastic,britton2018stochastic} for an overview. 
Limited work has been done for non-Markovian epidemic models,  with general infectious periods, exposing and/or  immune periods, etc. 
Chapter 3 of \cite{britton2018stochastic} provides a good review of the existing literature on the non-Markovian closed epidemic models. 
There is a lack of  functional law of large numbers (FLLN)  and functional central limit theorems (FCLT) for non-Markovian epidemic models. 

In this paper we study some well known non-Markovian epidemic models, including SIR, SIS, SEIR and SIRS models. In all these models, the process counting the cumulative number of individuals becoming infectious is Poisson as usual with a rate depending on the susceptible and infectious populations. 
In the SIR and SIS models, the infectious periods are assumed to be i.i.d. with any general distribution.
 In the SEIR model, the exposing (latent) and infectious periods are assumed to be i.i.d. random vectors with a general joint distribution (correlation between these two periods for each individual is allowed).

We provide a general representation of the evolution dynamics in these epidemic models, by tracking the time epochs that each individual experiences.  
In the SIR model, each individual has two time epochs, times of becoming infectious and immune (recovered). In the SEIR model, each individual has three time epochs, times of becoming exposed (latent), infectious and immune (recovered). Then the process counting the number of infectious individuals can be simply represented by using these time epochs.

With these representations, we proceed to prove the FLLN and FCLT for these non-Markovian epidemic models. The results for the SIS model directly follow from those of the SIR model, and similarly, the results for the SIRS model follow from those of the SEIR model, so we focus on the studies of the SIR and SEIR models, and the results of the SIS and SIRS are stated without proofs. 
The fluid limits for these non-Markovian models are given as the unique deterministic solution to a system of Volterra integral equations.  We also analyze the equilibrium behaviors in the SIS and SIRS models (see Proposition \ref{prop-SIRS-equilibrium}). 
The limits in the FCLT are solutions of multidimensional linear Volterra stochastic integral equations driven by continuous Gaussian processes. 
These processes are, of course, non-Markovian, but if the initial quantities converge to Gaussian random variables, then the limit processes are
jointly Gaussian.
The Gaussian driving force comes from two independent components.
One corresponds to the initial quantities: in the SIR model, these are initially infected individuals and in the SEIR model, these are initially exposed and infected individuals. 
The other corresponds to the newly infected individuals in the SIR model, and the newly exposed individuals in the SEIR model. These are written as functionals of a white noise with two time dimensions (which can be also regarded as space--time white noise).  
Although the limit processes appear very different in the Markovian case, they are equivalent to the It{\^o} diffusion limit driven by Brownian motions, see, e.g., the proof of Proposition~\ref{prop-SIS-equiv-M} for the SIS model. 

In the proof of the FCLT for the SIR model, we construct a Poisson random measure (PRM) with mean measure depending on the distribution of the infectious periods, such that the diffusion-scaled processes corresponding to the Gaussian process driving the limit can be represented via integrals of white noises. This helps to establish tightness of these diffusion-scaled processes. 
For the SEIR model, the PRM has mean measure depending on the joint distribution of the exposing (latent) and infectious periods. 
It is worth observing the correspondence between the diffusion-scaled processes represented via the PRM and the functionals of the white noise mentioned above. The PRMs are also used to prove tightness in the FLLNs. 
These PRM representations may turn out to be useful for other studies in future work.

 This approach of describing the epidemic dynamics by tracking the ``event" times of each individual and then counting the number of individuals in each compartment with the associated event times, can be used to study many other epidemic models, for example, the SEIJR and SIDARTHE models studied in \cite{chowell2003sars,gumel2004modelling,giordano2020sidarthe}.
It is expected that the FLLN limits for all the compartments will be characterized by solutions to a set of integral equations, where the convolutions of the distribution functions for the durations in the relevant compartments will be used. Similarly, the FCLT limits for the compartments can be characterized by Gaussian-driven stochastic integral equations.

\subsection{Literature review}
The Markovian models, their limiting ODE LLN limit as well as the diffusion approximation of the fluctuations have been well studied in the literature, see the recent survey \cite{britton2018stochastic} and \cite{allen2017primer}.
 Note that a number of papers start from the ODE model, and make it stochastic by replacing some of the coefficients by stochastic processes, see, e.g., \cite{gray2011stochastic}; our work is not connected to this kind of models. 
One common approach to study non-Markovian epidemic models is by 
Sellke \cite{sellke1983asymptotic}. He provided a construction to define the epidemic outbreak in continuous time using two sets of i.i.d. random variables, with which one can find the distribution of the number of remaining uninfected individuals in an epidemic affecting a large population. 
Reinert \cite{reinert1995asymptotic} generalized Sellke's construction, and proved a deterministic limit (LLN) for the empirical measure describing the system dynamics of the generalized SIR model with the infection rate dependent upon time and state of infection, using Stein's method. From her result, we can derive the fluid model dynamics in Theorem \ref{thm-FLLN-SIR}; however, no FCLTs have been establish using her approach.  A deterministic integral equation for the SEIR model is provided in Chapter 4.5.1 of \cite{BCF-2019}; however, the expression for the infectious function $\bar{I}(t)$ is somewhat different from ours and no FLLN has been established.
 While revising the paper, we found the papers by Wang \cite{wang1975limit, wang1977gaussian} which proved an FLLN as well as a Gaussian limit for the SIR model with the infection rate dependent on the number of infectious individuals,  while assuming a somewhat different initial condition.  That paper \cite{wang1977gaussian} assumes a $C^1$ condition on the infectious distribution for the FCLT, while we have no restriction on this distribution. The proof approach in \cite{wang1975limit, wang1977gaussian} is also different from ours, without using PRMs.
For the SIS model with general infectious periods, without proving an FLLN, the Volterra integral equation was developed to describe the proportion of infectious population, see, e.g., \cite{brauer1975nonlinear,cooke1976epidemic,diekmann1977limiting,hethcote1995sis,van2000simple}. 

Ball \cite{ball1986unified} provided a unified approach to derive the distribution of the total size and total area under the  trajectory of infectives using a Wald's identity for the epidemic process. This was extended to multi-type epidemic models in \cite{ball1993final}.  See also the LLN and CLT results for the final size of the epidemic in \cite{britton2018stochastic}.  Barbour \cite{barbour1975duration} proved  limit theorems for the distribution of the time between the first infection and the last removal in the closed stochastic epidemic. 
See also Section 3.4 in \cite{britton2018stochastic}. 

Clancy \cite{clancy2014sir} recently proposed to view the non-Markovian SIR model as a piecewise Markov deterministic process, and  derived the joint distribution of the number of survivors of the epidemic and the area under the trajectory of infectives using martingales constructed from the piecewise deterministic Markov process. G{\'o}mez-Corral and L{\'o}pez-Garc{\'\i}a \cite{gomez2017sir} further study the piecewise deterministic Markov process in  \cite{clancy2014sir} and analyze the population transmission number and the infection probability of a given susceptible individual.

In a followup work, the LLN limit in the SEIR model has been applied to estimate the state of the Covid-19 pandemic in
\cite{FPP2021}, where statistical  methods are developed to estimate the (unobserved) parameters of the model with limited information during the early stages of the pandemic. It is shown that using ODE compartment models without accounting for the general distributions of the infectious durations may underestimate the basic reproduction number $R_0$. Similar observations are made in \cite{FodorKatzKovacs} where ODE models with delays, corresponding to our models with deterministic infectious periods,  are used to estimate $R_0$ in the early-phase of the Covid-19 pandemic.

It may be worth mentioning the connection with the infinite-server queueing literature. 
It may appear that the infectious process in the SIS or SIR model can be regarded as an infinite-server queue with a state-dependent arrival rate, and the infectious process in the SIRS or SEIR model can be regarded as a tandem infinite-server queue with a state-dependent arrival rate; however there are also delicate differences. See detailed discussions in Remark \ref{sec-SIR-queue} and Section \ref{sec-SEIR}. We refer to the study of $G/GI/\infty$ queues with general i.i.d. service times in \cite{krichagina1997heavy}, \cite{decreusefond2008functional},  \cite{pang2010two} and \cite{reed2015distribution}. In particular, the representation of the infectious population dynamics resembles those of the queueing process of the infinite-server queueing models. However, the results in queueing cannot be directly applied to the epidemic models. Given that the infection process is Poisson with a rate being a function of the infectious and susceptible population sizes, we take advantage of the representations of the epidemic evolution dynamics via Poisson random measures (PRM) and use important properties and results on PRMs and stochastic integrals with respect to PRMs to prove the functional limit theorems.

\subsection{Organization of the paper}
In Section \ref{sec-SIR}, we first describe the SIR model in detail, state the FLLN and the FCLT for the SIR model, and then state the results for the SIS model. This is followed by the studies of the SEIR and SIRS models in   Section \ref{sec-SEIR-SIRS}. 
The proofs of the FLLN and FCLT of the SIR model are given in Sections \ref{sec-SIR-FLLN-proof} and  \ref{sec-SIR-FCLT-proof}, respectively.  
We discuss the special cases of Markovian models in Section \ref{sec-Markovian} and models with deterministic durations in Section \ref{sec-deterministic}, and analyze the equilibrium of the SIS and SIRS models in Section \ref{sec-equilibrium}. 
Those for the SEIR model are then given in Sections \ref{sec-SEIR-FLLN-proof} and  \ref{sec-SEIR-FCLT-proof}.  
In the Appendix, we state the auxiliary result of a system of two linear Volterra equations, and also prove Proposition~\ref{prop-SIS-equiv-M}.

\subsection{Notation}
Throughout the paper, $\NN$ denotes the set of natural numbers, and $\RR^k (\RR^k_+)$ denotes the space of $k$-dimensional vectors
with  real (nonnegative) coordinates, with $\RR (\RR_+)$ for $k=1$.  For $x,y \in\RR$, denote $x\wedge y = \min\{x,y\}$ and $x\vee y = \max\{x,y\}$. 
Let $D=D([0,T], \RR)$ denote the space of $\RR$--valued c{\`a}dl{\`a}g functions defined on $[0,T]$. Throughout the paper, convergence in $D$ means convergence in the  Skorohod $J_1$ topology, see chapter 3 of \cite{billingsley1999convergence}. 
 Also, $D^k$ stands for the $k$-fold product equipped with the product topology.  In particular, for $x^n = (x^n_1,\dots, x^n_k)$ and $x=(x_1,\dots,x_k)$, $x^n\to x$ in $D^k$ if $x^n_i\to x_i$ in $D$ for each $i=1,\dots,k$. We write $D([0,T], \RR^k)$ to indicate the convergence in the Skorohod $J_1$ topology. 
 The difference between the topologies of $D([0,T], \RR^k)$ and $D^k$ is that in the first case the implied time-change is the same for all directions, unlike in the second case. See page 83 in \cite{whitt2002stochastic} for further discussions on these topologies.  
 Let $C$ be the subset of $D$ consisting of continuous functions.    Let $C^1$ consist of all differentiable functions whose derivative is continuous. 
 For any function $x\in D$, we use $\|x\|_T= \sup_{t\in [0,T]} |x(t)|$. For two functions $x,y \in D$, we use $x\circ y(t) = x(y(t))$ denote their composition.
 All random variables and processes are defined in a common complete probability space $(\Omega, \sF, \P)$. The notation $\RA$ means convergence in distribution. We use $\bone(\cdot)$ for indicator function. 
 
 \medskip

\section{SIR and SIS Models with general infectious period distributions}  \label{sec-SIR}

\subsection{SIR Model with general infectious periods}
In the SIR model, the population consists of susceptible, infectious and recovered (immune) individuals, where susceptible individuals get infected through interaction with infectious ones, and then experience an infectious period until becoming immune (no longer subject to infection). 
Let $n$ be the population size. Let $S^n(t)$,  $I^n(t)$ and $R^n(t)$ represent the susceptible, infectious and recovered individuals, respectively, at time $t\ge 0$. (The processes and random quantities are indexed by $n$ and we let $n\to\infty$ in the asymptotic analysis.) 
WLOG, assume that $I^n(0)>0$, $S^n(0) = n - I^n(0)$ and $R^n(0) = 0$, that is,
each individual is either infectious or susceptible at time $0$.

An individual $i$ going through the susceptible-infectious-recovered (SIR) process
has the following time epochs: $\tau^n_i$ and $\tau_i^n+\eta_i$, representing the times of becoming infected and immune, respectively. Here we assume that the infectious period distribution is independent of the population size.  
 For the individuals $I^n(0)$ that are infectious at time $0$, let $\eta_i^0$ be the remaining infectious period. 
Assume that the $\eta_i$'s are i.i.d. with c.d.f. $F$, and $\eta_i^0$ are also i.i.d. with c.d.f. $F_0$. Let $F^c=1-F$ and $F_0^c=1-F_0$.  
Let $\lambda$ be the rate at which infectious individuals infect susceptible ones.

  The infection process is generated by the contacts of infectious individuals with susceptible ones according to a Poisson process with rate $\lambda$. Here we assume a homogeneous population and each infectious contact is chosen uniformly at random among the susceptibles. 
Let $A^n(t)$ be the cumulative process of individuals that become infected by time $t$. Then 
we can express it as 
\begin{equation} \label{eqn-An-SIR}
A^n(t) = A_*\left( \lambda n \int_0^t  \frac{S^n(s)}{n} \frac{I^n(s)}{n} ds \right)
\end{equation}
where $A_*$ is a unit rate Poisson process. The process $A^n(t)$ has event times $\tau_i^n$, $i\in\NN$. 
Assume that $A_*$, $I^n(0)$, $\{\eta^0_i\}$ and $\{\eta_i\}$ are mutually independent. 

We first observe the following balance equations:
\begin{align}
n &=S^n(t)+I^n(t) + R^n(t), \non\\
 S^n(t) &= S^n(0) - A^n(t) = n- I^n(0) - A^n(t),  \non\\ 
 I^n(t) & = I^n(0) + A^n(t) - R^n(t), \non
\end{align}
for each $t\ge 0$. The dynamics of $I^n(t)$ is given by
\begin{align}\label{eqn-SIR-In}
I^n(t) = \sum_{j=1}^{I^n(0)} \bone(\eta^0_j > t)+ \sum_{i=1}^{A^n(t)} \bone(\tau^n_i + \eta_i >t), \quad t\ge 0.
\end{align}
Here the first term counts the number of individuals that are initially infected at time $0$ and remain infected at time $t$, and the second term counts the number of individuals that get infected between time $0$ and time $t$, and remain infected at time $t$. $R^n(t)$ counts the number of recovered individuals, and can be represented as 
\begin{align}
R^n(t) = \sum_{j=1}^{I^n(0)} \bone(\eta^0_j \le t)+ \sum_{i=1}^{A^n(t)} \bone(\tau^n_i + \eta_i \le t), \quad t\ge 0. \non 
\end{align}

\begin{remark}\label{sec-SIR-queue}
We remark that the dynamics of $I^n(t)$ resembles that of an $M/GI/\infty$ queue with a ``state-dependent" Poisson arrival process $A^n(t)$ and i.i.d. service times $\{\eta_i\}$ under the initial condition $(I^n(0), \{\eta_j^0\})$. However, the ``state-dependent" arrival rate $ \lambda n  \frac{S^n(s)}{n} \frac{I^n(s)}{n}$ not only depends on the infection (``queueing") state $I^n(t)$, but also upon the susceptible state $S^n(t)$. On the other hand, $S^n(t) = n-I^n(0) - A^n(t)$, so the ``state-dependent'' arrival rate is ``self-exciting" in some sense. 
\end{remark}

\begin{assumption} \label{AS-SIR-1}
There exists a deterministic constant $\bar{I}(0)\in (0,1)$ such that 
 $\bar{I}^n(0) \to \bar{I}(0)$ in probability in $\RR_+$ as $n\to\infty$. 
\end{assumption}

Define the fluid-scaled process $\bar{X}^n := n^{-1} X^n$ for any process $X^n$.

\begin{theorem}\label{thm-FLLN-SIR} 
Under Assumption \ref{AS-SIR-1}, the processes 
$$
(\bar{S}^n, \bar{I}^n, \bar{R}^n) \to ( \bar{S},  \bar{I}, \bar{R})\quad \text{in } D^3$$
in probability 
as $n\to\infty$, 
where the limit process $(\bar{S},\bar{I}, \bar{R})$ is the unique solution to the system of deterministic equations
\begin{align}
\bar{S}(t) &= 1- \bar{I}(0) - \lambda \int_0^t \bar{S}(s) \bar{I}(s) ds, \label{SIR-barS}\\
\bar{I}(t)  & =  \bar{I}(0) F_0^c(t) + \lambda \int_0^t F^c(t-s)  \bar{S}(s) \bar{I}(s) ds,  \label{SIR-barI}\\
\bar{R}(t) &= \bar{I}(0) F_0(t) + \lambda \int_0^t F(t-s) \bar{S}(s) \bar{I}(s) ds, \label{SIR-barR}
\end{align}
for $t\ge 0$.  $\bar{S}$ is in $C$.  If $F_0$ is continuous, then $\bar{I}$ and $\bar{R}$ are in $C$; otherwise, they are in $D$. 
\end{theorem}

Define the diffusion-scaled processes 
\begin{align} \label{SIR-diff-def}
\hat{S}^n(t) & := \sqrt{n} \left( \bar{S}^n(t) - \bar{S}(t) \right) = \sqrt{n} \left( \bar{S}^n(t) - \left(1 - \bar{I}(0) -  \lambda \int_0^t \bar{S}(s) \bar{I}(s) ds \right) \right), \non\\
\hat{I}^n(t) &:=  \sqrt{n} \left(  \bar{I}^n(t)  -\bar{I}(t) \right) = \sqrt{n} \left(  \bar{I}^n(t) -  \bar{I} (0) F_0^c(t) - \lambda \int_0^t F^c(t-s)  \bar{S}(s) \bar{I}(s) ds \right),\non\\
\hat{R}^n(t) &:=  \sqrt{n} \left(  \bar{R}^n(t) -\bar{R}(t) \right) = \sqrt{n} \left(  \bar{R}^n(t) -  \bar{I} (0) F_0(t) - \lambda \int_0^t F(t-s)  \bar{S}(s) \bar{I}(s) ds \right). 
\end{align}
These represent the fluctuations around the fluid dynamics. 
Observe that 
\begin{equation} 
\hat{S}^n(t) + \hat{I}^n(t) + \hat{R}^n(t) = 0, \quad t\ge 0.  \non
\end{equation}

\begin{assumption} \label{AS-SIR-2}
There exist a deterministic constant $\bar{I}(0)\in (0,1)$ and a random variable $\hat{I}(0)$ such that 
 $\hat{I}^n(0):=\sqrt{n} (\bar{I}^n(0) - \bar{I}(0))  \RA \hat{I}(0) $ in $\RR$ as $n\to\infty$. In addition,  $\sup_n \E\big[\hat{I}^n(0)^2 \big]<\infty$ and thus by Fatou's lemma, $\E\big[\hat{I}(0)^2\big]<\infty$. 
\end{assumption}

 In the next statement, the process $(\hat{S},\hat{I})$ is the unique solution of the system of Volterra integral equations \eqref{SIR-Shat}, \eqref{SIR-Ihat}. Existence and uniqueness for such a system is well--known, see Lemma \ref{lem-Gamma-cont} below. Note that once we have $\hat{S}$ and $\hat{I}$, 
$\hat{R}$ is given by the formula \eqref{SIR-Rhat}.

 \begin{theorem} \label{thm-FCLT-SIR}
Under Assumption~\ref{AS-SIR-2}, the processes
 \begin{equation} \label{eqn-FCLT-conv-SIR}
(\hat{S}^n, \hat{I}^n, \hat{R}^n) \RA (\hat{S}, \hat{I}, \hat{R}) \qinq D^3 \qasq n \to\infty,
\end{equation}
where the limit $(\hat{S}, \hat{I}, \hat{R})$ is the unique solution to the following set of stochastic Volterra integral equations driven by Gaussian processes: 
\begin{align}\label{SIR-Shat}
\hat{S}(t)
&= -\hat{I}(0)  -   \lambda \int_0^t \left( \hat{S}(s) \bar{I}(s)+ \bar{S}(s)\hat{I}(s) \right) ds -  \hat{M}_A(t), 
\end{align}
\begin{equation}\label{SIR-Ihat}
\hat{I}(t) =  \hat{I}(0) F^c_0(t)   + \lambda \int_0^t F^c(t-s) \left( \hat{S}(s) \bar{I}(s) +\bar{S}(s)\hat{I}(s) \right) ds +  \hat{I}_{0}(t) +  \hat{I}_{1}(t),  
\end{equation}
\begin{equation}\label{SIR-Rhat}
\hat{R}(t) = \hat{I}(0) F_0(t)  + \lambda \int_0^t F(t-s) \left( \hat{S}(s) \bar{I}(s) +\bar{S}(s)\hat{I}(s) \right) ds + \hat{R}_{0}(t) +  \hat{R}_{1}(t), 
\end{equation}
with $\bar{S}(t)$ and $\bar{I}(t)$ given in Theorem \ref{thm-FLLN-SIR}.
Here $ (\hat{I}_{0},\hat{R}_0)$, independent of $\hat{I}(0)$, is a mean-zero two-dimensional Gaussian process with the covariance functions: for $t, t'\ge 0$,
 \begin{align*}
\Cov(\hat{I}_{0}(t), \hat{I}_{0}(t')) &= \bar{I}(0) (F_0^c(t\vee t') - F_0^c(t) F_0^c(t')), \\
\Cov(\hat{R}_{0}(t), \hat{R}_{0}(t')) &= \bar{I}(0) (F_0(t\wedge t') - F_0(t) F_0(t')), \\
\Cov(\hat{I}_0(t), \hat{R}_0(t') ) & = \bar{I}(0)\big[ (F_0(t') - F_0(t)) \bone(t'\ge t) - F_0^c(t) F_0(t') \big]. 
\end{align*}
If $F_0$ is continuous, then $ \hat{I}_{0}$ and $\hat{R}_0$ are continuous. 
The limit process
$ (\hat{M}_A, \hat{I}_1,\hat{R}_1)$,  is a continuous three-dimensional Gaussian process, independent of $(\hat{I}_0, \hat{R}_0, \hat{I}(0))$,  and has the representation
\begin{align}
\hat{M}_A(t) = W_F([0,t]\times[0,\infty)), \quad
 \hat{I}_1(t) = W_F([0,t]\times[t,\infty)), \quad
\hat{R}_1(t) = W_F([0,t]\times[0,t]),  \non
\end{align}
where  $W_F$ is
a Gaussian white noise process on $\RR_+^2$ with mean zero and 
$$\E \left[ W_F((a,b]\times (c,d])^2\right] = \lambda \int_a^b (F(d-s)-F(c-s)) \bar{S}(s) \bar{I}(s) ds,
$$
for $0 \le a \le b$ and $0 \le c \le d$. 
 The limit process $\hat{S}$ has continuous sample paths and $\hat{I}$ and $\hat{R}$ have  c{\`a}dl{\`a}g sample paths.  If the c.d.f. $F_0$ is continuous, then $\hat{I}$ and $\hat{R}$ have continuous sample paths. 
 If $\hat{I}(0)$ is a Gaussian random variable, then  $(\hat{S},\hat{I},\hat{R})$ is a Gaussian process. 
 \end{theorem}

 \begin{remark}
 From the representation of the limit processes $ (\hat{M}_A, \hat{I}_1,\hat{R}_1)$ using the white noise $W_F$, we easily obtain 
their covariance functions: for  $t, t'\ge 0$, 
$$
\Cov(\hat{M}_A(t), \hat{M}_A(t'))  = \lambda \int_0^{t\wedge t'}  \bar{S}(s) \bar{I}(s) ds, \quad
\Cov(\hat{I}_1(t), \hat{I}_1(t')) =  \lambda \int_0^{t\wedge t'} F^c(t\vee t'-s) \bar{S}(s) \bar{I}(s) ds, \non
$$
$$
\Cov(\hat{R}_1(t), \hat{R}_1(t')) = \lambda \int_0^{t\wedge t'} F(t\wedge t'-s) \bar{S}(s) \bar{I}(s) ds, \quad
\Cov(\hat{M}_A(t), \hat{I}_1(t'))  = \lambda \int_0^{t\wedge t'} F^c(t'-s) \bar{S}(s) \bar{I}(s) ds,\non
$$
$$
\Cov(\hat{M}_A(t), \hat{R}_1(t')) = \lambda \int_0^{t\wedge t'} F(t'-s) \bar{S}(s) \bar{I}(s) ds, $$$$
\Cov(\hat{I}_1(t), \hat{R}_1(t')) = \lambda \int_0^{t} (F(t'-s) - F(t-s)) \bone(t'>t)  \bar{S}(s) \bar{I}(s) ds.  \non
$$
 \end{remark}

\begin{remark}
The approach in this paper can be slightly modified to allow the rate $\lambda$ to be non-stationary $\lambda(t)$.  In epidemic models, a non-stationary $\lambda(t)$ can represent seasonal effects. 
The process $A^n$ is written as
$$
A^n(t) = A_*\left(  n \int_0^t \lambda(s) \frac{S^n(s)}{n} \frac{I^n(s)}{n} ds \right).
$$
For the SIR model, the fluid equation for $\bar{I}$ becomes 
$$
\bar{I}(t)  =  \bar{I}(0) F_0^c(t) + \int_0^t \lambda(s) F^c(t-s)  \bar{S}(s) \bar{I}(s) ds,
$$
and the FCLT limit $\hat{I}$ becomes
$$
\hat{I}(t) =  \hat{I}(0) F^c_0(t)   +  \int_0^t \lambda(s) F^c(t-s) \left( \hat{S}(s) \bar{I}(s) +\bar{S}(s)\hat{I}(s) \right) ds +  \hat{I}_{0}(t) +  \hat{I}_{1}(t), \quad t \ge 0, 
$$
where $\hat{I}_0(t)$ is the same as in the stationary case, and $\hat{I}_1(t)$ has covariance function 
$$
\Cov(\hat{I}_1(t), \hat{I}_1(t'))  =   \int_0^{t\wedge t'} \lambda(s) F^c(t\vee t'-s) \bar{S}(s) \bar{I}(s) ds, \quad t, t \ge 0. 
$$
The same applies to the other processes, and the study of other models. 
\end{remark}

\subsection{SIS Model with general infectious periods} \label{sec-SIS}

In the SIS model, individuals become susceptible immediately after they go through the infectious periods. With a population of size $n$, we have $S^n(t) + I^n(t) =n$ for all $t\ge 0$.
The cumulative infectious process $A^n$ has the same expression \eqref{eqn-An-SIR} as in the SIR model. 
Suppose that there are initially $I^n(0)$ infectious individuals whose remaining infectious times are $\eta_j^0$, $j=1,\dots, I^n(0)$, and each individual that become infectious after time $0$ has infectious periods $\eta_i$, corresponding to the infectious time $\tau_i^n$ of $A^n$. We use $F_0$ and $F$ for the distributions of $\eta^0_j$ and $\eta_i$, respectively. 
Then the dynamics of $I^n$ has the same representation  \eqref{eqn-SIR-In} as in the SIR model.
The only difference is that $S^n(0) = n - I^n(0)$ and $S^n(t) = n - I^n(t)$ so that the dynamics of $(S^n,I^n)$ is determined by the one-dimensional process $I^n$.  
Thus we will focus on the process $I^n$ alone. 
We will impose the same condition as in Assumption \ref{AS-SIR-1}. Define the fluid-scaled process $\bar{I}^n = n^{-1}I^n$. 
\begin{theorem}
Under Assumption \ref{AS-SIR-1}, 
$\bar{I}^n \to  \bar{I} $ in $D$
in probability as $n\to\infty$, 
where  
\begin{align} \label{SIS-barI}
\bar{I}(t) =   \bar{I}(0) F_0^c(t) + \lambda \int_0^t F^c(t-s) (1-  \bar{I}(s) ) \bar{I}(s) ds, \quad t\ge 0. 
\end{align}
$\bar{I} \in D$; if  $F_0$ is continuous, then $\bar{I} \in C$. 
\end{theorem}

Define the diffusion-scaled process $\hat{I}^n=\sqrt{n}(\bar{I}^n-\bar{I})$. Then we have the following FCLT.

\begin{theorem}
Under Assumptions \ref{AS-SIR-2},  $\hat{I}^n \RA  \hat{I}$ in  $D$ as $n \to \infty$,
where  
\begin{equation}\label{SIS-Ihat}
\hat{I}(t) =  \hat{I}(0) F^c_0(t)   + \lambda \int_0^t F^c(t-s)   (1- 2 \bar{I}(s))   \hat{I}(s) ds +  \hat{I}_{0}(t) +  \hat{I}_{1}(t), \quad t \ge 0, 
\end{equation}
where $ \hat{I}_{0}(t)$ is a mean-zero Gaussian process with the covariance function 
$$
\Cov(\hat{I}_{0}(t), \hat{I}_{0}(t')) = \bar{I}(0) (F_0^c(t\vee t') - F_0^c(t) F_0^c(t')), \quad t, t' \ge 0, 
$$
and $\hat{I}_1(t)$ is a continuous mean-zero Gaussian process with covariance function
$$
\Cov(\hat{I}_{1}(t), \hat{I}_{1}(t')) =  \lambda \int_0^{t\wedge t'} F^c(t\vee t'-s) (1-\bar{I}(s)) \bar{I}(s) ds, \quad t, t' \ge 0. 
$$
$\hat{I}(0)$, $ \hat{I}_{0}(t)$ and $\hat{I}_1(t)$ are mutually independent. $\hat{I}$ has c{\`a}dl{\`a}g sample paths;
if $F_0$ is continuous, then $ \hat{I}_{0}(t)$  is continuous and thus, $\hat{I}$ has continuous sample paths. 
If $\hat{I}(0)$ is a Gaussian random variable, then $\hat{I}$ is a Gaussian process. 
\end{theorem}

\section{Non-Markovian SEIR and SIRS Models} \label{sec-SEIR-SIRS} 

\subsection{ SEIR Model with general exposing and infectious periods} \label{sec-SEIR}

The SEIR model is described as follows. 
There are four groups in the population: Susceptible, Exposed, Infectious and Recovered (Immune). 
Susceptible individuals get infected through interactions with infectious ones. 
After getting infected, they become exposed and remain so during a latent period of time, and then transit to the infectious period. 
Afterwards, these individuals become  recovered and immune, and will not be susceptible or infected in the future. 

Let $n$ be the population size. 
Let $S^n(t)$, $E^n(t)$, $I^n(t)$ and $R^n(t)$ represent the susceptible, exposed, infectious and recovered individuals, respectively, at time $t$. Assume that $I^n(0)>0$, $E^n(0)>0$,  $R^n(0)=0$, and $S^n(0)=n - I^n(0) - E^n(0)$. 
An individual $i$ going through the S-E-I-R process  has the following time epochs: $\tau_i^n$, $\tau_i^n+\xi_i$, $\tau_i^n+\xi_i+\eta_i$, representing the times of becoming exposed, infectious and recovered (immune), respectively; namely, $\xi_i$ is the exposure period and $\eta_i$ is the infectious period. (It is reasonable to assume that $\xi_i$ and $\eta_i$ are independent of the population size $n$.) 
For the individuals $I^n(0)$ that are infectious at time $0$, let $\eta_j^0$ be the remaining infectious period. For the individuals $E^n(0)$ that are exposed at time $0$, let $\xi_j^0$ be the remaining exposure time.

Assume that $(\xi_i, \eta_i)$'s are i.i.d. bivariate random vectors with a joint distribution $H(du,dv)$, which has marginal c.d.f.'s $G$ and $F$ for $\xi_i$ and $\eta_i$, respectively, and a conditional c.d.f. of $\eta_i$, $F(\cdot |u)$ given that $\xi_i=u$. Assume that $(\xi^0_j, \eta_j)$'s are i.i.d. bivariate random vectors with a joint distribution $H_0(du,dv)$, which has marginal c.d.f.'s $G_0$ and $F$ for $\xi_j^0$ and $\eta_j$, respectively, and a conditional c.d.f. of $\eta_j$, $F_0(\cdot|u)$ given that  $\xi^0_j=u$.  (Note that the pair $(\xi^0_j, \eta_j)$ is the remaining exposing time and the subsequent infectious period for the $i^{\rm th}$ individual initially being exposed.)
In addition, we assume that $(\xi_i, \eta_i)$ and $(\xi^0_i, \eta_j)$ are independent for each $i$, and they are also independent of $\{\eta^0_j\}$ (that is, the remaining infectious times of the initially infected individuals are independent of all the other exposing and infectious times). 
We use the notation $G^c=1-G$, and similarly for $G_0^c$, $F^c$ and $F_0^c$. 
Define 
\begin{align}
 \Phi_0(t)&:= \int_0^t \int_0^{t-u} H_0(du, dv) =\int_0^t \int_0^{t-u} F_0(dv|u) d G_0(u),  \label{Phi-0-def}\\
  \Psi_0(t) &:= \int_0^t \int_{t-u}^\infty H_0(du, dv)  = \int_0^t \int_{t-u}^\infty F_0(dv| u) d G_0(u) = G_0(t)- \Phi_0(t), \label{Psi-0-def} 
\end{align}
and 
\begin{align}
 \Phi(t)&:= \int_0^t \int_0^{t-u} H(du, dv) =\int_0^t \int_0^{t-u} F(dv|u) d G(u),  \label{Phi-def}\\
  \Psi(t) &:= \int_0^t \int_{t-u}^\infty H(du, dv)  = \int_0^t \int_{t-u}^\infty F(dv| u) d G(u) = G(t)- \Phi(t). \label{Psi-def}
\end{align}
Note that in the case of independent $\xi_i$ and $\eta_i$, letting $F(dv) = F(dv|u)$, we have 
\begin{align} \label{Psi-def-ind}
\Phi(t) = \int_0^t F(t-u) d G(u), \quad \Psi(t) = \int_0^t F^c(t-u) d G(u) = G(t) - \Phi(t). 
\end{align} 
Similarly, with independent $\xi_j^0$ and $\eta_j$, letting  $F_0(dv) = F_0(dv|u)=F(dv)$, we have 
\begin{align}  \label{Psi-0-def-ind}
\Phi_0(t) = \int_0^t F(t-u) d G_0(u), \quad \Psi_0(t) = \int_0^t F^c(t-u) d G_0(u) = G_0(t) - \Phi_0(t). 
\end{align}

Let $A^n(t)$ be the cumulative process of individuals that become exposed between time 0 and time $t$. Let $\lambda$ be the rate of susceptible patients that become exposed. 
Then we can express it as 
\begin{equation} \label{eqn-An}
A^n(t) = A_*\left( \lambda n \int_0^t  \frac{S^n(s)}{n} \frac{I^n(s)}{n} ds \right)
\end{equation}
where $A_*$ is a unit rate Poisson process. 
(This has the same expression as the cumulative process $A^n$ in \eqref{eqn-An-SIR} of individuals becoming infectious in the SIR model.) 
The process $A^n(t)$ has event times $\tau_i^n$, $i\in\NN$.   
Assume that the quantities $A_*$,  $\{(\xi_j^0,\eta_j^0)\}$, $\{(\xi_i, \eta_i)\}$, and 
the initial quantities $(E^n(0), I^n(0))$ are mutually independent. 

We represent the dynamics of $(S^n, E^n, I^n, R^n)$ as follows: for $t\ge 0$, 
\begin{align}
S^n(t) &= S^n(0) - A^n(t) = n- I^n(0) - E^n(0)  - A^n(t), \label{SEIR-Sn}\\
E^n(t) &= \sum_{j=1}^{E^n(0)} \bone (\xi_j^0 >t) +  \sum_{i=1}^{A^n(t)} \bone(\tau_i^n + \xi_i >t), \label{SEIR-En}\\
I^n(t) &= \sum_{j=1}^{I^n(0)} \bone(\eta_j^0 > t) + \sum_{j=1}^{E^n(0)} \bone (\xi_j^0 \le t) \bone(\xi_j^0+ \eta_j >t ) \non \\
& \qquad + \sum_{i=1}^{A^n(t)} \bone(\tau_i^n+ \xi_i \le t) \bone(\tau_i^n+ \xi_i + \eta_i >t), \label{SEIR-In}\\
R^n(t) & = \sum_{j=1}^{I^n(0)} \bone(\eta_j^0 \le t) +  \sum_{j=1}^{E^n(0)} \bone(\xi_j^0+ \eta_j \le t ) + \sum_{i=1}^{A^n(t)} \bone(\tau_i^n+ \xi_i + \eta_i \le t). 
\label{SEIR-Rn}
\end{align}
Note that we are abusing notation of $\eta_j$ and $\eta_i$ in the second and third terms of $I^n(t)$ and $R^n(t)$. The variables $\eta_j$ (more precisely, $\eta_j^E$) in the second term of $I^n(t)$ correspond to the infectious periods of initially exposed individuals that have become infectious by time $t$, while the variables $\eta_i$
(more precisely, $\eta_i^A$) in the third term correspond to the infectious periods of individuals that has become exposed and infectious after time 0 and before time $t$. We drop the superscripts $E$ and $A$, since it should not cause any confusion.
 
We also let $L^n$ be the cumulative process that counts individuals that have 
 become infectious by time $t$. Then its dynamics can be represented by
$$
L^n(t) = \sum_{j=1}^{E^n(0)} \bone (\xi_j^0 \le t) +  \sum_{i=1}^{A^n(t)} \bone(\tau_i^n + \xi_i \le t), \quad t\ge 0. 
$$
We have the following balance equations: for each $t\ge 0$, 
\begin{align}
n&=S^n(t) + E^n(t) +  I^n(t) + R^n(t) , \non\\
E^n(t) &=  E^n(0) + A^n(t) - L^n(t), \non\\
 I^n(t) &= I^n(0) + L^n(t) - R^n(t). \non
\end{align}

Observe that the dynamics of the exposure process $E^n(t)$ is similar to the infectious process $I^n(t)$ in \eqref{eqn-SIR-In} in the SIR model. 
The dynamics of the infectious process $I^n(t)$ resembles the dynamics of the second service station of a tandem infinite-server queue $G/GI/\infty-GI/\infty$, where the arrival process is $A^n$, and the first station has initial customers $E^n(0)$ with remaining service times $\{\xi^0_j\}$ and the second station has the initial customers $I^n(0)$ with remaining service times $\{\eta^0_j\}$. 
The processes $L^n$ and $R^n$ correspond to the departure processes from the first and second stations (service completions), respectively. 
Similar to the SIR model, the arrival process is Poisson with a ``state-dependent" arrival rate $\lambda n \frac{S^n(s)}{n} \frac{I^n(s)}{n}$, which depends not only on the state of $I^n(s)$ (state of the second ``station" in the tandem queueing model), but also on the state of susceptible individuals, $S^n(s)= n- I^n(0) - E^n(0)  - A^n(t)$. However it is independent of the state of the exposure individuals $E^n(t)$.

\begin{assumption} \label{AS-SEIR-1}
There exist deterministic constants $\bar{I}(0)\in (0,1)$ and $\bar{E}(0) \in (0,1)$ such that  $\bar{I}(0) + \bar{E}(0) <1$ and
 $(\bar{I}^n(0),\bar{E}^n(0))  \to ( \bar{I}(0), \bar{E}(0)) \in \RR^2 $ in probability  as $n\to\infty$. 
\end{assumption}

Define the fluid-scaled processes as in the SIR model. 
We have the following FLLN for the fluid-scaled processes $(\bar{S}^n, \bar{E}^n,  \bar{I}^n, \bar{R}^n)$. 

\begin{theorem}\label{thm-FLLN-SEIR}
Under Assumption \ref{AS-SEIR-1},  
\begin{equation} \label{SIR-FLLN-conv}
\left(\bar{S}^n, \bar{E}^n,  \bar{I}^n, \bar{R}^n \right) \to \left(\bar{S}, \bar{E},  \bar{I}, \bar{R} \right) \qinq D^4
\end{equation}
in probability as $n\to\infty$, where the limit process $(\bar{S}, \bar{E}, \bar{I}, \bar{R}) $ is the unique solution to the system of deterministic equations: for each $t\ge 0$, 
\begin{align}
\bar{S}(t) &= 1- \bar{I}(0) - \bar{E}(0) - \bar{A}(t)  = 1- \bar{I}(0) - \bar{E}(0) -  \lambda \int_0^t \bar{S}(s) \bar{I}(s) ds,   \label{SEIR-barS} \\
\bar{E}(t) &= \bar{E}(0) G_0^c(t) +  \lambda \int_0^t G^c(t-s) \bar{S}(s) \bar{I}(s) ds,   \label{SEIR-barE}\\
\bar{I}(t) 
& =  \bar{I}(0) F^c_0(t) +\bar{E}(0) \Psi_0(t)  + \lambda   \int_0^t \Psi(t-s) \bar{S}(s) \bar{I}(s) ds,   \label{SEIR-barI} \\
\bar{R}(t) 
& =  \bar{I}(0) F_0(t) +\bar{E}(0) \Phi_0(t)  + \lambda   \int_0^t \Phi(t-s)  \bar{S}(s) \bar{I}(s) ds .  \label{SEIR-barR} 
\end{align} 
The limit $\bar{S}$ is in $C$ and $\bar{E}$, $\bar{I}$ and $\bar{R}$ are in $D$.  If $G_0$ and $F_0$ are continuous, then they are in $C$. 
\end{theorem}

We remark that given the input data $\bar{I}(0)$ and $\bar{E}(0)$ and the distribution functions, the solution to the set of equations above can be determined by the two equations \eqref{SEIR-barS} and \eqref{SEIR-barI} for $\bar{S}$ and $\bar{I}$, which is a $2$--dimensional system of linear Volterra integral equations. 
It is easy to check that we have the balance equation for the FLLN limits: 
\begin{align}
1=\bar{S}(t) + \bar{E}(t) +  \bar{I}(t) + \bar{R}(t), \non
\end{align}
As a consequence, we have the joint convergence with 
$(\bar{A}^n, \bar{L}^n) \to (\bar{A}, \bar{L})$ in $ D^2 $
in probability as $n\to \infty$, where
\begin{align*}
\bar{A}(t)  &=  \bar{E}(t) + \bar{L}(t)-  \bar{E}(0), \qquad 
  \bar{L}(t)  = \bar{I}(t)+ \bar{R}(t) - \bar{I}(0). 
\end{align*}
In particular, we have
\begin{align*}
\bar{A}(t) =  \lambda \int_0^t \bar{S}(s) \bar{I}(s) ds, 
\qquad
\bar{L}(t) = \bar{E}(0) G_0(t) +  \lambda \int_0^t G(t-s) \bar{S}(s) \bar{I}(s) ds. 
\end{align*}

Define the diffusion-scaled processes:
\begin{align}\label{SEIR-diff-def}
\hat{S}^n(t) & := \sqrt{n} \left( \bar{S}^n(t) - \bar{S}(t) \right) = \sqrt{n} \left( \bar{S}^n(t) -  1 + \bar{I}(0) +  \lambda \int_0^t \bar{S}(s) \bar{I}(s) ds  \right), \\
\hat{E}^n(t) &:=\sqrt{n} \left( \bar{E}^n(t) -\bar{E}(t) \right)  = \sqrt{n} \left( \bar{E}^n(t)  -  \bar{E}(0) G_0^c(t) -  \lambda \int_0^t G^c(t-s) \bar{S}(s) \bar{I}(s) ds \right), \non\\
\hat{I}^n(t) &:= \sqrt{n} \left( \bar{I}^n(t) - \bar{I}(t) \right) =  \sqrt{n} \Bigg( \bar{I}^n(t)  - \bar{I}(0) F^c_0(t) -   \bar{E}(0) \Psi_0(t)
- \lambda \int_0^t  \Psi(t-s)
\bar{S}(s) \bar{I}(s) ds \Bigg) , \non \\
\hat{R}^n(t)  &:=  \sqrt{n} \left( \bar{R}^n(t)  - \bar{R}(t) \right)  = \sqrt{n} \Bigg( \bar{R}^n(t) - \bar{I}(0) F_0(t) 
-   \bar{E}(0) \Phi_0(t)
- \lambda \int_0^t   \Phi(t-s)
\bar{S}(s) \bar{I}(s) ds \Bigg).  \non
\end{align}
It is clear that 
\begin{equation}
\hat{S}^n(t) + \hat{E}^n(t) + \hat{I}^n(t) + \hat{R}^n(t) =0, \quad t\ge 0.  \non
\end{equation}

We will establish a FCLT for the diffusion-scaled processes 
$(\hat{A}^n, \hat{S}^n, \hat{E}^n,\hat{L}^n, \hat{I}^n, \hat{R}^n)$. 
 For that purpose, we make the following assumption on the initial condition and on the law of the exposure / infectious periods. 

 \begin{assumption} \label{AS-SEIR-2}
 There exist deterministic constants $\bar{I}(0)\in (0,1)$ and $\bar{E}(0) \in (0,1)$  and random variables $\hat{I}(0)$ and $\hat{E}(0)$ such that  $\bar{I}(0) + \bar{E}(0) <1$ and
 $$
 \left(\sqrt{n}(\bar{I}^n(0) - \bar{I}(0)), \sqrt{n}(\bar{E}^n(0) - \bar{E}(0)) \right) \RA (\hat{I}(0), \hat{E}(0)) \qinq \RR^2 \qasq n\to\infty.
 $$
 In addition, $\sup_n \E\big[\hat{E}^n(0)^2 \big]<\infty$ and $\sup_n \E\big[\hat{I}^n(0)^2 \big]<\infty$,
  and thus by Fatou's lemma, $\E\big[\hat{E}(0)^2\big]<\infty$ and $\sup_n \E\big[\hat{I}^n(0)^2 \big]<\infty$. 
 \end{assumption}
 
In the next statement, $(\hat{S},\hat{I})$ is the unique solution of the system of linear integral equations \eqref{SEIR-Shat}, \eqref{SEIR-Ihat}, whose existence and uniqueness follows from an obvious extension of the first part of Lemma \ref{lem-Gamma-cont} below.  Once $(\hat{S},\hat{I})$ is specified, 
$\hat{E}$ and  $\hat{R}$ are given by the formulas \eqref{SEIR-Ehat}  and \eqref{SEIR-Rhat}.

 \begin{theorem} \label{thm-FCLT-SEIR}
Under Assumption~\ref{AS-SEIR-2}, 
 \begin{equation} \label{eqn-FCLT-conv-SEIR}
(\hat{S}^n, \hat{E}^n,\hat{I}^n, \hat{R}^n) \RA (\hat{S}, \hat{E}, \hat{I}, \hat{R}) \qinq D^4 \qasq n \to\infty,
\end{equation}
where the limit processes $(\hat{S}, \hat{E}, \hat{I}, \hat{R}) $ are the unique solution to the following set of stochastic Volterra integral equations driven by Gaussian processes: 
\begin{align}\label{SEIR-Shat}
\hat{S}(t) 
&= -\hat{I}(0)  -   \lambda \int_0^t \left( \hat{S}(s) \bar{I}(s)+ \bar{S}(s)\hat{I}(s) \right) ds - \hat{M}_A(t), 
\end{align}
\begin{equation} \label{SEIR-Ehat}
\hat{E}(t) =  \hat{E}(0) G^c_0(t)  + \lambda \int_0^t G^c(t-s) \left( \hat{S}(s) \bar{I}(s) +\bar{S}(s)\hat{I}(s) \right) ds  + \hat{E}_0(t) + \hat{E}_1(t),
\end{equation}
\begin{align} \label{SEIR-Ihat}
\hat{I}(t) & =\hat{I}(0) F^c_0(t) +  \hat{E}(0) \Psi_0(t) +   \hat{I}_{0,1}(t) + \hat{I}_{0,2}(t) +  \hat{I}_{1}(t) \non\\
& \qquad +  \lambda \int_0^t  \Psi(t-s)   \left( \hat{S}(s) \bar{I}(s) + \bar{S}(s) \hat{I}(s) \right) ds , 
\end{align}
\begin{align} \label{SEIR-Rhat}
\hat{R}(t) & =\hat{I}(0) F_0(t) +  \hat{E}(0) \Phi(t)
 +  \hat{R}_{0,1}(t) + \hat{R}_{0,2}(t) +  \hat{R}_{1}(t)  \non\\
& \qquad +  \lambda \int_0^t \Phi(t-s) \left( \hat{S}(s) \bar{I}(s) + \bar{S}(s) \hat{I}(s) \right) ds, 
\end{align}
with  $\bar{S}(t)$ and $\bar{I}(t)$ given in Theorem \ref{thm-FLLN-SEIR}.
Here $( \hat{E}_0,  \hat{I}_{0,1}, \hat{I}_{0,2}, \hat{R}_{0,1}, \hat{R}_{0,2})$, independent of $\hat{E}(0)$ and $\hat{I}(0)$,  is a mean-zero Gaussian process with covariance functions: for $t, t'\ge 0$, 
\begin{align*}
\Cov(\hat{E}_0(t), \hat{E}_0(t')) &= \bar{E}(0) (G_0^c(t\vee t') - G^c_0(t) G^c_0(t')), \\
\Cov(\hat{I}_{0,1}(t), \hat{I}_{0,1}(t')) &= \bar{I}(0) (F_0^c(t\vee t') - F_0^c(t) F_0^c(t')), \\
\Cov(\hat{I}_{0,2}(t), \hat{I}_{0,2}(t')) & = \bar{E}(0) \left(\Psi_0(t\wedge t') - \Psi_0(t) \Psi_0(t')\right),\\
\Cov(\hat{R}_{0,1}(t), \hat{R}_{0,1}(t')) & = \bar{I}(0) (F_0(t\wedge t') - F_0(t) F_0(t')), \\
\Cov(\hat{R}_{0,2}(t), \hat{R}_{0,2}(t')) & = \bar{E}(0) \left( \Phi_0(t\wedge t') -   \Phi_0(t) \Phi_0(t')  \right), \\
\Cov( \hat{E}_{0}(t),  \hat{I}_{0,2}(t')) &= \bar{E}(0)\bone(t'\ge t) \bigg( \int_t^{t'}  F^c_0(t'-s|s)dG_0(s) - G_0^c(t) \Psi_0(t') \bigg),  \non\\
\Cov( \hat{E}_{0}(t),  \hat{R}_{0,2}(t')) &= \bar{E}(0) \bone(t'\ge t) \bigg(\int_t^{t'}F_{0}(t'-s|s) d G_0(s) - G_0^c(t)  \Phi_0(t')\bigg),    \non\\
\Cov( \hat{I}_{0,2}(t),  \hat{R}_{0,2}(t')) &= \bar{E}(0) \bone(t'\ge t) \bigg( \int_0^{t} (F_0(t'-s|s) - F_0(t-s|s))d G_0(s) - \Psi_0(t) \Phi_0(t') \bigg).
\end{align*} 
$ \hat{I}_{0,1}(t)$ and $\hat{I}_{0,2}(t)$ are independent, so are the pairs $ \hat{R}_{0,1}(t)$ and $\hat{R}_{0,2}(t)$, 
$\hat{E}_{0}(t)$ and $ \hat{I}_{0,1}(t)$, $\hat{E}_{0}(t)$ and $ \hat{R}_{0,1}(t)$,  $\hat{I}_{0,1}(t)$ and $ \hat{R}_{0,j}(t)$ for $j=1,2$. 
 
The limit 
$ (\hat{M}_A, \hat{E}_1, \hat{I}_1,\hat{R}_1)$ is a four-dimensional continuous Gaussian process, independent of $\hat{E}_0$, $\hat{I}_{0,1}$, $\hat{I}_{0,2}$, $\hat{R}_{0,1}$, $\hat{R}_{0,2}$ and $\hat{I}(0)$,   and can be written as
\begin{align}
&\hat{M}_A(t) = W_H([0,t]\times[0,\infty)\times [0,\infty)), \quad
\hat{E}_1(t) = W_H([0,t]\times[t,\infty)\times [0,\infty)), \non\\
& \hat{I}_1(t) = W_H([0,t]\times[0,t)\times[t,\infty)), \quad
\hat{R}_1(t) = W_H([0,t]\times[0,t)\times [0,t)),  \non
\end{align}
where  $W_H$ is
a continuous Gaussian white noise process on $\RR_+^3$ with mean zero and 
\begin{align}
& \E\left[ W_H([s,t)\times [a,b)\times [c,d))^2\right] \non\\
& =  \lambda \int_s^t \left( \int_{a-s}^{b-s} (F(d-y-s|y) - F(c-y-s|y)) G(dy) \right)  \bar{S}(s) \bar{I}(s) ds,
\end{align}
for $0 \le s \le t$,  $0 \le a \le b$ and $0 \le c \le d$.

The limit process $\hat{S}$ has continuous sample paths and $\hat{E}_1$, $\hat{I}_1$ and $\hat{R}_1$ have   c{\`a}dl{\`a}g sample paths.
If the c.d.f.'s $G_0$ and $F_0$ are continuous, then $\hat{E}_0$, $\hat{I}_{0,1}$, $\hat{I}_{0,2}$, $\hat{R}_{0,1}$ and $\hat{R}_{0,2}$ are continuous, and thus, $\hat{E}_1$, $\hat{I}_1$ and $\hat{R}_1$ have continuous sample paths. 
  If $(\hat{I}(0),\hat{E}(0))$ is a Gaussian random vector, then  $(\hat{S}, \hat{E}, \hat{I},\hat{R})$ is a Gaussian process. 
 \end{theorem}
  
  \begin{remark}
 The  processes $(\hat{S}(t),\hat{E}(t), \hat{I}(t), \hat{R}(t))$ in \eqref{SEIR-Shat}, \eqref{SEIR-Ehat}, \eqref{SEIR-Ihat} and \eqref{SEIR-Rhat} can be regarded as the solution of a four-dimensional Gaussian-driven linear Volterra stochastic integral equation. The existence and uniqueness of solution can be easily verified.
 From the representations of 
the limit processes $ (\hat{M}_A,\hat{E}_1, \hat{I}_1,\hat{R}_1)$ 
using the white noise $W_H$, we easily obtain 
the covariance functions: for  $t, t'\ge 0$, 
$$
\Cov(\hat{M}_A(t), \hat{M}_A(t')) = \lambda \int_0^{t\wedge t'}  \bar{S}(s) \bar{I}(s) ds, \quad
\Cov(\hat{E}_1(t), \hat{E}_1(t'))  =  \lambda \int_0^{t\wedge t'} G^c(t\vee t'-s) \bar{S}(s) \bar{I}(s) ds, 
$$
$$
\Cov(\hat{I}_1(t), \hat{I}_1(t'))  =  \lambda   \int_0^{t\wedge t'}\int_0^{t\wedge t'-s} F^c(t\vee t'-s-u|u)dG(u)  \bar{S}(s) \bar{I}(s) ds, 
$$
$$
\Cov(\hat{R}_1(t), \hat{R}_1(t')) = \lambda \int_0^{t\wedge t'} \Phi(t\wedge t'-s) \bar{S}(s) \bar{I}(s) ds, 
$$
\begin{align}
\Cov(\hat{E}_1(t), \hat{I}_1(t')) & = \lambda \int_0^{t\wedge t'}  (G^c(t-s) - \Psi(t'-s))\bone(t'\ge t)  \bar{S}(s) \bar{I}(s) ds,\non\\
\Cov(\hat{E}_1(t), \hat{R}_1(t')) & = \lambda \int_0^{t\wedge t'}(G^c(t-s) - \Phi(t'-s))\bone(t'\ge t) \bar{S}(s) \bar{I}(s) ds,\non\\
\Cov(\hat{I}_1(t), \hat{R}_1(t')) & = \lambda \int_0^{t\wedge t'} \int_0^{t'-s}  (F(t'-s-y|y) - F(t-s-y|y)) \bone(t'\ge t) d G(y) \bar{S}(s) \bar{I}(s) ds, \non 
\end{align}
$$
\Cov(\hat{M}_A(t), \hat{E}_1(t'))  = \lambda \int_0^{t\wedge t'} G^c(t'-s) \bar{S}(s) \bar{I}(s) ds,\,\,
\Cov(\hat{M}_A(t), \hat{I}_1(t'))  = \lambda \int_0^{t\wedge t'} \Psi(t'-s) \bar{S}(s) \bar{I}(s) ds,
$$
$$
\Cov(\hat{M}_A(t), \hat{R}_1(t')) = \lambda \int_0^{t\wedge t'} \Phi(t'-s) \bar{S}(s) \bar{I}(s) ds.\non
$$
 \end{remark}

  \begin{remark}
  We remark that the exposing and infectious periods are allowed to be dependent, and the effect of such dependence is exhibited in the covariances of the functions of the limit processes $ (\hat{M}_A, \hat{E}_1, \hat{I}_1,\hat{R}_1)$ and in the drift of $\hat{I}$ and $\hat{R}$. Of course, the dependence also affects the deterministic equations for $(\bar{S}, \bar{I})$.  
  
It is also worth noting that the FLLN and FCLT limits for the SIR model can be derived from those for the SEIR model by setting $G=\delta_0$.  Similarly the limits for the SIS model can be also derived from those for the SIRS model, see the next subsection. 
  \end{remark}

\subsection{SIRS model with general infectious and immune periods} 
In the SIRS model, there are three groups in the population: Susceptible, Infectious, Recovered (Immune). Susceptible individuals get infected through interactions with infectious ones, and they become infectious immediately (no exposure period like in the SEIR model). The infectious individuals become recovered and immune, and after the immune periods, they become susceptible. 
This has a lot of resemblance with the SEIR model, where the exposure and infectious periods in the SEIR model correspond to the infectious and immune periods in the SIRS model, respectively. 
We let $S^n(t), I^n(t), R^n(t)$ represent the susceptible, infectious and immune individuals, respectively at each time $t$ in the SIRS model. Note that $I^n(t)$ (resp. $R^n(t)$) in the SIRS model corresponds to  $E^n(t)$ (resp. $I^n(t)$) in the SEIR model, and $S^n(t)$ in the SIRS model satisfies the balance equation: 
$$
n = S^n(t) + I^n(t) + R^n(t), \quad t\ge 0. 
$$
Since $S^n(t) = n - I^n(t) - R^n(t)$, it suffices to only study the dynamics of the two processes $(I^n, R^n)$. 
We use the variables $\xi_i, \eta_i$ represent the infectious and immune periods, respectively, in the SIRS model, and similarly for the initial quantities $\xi^0_j, \eta_j$. 
We also use the same distribution functions associated with these variables as in the SEIR model. We impose the same conditions in Assumptions \ref{AS-SEIR-1}--\ref{AS-SEIR-2}, where the quantities $E^n(0)$ and $I^n(0)$ are replaced by $I^n(0)$ and $R^n(0)$, respectively. To distinguish the differences, we refer to these as Assumptions \ref{AS-SEIR-1}'--\ref{AS-SEIR-2}'.

We first obtain the following FLLN for the fluid-scaled processes $(\bar{I}^n, \bar{R}^n)$. 
\begin{theorem}\label{thm-FLLN-SIRS}
Under Assumption \ref{AS-SEIR-1}',  $\left(\bar{I}^n,  \bar{R}^n \right) \to \left(\bar{I},  \bar{R} \right)$ in  $D^2$
in probability as $n\to\infty$, where the limits $(\bar{I}, \bar{R}) $ are the unique solution to the system of deterministic equations:  
\begin{align}
\bar{I}(t) &= \bar{I}(0) G_0^c(t) +  \lambda \int_0^t G^c(t-s) (1-\bar{I}(s) -\bar{R}(s))  \bar{I}(s)  ds,   \label{SIRS-barE}\\
\bar{R}(t) 
& =  \bar{R}(0) F^c_0(t) +\bar{I}(0) \Psi_0(t)  + \lambda   \int_0^t \Psi(t-s)  (1-\bar{I}(s) -\bar{R}(s)) \bar{I}(s) ds, \label{SIRS-barI} 
\end{align} 
for each $t\ge 0$.
  If $G_0$ and $F_0$ are continuous, then $\bar{I}$ and $\bar{R}$  are in $C$.
\end{theorem}

We define the diffusion-scaled processes $\hat{I}^n$ and $\hat{R}^n$ as in the SEIR model, but replacing $\bar{S}= 1-\bar{I} -\bar{R}$. 
We impose similar conditions on the initial quantities as in Assumption  \ref{AS-SEIR-2}, which refer to as Assumption~\ref{AS-SEIR-2}'.

 \begin{theorem} \label{thm-FCLT-SIRS}
Under Assumption~\ref{AS-SEIR-2}', 
 $(\hat{I}^n,\hat{R}^n) \RA (\hat{I}, \hat{R})$ in $D^2$ as $n \to\infty$,
where 
\begin{equation} \label{SIRS-Ehat}
\hat{I}(t) =  \hat{I}(0) G^c_0(t)  + \lambda \int_0^t G^c(t-s) \left( - \hat{I}(s) \bar{R}(s) + (1-\bar{I}(s) - 2 \bar{R}(s)) \hat{R}(s) \right) ds  + \hat{I}_0(t) + \hat{I}_1(t),
\end{equation}
\begin{align} \label{SIRS-Ihat}
\hat{R}(t) & =\hat{R}(0) F^c_0(t) +  \hat{I}(0) \Psi_0(t) +  \lambda \int_0^t  \Psi(t-s)    \left( - \hat{I}(s) \bar{R}(s) + (1-\bar{I}(s) - 2 \bar{R}(s)) \hat{R}(s) \right)  ds \non\\
& \qquad  \qquad  +   \hat{R}_{0,1}(t) + \hat{R}_{0,2}(t) +  \hat{R}_{1}(t),  
\end{align}
where  $ \hat{I}_0(t)$, $\hat{I}_1(t)$, $ \hat{R}_{0,1}(t)$ and $\hat{R}_{0,2}(t)$ are 
as given as $ \hat{E}_0(t)$, $\hat{E}_1(t)$, $ \hat{I}_{0,1}(t)$ and $\hat{I}_{0,2}(t)$, respectively,  in Theorem \ref{thm-FCLT-SEIR}. 
If the c.d.f.'s $G_0$ and $F_0$ are continuous, then  $ \hat{I}_0(t)$,  $ \hat{R}_{0,1}(t)$ and $\hat{R}_{0,2}(t)$ are continuous, and thus, the limit processes $\hat{I}_1$ and $\hat{R}_1$  have continuous sample paths. 
  If $(\hat{I}(0),\hat{R}(0))$ is a Gaussian random vector, then  $(\hat{I}, \hat{R})$ is a Gaussian process. 
   \end{theorem}

\section{Special cases} \label{sec-special}

\subsection{Markovian models} \label{sec-Markovian}

We recall the Markovian SEIR model, 
with independent $\xi_i$ and $\eta_i$ for each $i$, and independent $\xi_j^0$ and $\eta_j^0$ for each $j$,   assuming that $G_0(t) = G(t) = 1-e^{-\gamma t}$ and $F_0(t)=F(t) = 1- e^{-\mu t}$.  It is well know that the FLLN limit $(\bar{S}, \bar{E}, \bar{I}, \bar{R})$ satisfies the following ODEs: 
 \begin{align}\label{SEIR-fluid-M}
\bar{S}'(t) &= -\lambda \bar{S}(t) \bar{I}(t), \quad
\bar{E}'(t) = \lambda \bar{S}(t) \bar{I}(t) - \gamma \bar{E}(t), \quad
\bar{I}'(t) =  \gamma \bar{E}(t) - \mu \bar{I}(t), \quad
\bar{R}'(t) = \mu \bar{I}(t). 
\end{align}
These ODEs are  referred to as the Kermack-McKendrick equations \cite{anderson1992infectious,britton2018stochastic}.

It is easy to see that the FLLN in Theorem \ref{thm-FLLN-SEIR} reduces to the above ODEs in this case. In particular, we obtain 
$$
\bar{E}(t) = \bar{E}(0) e^{-\gamma t} +\lambda \int_0^t e^{-\gamma(t-s)} \bar{S}(s) \bar{I}(s) ds, 
$$
and
\begin{align}
\bar{I}(t) & =  \bar{I}(0) e^{-\mu t} +\bar{E}(0) \int_0^t e^{-\mu (t-s)} \gamma e^{-\gamma s} d s    + \lambda \int_0^t \int_0^{t-s} e^{-\mu(t-s-u)} \gamma e^{-\gamma u} du \bar{S}(s) \bar{I}(s) ds,  \non
\end{align}
which lead to 
\begin{align}
\bar{E}'(t) &= -\gamma e^{-\gamma t} \bar{E}(0) + \lambda \bar{S}(t) \bar{I}(t) + \lambda \int_0^t (-\gamma) e^{-\gamma (t-s)} \bar{S}(s) \bar{I}(s)ds  = \lambda  \bar{S}(t) \bar{I}(t) - \gamma \bar{E}(t), \non
\end{align}
and
\begin{align}
\bar{I}'(t) &=  -\mu  e^{-\mu t} \bar{I}(0) +\bar{E}(0) \gamma e^{-\gamma t} + \bar{E}(0)
\int_0^t (-\mu)e^{-\mu (t-s)} \gamma e^{-\gamma s} d s   \non\\
& \qquad + \lambda \int_0^t \gamma e^{-\gamma (t-s)}  \bar{S}(s) \bar{I}(s) ds +
\lambda \int_0^t \int_0^{t-s} (-\mu) e^{-\mu(t-s-u)} \gamma e^{-\gamma u} d u   \bar{S}(s) \bar{I}(s) ds \non \\
&= \gamma \bar{E}(t) -\mu \bar{I}(t). \non 
\end{align}
Similarly we also get $\bar{R}'(t) = \mu \bar{I}(t)$. 
Together with $\bar{S}'(t) = - \lambda \bar{S}(t) \bar{I}(t)$, we obtain the ODEs in \eqref{SEIR-fluid-M}.

It is well known (see \cite{britton2018stochastic}) that  Theorem~\ref{thm-FCLT-SEIR} holds with the limits $\hat{E}$ and $\hat{I}$ given by
\begin{align} \label{SEIR-hatE-M}
\hat{E}(t) &=\hat{E}(0)  + \lambda  \int_0^t ( \hat{S}(s) \bar{I}(s) + \bar{S}(s) \hat{I}(s) ) ds -  \gamma \int_0^t \hat{E}(s)ds \non\\
& \qquad \qquad  +  B_A\left( \lambda  \int_0^t  \bar{S}(s) \bar{I}(s)ds\right) -B_K\left( \gamma\int_0^t  \bar{E}(s) ds\right), 
\end{align}
and
\begin{align} \label{SEIR-hatI-M}
\hat{I}(t) &=\hat{I}(0)  + \gamma \int_0^t \hat{E}(s)ds - \mu\int_0^t \hat{I}(s) ds  +  B_K\left( \gamma \int_0^t \bar{E}(s) ds \right) - B_L\left( \mu \int_0^t \bar{I}(s) ds \right), 
\end{align}
where $B_A$, $B_K$ and $B_L$ are independent Brownian motions. 
It can be shown that the Volterra stochastic integral equations are equivalent to these linear SDEs in distribution. For brevity, we present the detailed proof for the simpler SIS model in the following proposition.

For the Markovian SIS model with exponential infectious periods of rate $\mu$, we get the limit
\begin{align} \label{SIS-Ihat-M}
\hat{I}(t) &=\hat{I}(0)  +\int_0^t  \Big( \lambda   (1-2 \bar{I}(s)) -  \mu\Big)  \hat{I}(s) ) ds  \non\\
& \qquad \qquad  +  B_A\left( \lambda  \int_0^t  (1- \bar{I}(s))  \bar{I}(s)ds\right) -B_I\left( \mu\int_0^t  \bar{I}(s) ds\right), 
\end{align}
where $B_A$ and $B_I$ are independent Brownian motions.The following Proposition states an equivalence property, whose proof is given in the  Section~\ref{sec-proof-equivalence}.

\begin{prop} \label{prop-SIS-equiv-M}
The expressions of $\hat{I}(t)$ in \eqref{SIS-Ihat} and  \eqref{SIS-Ihat-M}    for the SIS model are equivalent in distribution. 
\end{prop}

\subsection{Deterministic infectious periods} \label{sec-deterministic}

When the infectious periods are deterministic, that is, $\eta_i$ is equal to a positive constant $\eta$ with probability one,  it is natural to assume that the remaining infectious duration for the initially infected individual at time zero has a uniform distribution on the interval $[0,\eta]$, that is, $F_0(t)=t/\eta$ for $t \in [0, \eta]$. In fact, $F_0(t)=t/\eta$ is the equilibrium (stationary excess) distribution (see the definition in \eqref{eqn-Fe}) of the deterministic distribution $F(t) = \bone(t\ge \eta)$ for $t \in [0, \eta]$. We can write down the explicit expressions for the FLLN limits in all the models discussed in the paper. We use the SIRS model to illustrate below. 

In the SIRS model, suppose both the infectious and immune times are deterministic, taking values $\xi$ and $\eta$, respectively. 
The remaining infectious and immune times of the initially infected and immune individuals at time 0, $\xi_j^0$ and $\eta^0_j$, have uniform distributions on the intervals $[0, \xi]$ and $[0,\eta]$, respectively. 
That is, $G(t) = \bone(t\ge \xi)$, $F(t)= \bone(t\ge \eta)$,  for $t\ge 0$,   $G_0 (t) = t/\xi$ for $t\in [0, \xi]$ and $F_0(t)=t/\eta$ for $t \in [0, \eta]$. Thus we have  $\Psi_0(t) = \xi^{-1} \int_0^t \bone(t-u < \eta) d u = \xi^{-1}(t- (t-\eta)^+)$, and 
$\Psi(t) = \bone(\xi \le t < \xi+\eta)$ for $t\ge 0$. 
We can write 
\begin{align}
I^n(t) &= \sum_{j=1}^{I^n(0)}\bone(\xi^0_j >t) +A^n(t) - A^n((t-\xi)^+), \non\\
R^n(t) &= \sum_{j=1}^{R^n(0)}\bone(\eta^0_j >t) +  \sum_{j=1}^{I^n(0)}\bone((t- \eta)^+ <\xi^0_j \le t)   + A^n((t-\xi)^+) - A^n((t-\xi-\eta)^+). \non
\end{align}
In the FLLN, we have the deterministic equations (ODEs with delay): 
\begin{align}
\bar{I}(t) &= \bar{I}(0) (1-t/\xi)^+ + \lambda \int_{((t-\xi)^+, t]} (1-\bar{I}(s) - \bar{R}(s)) \bar{I}(s) ds,  \non\\
\bar{R}(t) &= \bar{R}(0) (1-t/\eta)^+ + \bar{I}(0) \xi^{-1} (t-(t-\eta)^+) +  \lambda \int_{((t-\xi-\eta)^+, (t-\xi)^+]} (1-\bar{I}(s) - \bar{R}(s)) \bar{I}(s) ds. \non
\end{align}
In the FCLT, we obtain 
\begin{equation}
\hat{I}(t) =  \hat{I}(0) (1-t/\xi)^+ + \lambda \int_{((t-\xi)^+, t]} \left( - \hat{I}(s) \bar{R}(s) + (1-\bar{I}(s) - 2 \bar{R}(s)) \hat{R}(s) \right) ds  + \hat{I}_0(t) + \hat{I}_1(t), \non 
\end{equation}
\begin{align*}
\hat{R}(t) & =\hat{R}(0)  (1-t/\eta)^+ +  \hat{I}(0) \xi^{-1} (t-(t-\eta)^+)  \non\\
& \quad +  \lambda \int_{((t-\xi-\eta)^+, (t-\xi)^+]}   \left( - \hat{I}(s) \bar{R}(s) + (1-\bar{I}(s) - 2 \bar{R}(s)) \hat{R}(s) \right)  ds   +   \hat{R}_{0,1}(t) + \hat{R}_{0,2}(t) +  \hat{R}_{1}(t),  
\end{align*}
where $\bar{I}$ and $\bar{R}$ are the fluid equations given above, and 
$\hat{I}_0(t)$, $\hat{I}_1(t)$, $ \hat{R}_{0,1}(t)$, $ \hat{R}_{0,2}(t)$ and $ \hat{R}_{1}(t)$ have the covariance functions: for $t, t'\ge 0$, 
\begin{align}
\Cov(\hat{I}_0(t),\hat{I}_0(t')) &= \bar{I}(0) ((1-t\vee t'/\xi)^+ - (1-t/\xi)^+ (1-t'/\xi)^+  ),  \non\\
\Cov(\hat{I}_1(t),\hat{I}_1(t')) &= \lambda \int_0^{t\wedge t'} \bone(t\vee t' -s <\xi) (1-\bar{I}(s) - \bar{R}(s)) \bar{I}(s) ds, \non\\
\Cov(\hat{R}_{0,1}(t),\hat{R}_{0,1}(t')) &= \bar{R}(0)  ((1-t\vee t'/\eta)^+ - (1-t/\eta)^+ (1-t'/\eta)^+  ), \non\\
\Cov(\hat{R}_{0,2}(t),\hat{R}_{0,2}(t'))  &= \bar{I}(0)  \xi^{-1}[ (t\vee t'- (t\vee t'-\eta)^+) -  (t- (t-\eta)^+) (t'- (t'-\eta)^+) ] , \non\\
\Cov(\hat{R}_{1}(t),\hat{R}_{1}(t'))  &= \lambda  \int_0^{t\wedge t'} \bone ( \xi \le t\vee t'-s < \xi +\eta)(1-\bar{I}(s) - \bar{R}(s)) \bar{I}(s) ds, \non
\end{align}
and similarly for the covariances between them.

\subsection{Equilibrium analysis for the SIS and SIRS models} \label{sec-equilibrium}

For a general distributions $F$ on $\RR_+$, its equilibrium  (stationary excess) distribution is defined by 
\begin{equation} \label{eqn-Fe}
F_e(t) := \frac{\int_0^t F^c(s)ds}{\int_0^\infty F^c(s) ds} = \mu\int_0^t F^c(s)ds, \quad t\ge 0,
\end{equation}
where $\mu^{-1} = \int_0^\infty F^c(s) ds$ is the mean of $F$. 

For the SIS model, in the Markovian case with $F_0(t) = F(t) = 1- e^{-\mu t}$, it is well known that  the ODE for $\bar{I}$, $\bar{I}' = \lambda (1-\bar{I}) \bar{I} - \mu \bar{I}$, has two equilibria, $\bar{I}^*=0$ or $\bar{I}^* = 1- \mu/\lambda$ if $\mu < \lambda$. 
For a general distributions $F$, if $F_0=F_e$,  by \eqref{SIS-barI}, 
an equilibrium $\bar{I}^*$ must satisfy 
\[ \bar{I}^*=\mu\bar{I}^*\int_t^\infty F^c(s)ds+\lambda\bar{I}^*(1-\bar{I}^*)\int_0^tF^c(s)ds,\]
hence either $\bar{I}^*=0$, or else by differentiating the last expression we find again $\bar{I}^*=1-\mu/\lambda$.

For the SIRS model, we obtain the following proposition for the nontrivial equilibrium point. 
\begin{prop} \label{prop-SIRS-equilibrium}
In the SIRS model with independent infectious and immune times,  assuming
$\E[\xi_1]= \gamma^{-1}$ and $\E[\eta_1]= \mu^{-1}$ satisfy $\lambda>\gamma$, 
  if $G_0(t)= G_e(t) $ and $F_0(t) =F_e(t)$, 
  there exists a unique nontrivial equilibrium $(\bar{S}^*,\bar{I}^*,\bar{R}^*)$, given by
\begin{equation}\label{eqn-SIRS-eql}
\bar{S}^*=\frac{\gamma}{\lambda}, \quad \bar{I}^*=\frac{1-\gamma/\lambda}{1+\gamma/\mu}, \quad \text{and} \quad \bar{R}^*=\frac{\gamma}{\mu} \bar{I}^* .
\end{equation}
\end{prop}

\begin{proof}
We prove the following two identities: 
\begin{align}
 \lambda(1-\bar{I}^*-\bar{R}^*)&=\gamma,  \label{SIRS-eqlb-one}\\
 \mu \bar{R}^* &= \gamma \bar{I}^*.  \label{SIRS-eqlb-two}
\end{align}
From these, by the identity $\bar{S}^*+\bar{I}^*+\bar{R}^*=1$, we obtain \eqref{eqn-SIRS-eql}. (The equations \eqref{SIRS-eqlb-one} and \eqref{SIRS-eqlb-two} are easily seen from the ODEs in the Markovian case.) 
By the equations for $\bar{I}(t)$ in \eqref{SIRS-barE}  and $\bar{R}(t)$ in \eqref{SIRS-barI}, 
the equilibrium quantities must satisfy 
 \begin{align*} 
 \bar{I}^*&=\gamma\bar{I}^*\int_t^\infty G^c(s)ds+\lambda\bar{I}^*(1-\bar{I}^*-\bar{R}^*)\int_0^tG^c(s)ds,\\
 \bar{R}^*&= \mu\bar{R}^* \int_t^\infty F^c(s)ds + \bar{I}^*\Psi_0(t) + \lambda (1-\bar{I}^* -\bar{R}^*)  \bar{I}^*   \int_0^t \Psi(s)  ds .  
 \end{align*}
 This system has the trivial solution $\bar{I}^*= \bar{R}^*=0$. We now look for another solution. Dividing the first identity by $\bar{I}^*$ and differentiating, we 
 obtain  \eqref{SIRS-eqlb-one}, and the second identity becomes
\[ \bar{R}^*= \mu\bar{R}^* \int_t^\infty F^c(s)ds + \bar{I}^*\Psi_0(t) + \gamma  \bar{I}^*   \int_0^t \Psi(s)  ds\, .\]
\eqref{SIRS-eqlb-two} now follows from the identity $\gamma^{-1}\Psi_0(t)+\int_0^t\Psi(s)ds=\int_0^t F^c(s)ds$.
To verify this, first note that from the definitions of $\Psi_0$ in the independent case, and of $G_0$,
\begin{align*}
\gamma^{-1}\Psi_0(t)=\int_0^tF^c(t-u)G^c(u)du=\int_0^tF^c(s)ds-\int_0^tF^c(t-u)G(u)du\,.
\end{align*}
It remains to note that by integration by parts and interchange of orders of integration
\begin{align*}
\int_0^tF^c(t-u)G(u)du&=\int_0^t\int_0^{t-u}F^c(v)dv dG(u) =\int_0^t\int_u^tF^c(v-u)dv dG(u)\\
&=\int_0^t\int_0^vF^c(v-u)dG(u)dv=\int_0^t\Psi(s)ds
\end{align*}
This completes the proof.
\end{proof}

For the FCLT in the SIRS model, if the system starts from the equilibrium, then  we can define the diffusion-scaled processes $\hat{I}^n= \sqrt{n} (\bar{I}^n- \bar{I}^*)$ and $\hat{R}^n= \sqrt{n} (\bar{R}^n- \bar{R}^*)$ and the FCLT holds with the limit processes $\hat{I}$ and $\hat{R}$ as given in Theorem \ref{thm-FCLT-SIRS} where the fluid limits $\bar{I}$ and $\bar{R}$ are replaced by $\bar{I}^*$ and $\bar{R}^*$. The same is true for the FCLT in the SIS model starting from the equilibrium.

\section{Proof of the FLLN for the SIR model}  \label{sec-SIR-FLLN-proof}

In this section we prove Theorem~\ref{thm-FLLN-SIR}.

We write the process $\bar{A}^n$ as 
\begin{equation}\label{barAn-rep}
\bar{A}^n(t) =\frac{1}{\sqrt{n}} \hat{M}^n_A(t)  +  \bar \Lambda^n(t), 
\end{equation}
where
$$
\bar{\Lambda}^n(t) := \lambda  \int_0^t  \bar{S}^n(s) \bar{I}^n(s) ds,
$$
and
\begin{align}\label{hatMn-def}
\hat{M}^n_A(t) :=  \frac{1}{\sqrt{n}} \left(A_*\left( n \bar\Lambda^n(t) \right) -  n \bar\Lambda^n(t) \right). 
\end{align}
The process $\{\hat{M}^n_A(t): t\ge 0\}$ is a square-integrable martingale with respect to the filtration $\{\sF^n_t: t \ge 0\}$ defined by 
\begin{align}
\sF^n_t := \sigma\left\{I^n(0), A_*\left(n \bar{\Lambda}^n(u) \right): 0 \le u \le t \right\}, \non
\end{align}
with the predictable quadratic variation 
\begin{align} \label{hatMn-QV}
\langle \hat{M}^n_A\rangle (t) =   \bar{\Lambda}^n(t),  \quad t \ge 0. 
\end{align}
These properties are straightforward to verify; see, e.g. \cite{pang2007martingale} or \cite{britton2018stochastic}. 
Note that   by the simple bound 
\begin{equation} \label{barSIn-bound}
\bar{S}^n(t) \le 1,\quad \bar{I}^n(t) \le 1, \quad \forall t \ge 0,  
\end{equation}
we have, w.p.1., for $0 < s \le t$,
\begin{equation} \label{barLambda-n-Lip}
0 \le \bar{\Lambda}^n(t) -  \bar{\Lambda}^n(s) \le \lambda (t-s). 
\end{equation}

\begin{lemma}\label{le:tightS} 
The sequence $\{(\bar{A}^n, \bar{S}^n): n \ge 1\}$ is tight in  $D^2$. 
\end{lemma}
\begin{proof}
By \eqref{barLambda-n-Lip},
 we have $\langle \hat{M}^n_A\rangle (t) \le  \lambda t $,  w.p.1. 
Thus, by \cite[Lemma 5.8]{pang2007martingale}, the martingale $\{\hat{M}^n_A(t): t\ge 0\}$ is stochastically bounded in $D$.  Then by \cite[Lemma 5.8]{pang2007martingale}, we have 
 \begin{equation} \label{hatMnA-conv}
\frac{1}{\sqrt{n}} \hat{M}^n_A \RA 0 \qinq D   \qasq n \to\infty. 
\end{equation}
Then, by \eqref{barAn-rep}, the tightness of the sequence $\{\bar{A}^n: n \ge 1\}$ follows directly by \eqref{barLambda-n-Lip}. Since $\bar{S}^n=1 - \bar{I}^n(0) -\bar{A}^n$, we obtain the tightness of $\{\bar{S}^n:n\ge 1\}$ in $D$ immediately.  
\end{proof}

\medskip

We work with a convergent subsequence of $(\bar{A}^n, \bar{S}^n)$. 
We denote the limit of $\bar{A}^n$ along the subsequence by $\bar{A}$.  It is clear from \eqref{barAn-rep} that the limit $\bar{A}$ satisfies
\begin{equation} \label{barAn-conv1}
\bar{A} = \lim_{n\to\infty} \bar{A}^n =  \lim_{n\to\infty} \bar{\Lambda}^n = \lim_{n\to\infty} \lambda \int_0^{\cdot} \bar{S}^n(s)\bar{I}^n(s) ds, 
\end{equation}
and for $0 < s \le t$, w.p.1, 
\begin{equation} \label{barA-Lip}
0 \le \bar{A}(t) - \bar{A}(s) \le \lambda (t-s). 
\end{equation}
By definition and Assumption~\ref{AS-SIR-1}, we have 
\begin{equation}\label{barSn-conv1}
\bar{S}^n = 1- \bar{I}^n(0) - \bar{A}^n \RA \bar{S}= 1- \bar{I}(0) - \bar{A}\  \text{ in } D, \qasq n\to\infty. 
\end{equation}

We next consider the process $\bar{I}^n$. Recall the expression of $I^n$ in \eqref{eqn-SIR-In}. 
Let 
$$
\bar{I}_0^n(t)  := \frac{1}{n} \sum_{j=1}^{n\bar{I}^n(0)}  \bone(\eta^0_j > t), 
\qandq
\breve{I}_0^n(t)  := \frac{1}{n} \sum_{j=1}^{n\bar{I}(0)}  \bone(\eta^0_j > t), \quad t\ge 0.
$$
We clearly have
\begin{equation}\label{barIn0-rep}
\left| \bar{I}_0^n(t) - \breve{I}_0^n(t) \right| \le \frac{1}{n} \sum_{j=n(\bar{I}^n(0) \wedge \bar{I}(0))}^{n(\bar{I}^n(0) \vee \bar{I}(0))}  \bone(\eta^0_j > t), \quad  t\ge 0.
\end{equation}
Note that by Assumption~\ref{AS-SIR-1}, the right--hand side satisfies
\begin{equation}\label{barIn0-breveIn0-conv}
\E\left[\frac{1}{n} \sum_{j=n(\bar{I}^n(0) \wedge \bar{I}(0))}^{n(\bar{I}^n(0) \vee \bar{I}(0))}  \bone(\eta^0_j > t) \Big| \sF^n_0\right] \le F_0^c(t) |\bar{I}^n(0) - \bar{I}(0)|  \to 0
\end{equation}
in probability as $n\to\infty$. 
Thus, by the FLLN of empirical processes  (that is, for a sequence of i.i.d. random variables $\{\xi_i\}$ with c.d.f. $F$,  $ n^{-1}\sum_{i=1}^{n} \bone_{\xi_i \le t} \to F(t)$ in $D$ in probability as $n\to \infty$; this follows from the FCLT in Theorem 14.3 in \cite{billingsley1999convergence}), we obtain that in probability,
\begin{align} \label{barIn0-conv}
\bar{I}_0^n   \to \bar{I}_0 = \bar{I}(0) F_0^c(\cdot) \qinq D \qasq n \to\infty. 
\end{align}

Let 
$$
\bar{I}_1^n(t):= \frac{1}{n}  \sum_{i=1}^{n\bar{A}^n(t)} \bone(\tau^n_i + \eta_i >t), \quad t\ge 0, 
$$
and its conditional expectation
$$
\breve{I}_1^n(t) := \E[\bar{I}_1^n(t)|\sF^n_t]  = \frac{1}{n}  \sum_{i=1}^{n\bar{A}^n(t)} F^c(t-\tau^n_i ) = \int_0^t F^c(t-s) d\bar{A}^n(s), \quad t\ge 0.
$$
By integration by parts, we have
\begin{align}
\breve{I}_1^n(t)  = \bar{A}^n(t) - \int_0^t \bar{A}^n(s) d F^c(t-s). \non 
\end{align}
Here $d F^c(t-s)$ is the differential of the map $s\to F^c(t-s)$. 
By the continuous mapping theorem applied to the map $x\in D \to  x-\int_0^{\cdot}  x(s)dF^c(\cdot -s) \in D$, exploiting the fact that 
$\bar{A}=\lim\bar{A}^n\in C$ a.s., 
\begin{equation} \label{breve-In1-conv}
\breve{I}_1^n\to \bar{I}_1 \qinq D
\end{equation}
in probability as $n\to\infty$,  where
\begin{align}
\bar{I}_1(t)&=\bar{A}(t)- \int_0^t \bar{A}(s) d F^c(t-s)=\int_0^t F^c(t-s)d\bar{A}(s)\,, \quad t \ge 0\,. \non
\end{align}

Let 
$$
V^n(t) := \bar{I}_1^n(t) - \breve{I}_1^n(t) =  \frac{1}{n}  \sum_{i=1}^{n\bar{A}^n(t)} \chi^n_i(t), \quad t\ge 0,
$$
where
$$
\chi^n_i(t) :=  \bone(\tau^n_i + \eta_i >t) - F^c(t-\tau^n_i ). 
$$

We next show the following lemma.

\begin{lemma} \label{lem-Vn-SIR-conv}
For any $\ep>0$,
\begin{equation} \label{supVn-conv}
\P\left(\sup_{t \in [0,T]} |V^n(t)| \ge \ep\right) \to 0 \qasq n \to \infty. 
\end{equation}
\end{lemma}

\begin{proof}

Note that  by partitioning $[0,T]$ into intervals of length $\delta$, that is, $[t_i, t_{i+1})$, $i =0,\dots, [T/\delta]$ with $t_0=0$,  we have 
\begin{equation} \label{supVn-decomp}
\sup_{t \in [0,T]} |V^n(t)|  \le  \sup_{i=1, \dots, [T/\delta]} |V^n(t_i)| + \sup_{i=1, \dots, [T/\delta]} \sup_{u \in [0,\delta]}  |V^n(t_i+u) - V^n(t_i)|. 
\end{equation}

It is easy to check that
$$
\E[\chi^n_i(t)| \sF^n_t] = 0, \quad \forall i;
\quad
\E[\chi^n_i(t)\chi^n_j(t)| \sF^n_t] = 0, \quad \forall i\neq j. 
$$
Thus, we have 
\begin{align}
\E\big[V^n(t)^2 \big| \sF^n_t\big] & =  \frac{1}{n^2}\sum_{i=1}^{A^n(t)}E\big[\chi^n_i(t)^2| \sF^n_t\big] = \frac{1}{n^2}\sum_{i=1}^{A^n(t)}F(t-\tau_i^n)F^c(t-\tau_i^n) \non\\
& =  \frac{1}{n} \int_0^t F(t-s)F^c(t-s) d \bar{A}^n(s)  \non\\
& =  \frac{1}{n^{3/2}}  \int_0^t F(t-s)F^c(t-s) d \hat{M}_A^n(s) +   \frac{1}{n} \int_0^t F(t-s)F^c(t-s) d \bar{\Lambda}^n(s) \non\\
& \le  \frac{1}{n^{3/2}}  \int_0^t F(t-s)F^c(t-s) d \hat{M}_A^n(s)  + \frac{\lambda t}{n}, \non
\end{align}
where the inequality follows from \eqref{barSIn-bound} and  \eqref{barLambda-n-Lip}.
Thus
\begin{equation}\label{Vn-t-conv}
\E[ |V^n(t)|^2]  \le \frac{\lambda t}{n},
\end{equation}
and for any $\ep>0$,
\[  \P(|V^n(t)|>\ep)\le \frac{\lambda t}{n\ep^2}\to 0, \ \text{ as }n\to\infty\,.\]
We now consider $V^n(t+u) - V^n(t)$ for $t,u\ge 0$. 
By definition, we have
\begin{align}
|V^n(t+u) - V^n(t)| & = \left|  \frac{1}{n}\sum_{i=1}^{A^n(t+u)} \chi^n_i(t+u) -   \frac{1}{n}\sum_{i=1}^{A^n(t)} \chi^n_i(t) \right| \non\\
&=  \left| \frac{1}{n}\sum_{i=1}^{A^n(t)} (\chi^n_i(t+u) - \chi^n_i(t))  +   \frac{1}{n}\sum_{i=A^n(t)}^{{A^n(t+u)}} \chi^n_i(t+u) \right| \non\\
& \le  \frac{1}{n}\sum_{i=1}^{A^n(t)} \bone(t< \tau^n_i+\eta_i\le t+u)  +  \int_0^{t+u} (F^c(t-s) -F^c(t+u-s)) d \bar{A}^n(s) \non\\
&\qquad +   \frac{1}{n}\sum_{i=A^n(t)}^{{A^n(t+u)}} | \chi^n_i(t+u)| .  \non
\end{align}
Observing that the first and second terms on the right hand  are increasing in $u$, and that $|\chi^n_i(t)| \le 1$, we obtain
\begin{align} \label{deltaVn-bound}
\sup_{u\in [0,\delta]}|V^n(t+u) - V^n(t)| 
&\le  \frac{1}{n}\sum_{i=1}^{A^n(t)} \bone(t< \tau^n_i+\eta_i\le t+\delta)  \\
& \qquad +  \int_0^{t+\delta} (F^c(t-s) -F^c(t+\delta-s)) d \bar{A}^n(s)  + \big( \bar{A}^n(t+\delta) - \bar{A}^n(t) \big). \non
\end{align}
Thus, for any $\ep>0$, 
\begin{align} \label{deltaVn-bound-p1}
& \P \left( \sup_{u\in [0,\delta]}|V^n(t+u) - V^n(t)| > \ep\right) \non\\
&\le  \P\left( \frac{1}{n}\sum_{i=1}^{A^n(t)} \bone(t< \tau^n_i+\eta_i\le t+\delta)   > \ep/3 \right) \non\\
& \qquad +  \P\left( \int_0^{t+\delta} (F^c(t-s) -F^c(t+\delta-s)) d \bar{A}^n(s)  > \ep/3\right) +   \P\left( \bar{A}^n(t+\delta) - \bar{A}^n(t)> \ep/3\right) \non\\
& \le  \frac{9}{\ep^2}   \E \left[ \left( \frac{1}{n}\sum_{i=1}^{A^n(t)} \bone(t< \tau^n_i+\eta_i\le t+\delta)  \right)^2 \right] \non\\
& \qquad + \frac{9}{\ep^2} \E\left[ \left( \int_0^{t+\delta} (F^c(t-s) -F^c(t+\delta-s)) d \bar{A}^n(s)\right)^2 \right]  + \frac{9}{\ep^2}  \E\left[\left( \bar{A}^n(t+\delta) - \bar{A}^n(t)\right)^2 \right]. 
\end{align}

We need the following definition to treat the first term on the right hand side of  \eqref{deltaVn-bound-p1}.  

 \begin{definition} \label{def-PRM-SIR-1}
 Let $M(ds,dz,du)$ denote a Poisson random measure (PRM)   on $[0,T]\times \RR_+ \times  \RR_+$  which is the sum of the Dirac masses at the points $(\tau_i^n,  \eta_i, U^n_i)$ with  mean measure $\nu(ds, dz, du)=ds F(dz) du $, and $\overline{M}(ds,dz,du)$ denote the associated compensated PRM. 
 \end{definition}
We have 
\begin{align} \label{deltaVn-bound-p2}
&  \E \left[ \left( \frac{1}{n}\sum_{i=1}^{A^n(t)} \bone(t< \tau^n_i+\eta_i\le t+\delta)  \right)^2 \right] \non\\
&= \E \left[ \left(  \frac{1}{n}\int_0^t \int_{t-s}^{t+\delta-s} \int_0^\infty \bone(u \le \lambda n \bar{S}^n(s^-) \bar{I}^n(s^-)) M(ds,dz,du)   \right)^2 \right] \non\\
& \le  2\E \left[ \left(  \frac{1}{n}\int_0^t \int_{t-s}^{t+\delta-s} \int_0^\infty \bone(u \le \lambda n \bar{S}^n(s^-) \bar{I}^n(s^-)) \overline{M}(ds,dz,du)   \right)^2 \right] \non\\
& \qquad +  2\E\left[ \left(\int_0^{t} (F^c(t-s) -F^c(t+\delta-s)) d \bar{\Lambda}^n(s)\right)^2\right]  \non\\
&=    \frac{2}{n} \E\left[  \int_0^{t} (F^c(t-s) -F^c(t+\delta-s)) d \bar{\Lambda}^n(s) \right]  +    2\E\left[ \left(\int_0^{t} (F^c(t-s) -F^c(t+\delta-s)) d \bar{\Lambda}^n(s)\right)^2\right]  \non\\
& \le 
\frac{2}{n} \lambda \int_0^{t+\delta} (F^c(t-s) -F^c(t+\delta-s)) ds
 +  2\left(  \lambda \int_0^{t} (F^c(t-s) -F^c(t+\delta-s)) ds\right)^2\,. 
\end{align}
The last inequality follows from \eqref{barLambda-n-Lip}.
The first term on the right hand converges to zero as $n\to \infty$, and for the second term, we have
\begin{align} \label{deltaVn-bound-p3}
\frac{1}{\delta}\left(  \int_0^{t} (F^c(t-s) -F^c(t+\delta-s)) ds\right)^2
&=\frac{1}{\delta}\left(\int_t^{t+\delta}F(s)ds-\int_0^\delta F(s)ds\right)^2\non\\
&\le\delta\to0 \qasq \delta \to 0. 
\end{align}

For the second term on the right hand side of \eqref{deltaVn-bound-p1},  by \eqref{barAn-rep}, we have 
\begin{align}
& \E\left[ \left( \int_0^{t+\delta} (F^c(t-s) -F^c(t+\delta-s)) d \bar{A}^n(s)\right)^2 \right]  \non\\
& \le 2  \E\left[ \left( \frac{1}{\sqrt{n}} \int_0^{t+\delta} (F^c(t-s) -F^c(t+\delta-s)) d \hat{M}_A^n(s)\right)^2 \right]\non  \\
& \qquad +  2\E\left[ \left( \int_0^{t+\delta} (F^c(t-s) -F^c(t+\delta-s)) d \bar{\Lambda}^n(s)\right)^2 \right]. \non
\end{align}
By \eqref{hatMnA-conv}, the first term converges to zero as $n\to \infty$. By \eqref{barLambda-n-Lip}, the second term is bounded by
$$
2 \left( \lambda \int_0^{t+\delta} (F^c(t-s) -F^c(t+\delta-s)) ds \right)^2\,,
$$
to which \eqref{deltaVn-bound-p3} again applies.

By \eqref{barAn-rep} and \eqref{barLambda-n-Lip}, we have
\begin{align}
\bar{A}^n(t+\delta) - \bar{A}^n(t)\le \frac{1}{\sqrt{n}} (\hat{M}^n(t+\delta) - \hat{M}^n(t)) + \lambda \delta. \non
\end{align}
Thus, for the third term on the right hand side of \eqref{deltaVn-bound-p1}, we have
\begin{align} \label{deltaVn-bound-p4}
 \E\left[\left( \bar{A}^n(t+\delta) - \bar{A}^n(t)\right)^2 \right] &\le
2   \E\left[\left( \ \frac{1}{\sqrt{n}} (\hat{M}^n(t+\delta) - \hat{M}^n(t))\right)^2 \right] +  2 \lambda^2 \delta^2. 
\end{align}
Again, by \eqref{hatMnA-conv}, the first term converges to zero as $n\to \infty$.

By \eqref{supVn-decomp}, we have for $\delta>0$, 
\begin{align}\label{supVn-decomp-ineq}
 \P\left(\sup_{t \in [0,T]} |V^n(t)| \ge \ep\right) &  \le  \left[\frac{T}{\delta}\right] \sup_{t\in[0,T]}  \P\left( |V^n(t)| \ge \ep/2\right) \non\\
 &\qquad +  \left[\frac{T}{\delta}\right]\sup_{t\in[0,T]}  \P \left( \sup_{u\in [0,\delta]}|V^n(t+u) - V^n(t)| > \ep/2\right).
\end{align}
The first term converges to zero as $n\to \infty$ by \eqref{Vn-t-conv}. 
By \eqref{deltaVn-bound-p1}--\eqref{deltaVn-bound-p4} and the above arguments, we obtain
\begin{equation}
\lim_{\delta \to 0} \limsup_{n\to\infty}  \,  \left[\frac{T}{\delta}\right] \sup_{t\in[0,T]} \,   \P\left(\sup_{u\in [0,\delta]}|V^n(t+u) - V^n(t)|  \ge \ep \right) =0.  \non
\end{equation}
Therefore, we have shown that \eqref{supVn-conv} holds.  
\end{proof}

By the convergence of $\breve{I}_1^n$ in \eqref{breve-In1-conv} and Lemma \ref{lem-Vn-SIR-conv}, we obtain in probability 
\begin{equation} 
\bar{I}_1^n(t) \to \bar{I}_1(t) =  \int_0^t F^c(t-s) d\bar{A}(s)  \qinq D \qasq n \to\infty. \non
\end{equation}
Combining this  with \eqref{barIn0-conv}, we  have 
\begin{equation}
\bar{I}^n = \bar{I}_0^n+ \bar{I}_1^n  \to  \bar{I}:=\bar{I}_0 + \bar{I}_1 = \bar{I}(0) F_0^c(\cdot) +   \int_0^{\cdot} F^c(\cdot-s) d\bar{A}(s) \non
\end{equation}
in $ D$ in probability as $n \to\infty$.  Note that there are no common jumps in $\bar{I}_0^n$ and $\bar{I}_1^n$ and the limit of  $\bar{I}_0^n$ is in $D$ if $F_0$ is discontinuous while the limit of  $\bar{I}_1^n$ is in $C$; thus, the continuous mapping theorem can be applied for the addition.

We now show the joint convergence in probability 
\begin{equation}\label{barSn-breveIn-jointconv}
(\bar{S}^n, \bar{I}^n) \to (\bar{S}, \bar{I}) \qinq D^2 \qasq n\to\infty. 
\end{equation}
Recall that $\bar{S}^n = 1- \bar{I}^n(0) - \bar{A}^n$ as in \eqref{barSn-conv1}.
We first prove the joint convergence of $ (\bar{I}^n(0), \bar{I}^n_0) \to (\bar{I}(0), \bar{I}_0)$ in $\RR_+\times D$. This follows from the joint convergence of $ (\bar{I}^n(0), \breve{I}^n_0) \to (\bar{I}(0), \bar{I}_0)$ in $\RR_+\times D$ by independence and the asymptotic negligence of the difference $ \bar{I}^n_0 -  \breve{I}^n_0 \to 0$ as shown in \eqref{barIn0-rep} and \eqref{barIn0-breveIn0-conv}. 
We next prove the joint convergence of $(\bar{A}^n, \bar{I}^n_1) \to (\bar{A}, \bar{I}_1)$ in $D^2$. We obtain the joint convergence of  $(\bar{A}^n, \breve{I}^n_1) \to (\bar{A}, \bar{I}_1)$ in $D^2$ by applying the continuous mapping theorem  to the map $x\in D \to (x, x-\int_0^{\cdot}  x(s)dF^c(\cdot -s)) \in D^2$. Then the claim follows from Lemma \ref{lem-Vn-SIR-conv}. Since the two groups of processes $ (\bar{I}^n(0), \bar{I}^n_0)$ and
$(\bar{A}^n, \bar{I}^n_1)$ are independent,  we have the joint convergence $ (\bar{I}^n(0), \bar{A}^n, \bar{I}^n_0,  \bar{I}^n_1)$, and thus 
 conclude the joint convergence of $(\bar{S}^n, \bar{I}^n)$ in  \eqref{barSn-breveIn-jointconv} by applying the continuous mapping theorem again. 

Thus we obtain in probability
\begin{align} \label{eqn-intSnIn-conv2}
\int_0^\cdot \bar{S}^n(s) \bar{I}^n(s) ds \to \int_0^\cdot \bar{S}(s) \bar{I}(s) ds \qinq D \qasq n \to \infty. 
\end{align}
By \eqref{barAn-rep} and \eqref{hatMnA-conv}, this implies that in probability 
\begin{align}
\bar{A}^n \to \bar{A}=  \lambda \int_0^\cdot \bar{S}(s) \bar{I}(s)ds \qinq D \qasq n \to \infty. \non 
\end{align}
Therefore, the limits $\bar{S}$ and $\bar{I}$ satisfy the integral equations given in \eqref{SIR-barS} and \eqref{SIR-barI}.

We next prove uniqueness of the solution to the system of equations \eqref{SIR-barS} and \eqref{SIR-barI}. 
 The two equations \eqref{SIR-barS} and \eqref{SIR-barI} can be regarded as Volterra integral equations of the second kind for two functions. 
For uniqueness,  suppose there are two solutions $(\bar{S}_1, \bar{I}_1)$ and 
$(\bar{S}_2, \bar{I}_2)$. 
Then we have 
\begin{align}
\bar{S}_1(t) - \bar{S}_2(t) &= -\lambda \int_0^t \Big( (\bar{S}_1(s) -\bar{S}_2(s)) \bar{I}_1(s) +  \bar{S}_2(s) (\bar{I}_1(s) - \bar{I}_2(s)  ) \Big) ds , \non\\
\bar{I}_1(t) - \bar{I}_2(t) &= \lambda \int_0^t F^c(t-s)  \Big( (\bar{S}_1(s) -\bar{S}_2(s)) \bar{I}_1(s) +  \bar{S}_2(s) (\bar{I}_1(s) - \bar{I}_2(s)  ) \Big)  ds.\non
\end{align}
Hence,
\begin{align}
& |\bar{S}_1(t) - \bar{S}_2(t) | + |\bar{I}_1(t) - \bar{I}_2(t)|  \le 2 \lambda  \int_0^t \Big( |\bar{S}_1(s) -\bar{S}_2(s)|  +  |\bar{I}_1(s) - \bar{I}_2(s)  | \Big) ds, \non 
\end{align}
where we use the simple bounds $\bar{S}_i(s)\le 1$ and $\bar{I}_i(s)\le 1$. 
The uniqueness follows from applying Gronwall's inequality.

Since the system of integral equations \eqref{SIR-barS} and \eqref{SIR-barI} has a unique deterministic solution (existence is easily established by a standard Picard iteration argument, identical to the classical one for Lipschitz ODEs), the whole sequence converges, and we have convergence in probability.

\section{Proof of the FCLT for the SIR model} \label{sec-SIR-FCLT-proof}

In this section we prove Theorem~\ref{thm-FCLT-SIR}. 
Recall the definitions of the diffusion-scaled processes $(\hat{S}^n, \hat{I}^n, \hat{R}^n)$ in \eqref{SIR-diff-def}, and $\hat{M}_A^n$ defined in \eqref{hatMn-def}. 
We also define
$$
\hat{A}^n(t) :=\sqrt{n} \left(\bar{A}^n(t)  -\bar{A}(t) \right) =  \sqrt{n} \left(\bar{A}^n(t) -  \lambda \int_0^t \bar{S}(s) \bar{I}(s) ds\right). 
$$
Note that under Assumption~\ref{AS-SIR-2}, we have $\bar{I}^n(0) \RA \bar{I}(0)$ in $\RR$ as $n\to\infty$, and thus the convergence of the fluid-scaled processes holds in Theorem~\ref{thm-FLLN-SIR}. This is taken as given in the proceeding proof of the FCLT.

 By the definitions of the diffusion-scaled processes in \eqref{SIR-diff-def}, we have
  \begin{align}
 \hat{A}^n(t) &= \hat{M}_A^n(t) +   \lambda  \int_0^t ( \hat{S}^n(s) \bar{I}^n(s) + \bar{S}(s) \hat{I}^n(s) ) ds, \label{hatAn-rep-SIR}\\
  \hat{S}^n(t) &= -\hat{I}^n(0) - \hat{A}^n(t)  = -\hat{I}^n(0)-\hat{M}_A^n(t) -   \lambda  \int_0^t ( \hat{S}^n(s) \bar{I}^n(s) + \bar{S}(s) \hat{I}^n(s) ) ds, \label{hatSn-rep-SIR}
 \end{align}
\begin{align}
\hat{I}^n(t)  &=  \hat{I}^n(0) F_0^c(t)  + \hat{I}_0^n(t)  + \hat{I}_1^n(t)   +  \lambda  \int_0^t F^c(t-s) \left( \hat{S}^n(s) \bar{I}^n(s) + \bar{S}(s) \hat{I}^n(s) \right) ds, \label{hatIn-rep-SIR}
\end{align}
and
\begin{align}
\hat{R}^n(t) &= \hat{I}^n(0) F_0(t)  + \hat{R}_0^n(t)  + \hat{R}_1^n(t)   +  \lambda  \int_0^t F(t-s) ( \hat{S}^n(s) \bar{I}^n(s) + \bar{S}(s) \hat{I}^n(s) ) ds, \label{hatRn-rep-SIR}
\end{align}
where 
\begin{align}
\hat{I}_0^n(t) &:= \frac{1}{\sqrt{n}} \sum_{j=1}^{n\bar{I}^n(0)} \big( \bone(\eta^0_j > t) -  F_0^c(t)\big), \label{hatIn0-rep} \\
\hat{I}_1^n(t) &:= \frac{1}{\sqrt{n}}   \sum_{i=1}^{n\bar{A}^n(t)} \bone(\tau^n_i + \eta_i >t)  - \lambda \sqrt{n} \int_0^t F^c(t-s)  \bar{S}^n(s) \bar{I}^n(s) ds, \label{hatIn1-rep}
\end{align}
 \begin{align}
\hat{R}_0^n(t) &:= \frac{1}{\sqrt{n}} \sum_{j=1}^{n\bar{I}^n(0)} \big( \bone(\eta^0_j \le t) -  F_0(t)\big),  \label{hatRn0-rep}\\
\hat{R}_1^n(t) &:=\frac{1}{\sqrt{n}}   \sum_{i=1}^{n\bar{A}^n(t)} \bone(\tau^n_i + \eta_i \le t) - \lambda \sqrt{n} \int_0^t F(t-s)  \bar{S}^n(s) \bar{I}^n(s) ds.  \label{hatRn1-rep}
\end{align}

 We first establish the following joint convergence of the initial quantities. 
 \begin{lemma}\label{lem-hatInitials-conv-SIR}
 Under Assumption \ref{AS-SIR-2}, we have
 \begin{equation}
 (\hat{I}^n(0) F_0^c(\cdot), \hat{I}^n(0) F_0(\cdot), \hat{I}^n_0, \hat{R}^n_0)\RA  \left(\hat{I}(0) F_0^c(\cdot), \hat{I}(0) F_0(\cdot), \hat{I}_0, \hat{R}_0\right)
 \end{equation}
 in $D^4$ as $n\to\infty$, 
 where the limit processes $\hat{I}_0$ and $\hat{R}_0$ are as defined in Theorem \ref{thm-FCLT-SIR}.  
 \end{lemma}
 
 \begin{proof}
 We define 
 \begin{align*}
\widetilde{I}_0^n(t) &:= \frac{1}{\sqrt{n}} \sum_{j=1}^{n\bar{I}(0)} \big( \bone(\eta^0_j > t) -  F_0^c(t)\big), \quad
\widetilde{R}_0^n(t) := \frac{1}{\sqrt{n}} \sum_{j=1}^{n\bar{I}(0)} \big( \bone(\eta^0_j \le t) -  F_0(t)\big). 
\end{align*}
By the FCLT for empirical processes, see, e.g., \cite[Theorem 14.3]{billingsley1999convergence},  we have the joint convergence 
 \begin{equation*}
 (\hat{I}^n(0) F_0^c(\cdot), \hat{I}^n(0) F_0(\cdot), \widetilde{I}^n_0, \widetilde{R}^n_0)\RA  \left(\hat{I}(0) F_0^c(\cdot), \hat{I}(0) F_0(\cdot), \hat{I}_0, \hat{R}_0\right) 
 \end{equation*}
 in $D^4$ as $n\to\infty$. 
   The claim then follows by showing that $\widetilde{I}^n_0 - \hat{I}^n_0\RA 0$ in $D$ as $n\to\infty$,  and $\widetilde{R}^n_0 - \hat{R}^n_0\RA 0$ in $D$ as $n\to\infty$.  We focus on $\widetilde{I}^n_0 - \hat{I}^n_0\RA 0$.  We have    for each $t\ge 0$, 
   $\E[\widetilde{I}^n_0(t) - \hat{I}^n_0(t)]=0$ and 
$$
\E[|\widetilde{I}^n_0(t) - \hat{I}^n_0(t)|^2] = F_0^c(t) F_0(t) E[|\bar{I}^n(0) - \bar{I}(0)|] \to 0 \qasq n \to \infty,
$$
where the convergence follows from Assumption~\ref{AS-SIR-2}. 
It then suffices to show that $\{\widetilde{I}^n_0 - \hat{I}^n_0: n \ge 1\}$ is tight.
 We have 
 \begin{align*}
\text{sign}(\bar{I}(0)-\bar{I}^n(0))\left( \widetilde{I}^n_0(t) - \hat{I}^n_0(t)\right) & = \frac{1}{\sqrt{n}} \sum_{j=n(\bar{I}^n(0) \wedge \bar{I}(0)}^{n(\bar{I}^n(0) \vee \bar{I}(0)} \big( \bone(\eta^0_j > t) -  F_0^c(t)\big)  \non\\
 & = |\hat{I}^n(0)| F_0(t) 
- \frac{1}{\sqrt{n}} \sum_{j=n(\bar{I}^n(0) \wedge \bar{I}(0)}^{n(\bar{I}^n(0) \vee \bar{I}(0)} \bone(\eta^0_j \le t). 
 \end{align*}
 By Assumption \ref{AS-SIR-2}, the first term on the right hand side is tight. 
 Denoting the second term by $\Theta_{0}^n(t)$, since it is increasing in $t$, 
 by the Corollary on page 83 in \cite{billingsley1999convergence}, see also the use of \eqref{supVn-decomp} in the proof of Lemma \ref{lem-Vn-SIR-conv} above, its tightness will follow from the fact that  
for any $\ep>0$, 
\begin{equation*}\label{Theta-n-tight}
\limsup_{n\to\infty} \frac{1}{\delta}\P\big( \big|\Theta_{0}^n(t+\delta) - \Theta_{0}^n(t) \big| \ge \ep \big) \to 0\qasq \delta \to 0.
\end{equation*}
This is immediate since by Assumption~\ref{AS-SIR-2}, 
\begin{align*}
\E\left[  \big|\Theta_{0}^n(t+\delta) - \Theta_{0}^n(t) \big|^2\right]  =E[|\bar{I}^n(0) - \bar{I}(0)|]  |F_0(t+\delta) -F_0(t) | \to 0 \qasq n \to \infty.
\end{align*}
This completes the proof. 
 \end{proof}

 Recall the PRM $M(ds,dz,du)$ and the compensated PRM   $\overline{M}(ds,dz,du)$ in Definition \ref{def-PRM-SIR-1}. 
 \begin{definition} \label{def-PRM-SIR-2}
 Let  $M_1(ds,dz,du)$ be the PRM on $[0,T]\times \RR_+ \times  \RR_+$  with mean measure $\tilde\nu(ds,dz,du) = ds F_s(dz) du $, where $F_s((a,b]) = F((a+s, b+s])$.  Denote the associated compensated PRM by $\widetilde{M}(ds,dz,du)$. 
 \end{definition}
 
 We can rewrite the processes $\hat{I}^n_1$ and $\hat{R}^n_1$ as
 \begin{align*}
 \hat{I}^n_1(t) & = \frac{1}{\sqrt{n}} \int_0^t \int_{t-s}^{\infty} \int_0^\infty \varphi_n(s,u)  \overline{M}(ds,dz,du) = \frac{1}{\sqrt{n}} \int_0^t \int_{t}^{\infty} \int_0^\infty \varphi_n(s,u)  \widetilde{M}(ds,dz,du) ,\\
  \hat{R}^n_1(t) & = \frac{1}{\sqrt{n}} \int_0^t \int_0^{t-s} \int_0^\infty \varphi_n(s,u)  \overline{M}(ds,dz,du)  = \frac{1}{\sqrt{n}} \int_0^t \int_{0}^{t} \int_0^\infty \varphi_n(s,u)  \widetilde{M}(ds,dz,du), 
 \end{align*}
 where
 $$
 \varphi_n(s,u) = \bone\left(u \le n \lambda \bar{S}^n(s^-)\bar{I}^n(s^-)\right).  
 $$
 We also observe that the process $\hat{M}^n_A$ can also be represented by the same PRMs:
 $$
 \hat{M}^n_A(t) = \frac{1}{\sqrt{n}} \int_0^t \int_0^{\infty} \int_0^\infty \varphi_n(s,u)  \overline{M}(ds,dz,du)  = \frac{1}{\sqrt{n}} \int_0^t \int_{0}^{\infty} \int_0^\infty \varphi_n(s,u)  \widetilde{M}(ds,dz,du), 
 $$
 and that 
 $$
  \hat{M}^n_A(t) =   \hat{I}^n_1(t)  +  \hat{R}^n_1(t), \quad t \ge 0. 
 $$ 
 We define the auxiliary processes 
 $\widetilde{I}^n_1$ and $\widetilde{R}^n_1$ by
 \begin{align*}
 \widetilde{I}^n_1(t) & = \frac{1}{\sqrt{n}} \int_0^t \int_{t}^{\infty} \int_0^\infty \widetilde\varphi_n(s,u)  \widetilde{M}(ds,dz,du) ,\non\\
  \widetilde{R}^n_1(t) & = \frac{1}{\sqrt{n}} \int_0^t \int_0^{t} \int_0^\infty \widetilde\varphi_n(s,u) \widetilde{M}(ds,dz,du), \non \\
  \widetilde{M}^n_A(t) &= \frac{1}{\sqrt{n}} \int_0^t \int_0^{\infty}  \int_0^\infty \widetilde\varphi_n(s,u) \widetilde{M}(ds,dz,du), 
 \end{align*}
 where
 $$
 \widetilde\varphi_n(s,u) = \bone\left(u \le n \lambda \bar{S}(s^-)\bar{I}(s^-)\right).  
 $$
Note that in the definitions of $ \widetilde{I}^n_1(t)$ and $ \widetilde{R}^n_1(t)$, we have replaced $\bar{S}^n(s)$ and $\bar{I}^n(s)$ in the integrands $\varphi_n(s,u)$ by the deterministic fluid functions $\bar{S}(s)$ and $\bar{I}(s)$. 
Also, it is clear that 
$$
  \widetilde{M}^n_A(t)  =  \widetilde{I}^n_1(t) +   \widetilde{R}^n_1(t), \quad t \ge 0. 
$$

We first prove the following result.
\begin{lemma} \label{lem-hatSn-t-2-bound}
$$
\sup_n\E\left[\sup_{t \in [0,T]}|\hat{S}^n(t)|^2\right] <\infty, \quad \sup_n\E\left[\sup_{t \in [0,T]}|\hat{I}^n(t)|^2\right] <\infty, \quad \sup_n\E\left[\sup_{t \in [0,T]}|\hat{R}^n(t)|^2\right] <\infty. 
$$
\end{lemma}

\begin{proof}
The proof will be split in two steps. In step 1, we shall prove the estimates with $\sup_{t \in [0,T]}$ outside the expectations, and in step 2 we shall prove the result.

{\sc Step 1}
We have 
$$
\sup_{t \in [0,T]} \E[\hat{M}_A^n(t)^2] \le \lambda T . 
$$
It is clear that there exists a constant $C$ such that for all $n\ge1$,
$$
\sup_{t \in [0,T]} \E[(\hat{I}^n(0) F_0^c(t))^2]  \le \E[\hat{I}^n(0)^2] \le C,
$$
$$
\sup_{t \in [0,T]} \E[(\hat{I}_0^n(t))^2]  = \sup_{t \in [0,T]} \E[\bar{I}^n(0)] F_0(t) F_0^c(t) \le \E[\bar{I}^n(0)]  \le C,
$$
and 
$$
\sup_{t \in [0,T]} \E[(\hat{I}_1^n(t))^2] =  \sup_{t \in [0,T]}  \lambda \int_0^t F^c(t-s) \bar{S}^n(s) \bar{I}^n(s) ds  \le \lambda T. 
$$
Then by taking the square of the representations of $\hat{S}^n(t)$ in \eqref{hatSn-rep-SIR} and  $\hat{I}^n(t)$ in \eqref{hatIn-rep-SIR}, then using Cauchy-Schwartz inequality and the simple bounds $\bar{I}^n(t) \le 1$ and $\bar{S}(t) \le 1$, we can apply Gronwall's inequality to conclude the claim. 

{\sc Step 2} It follows from Step 1 and Doob's inequality that $\E\big[\sup_{t\in[0,T]}|\hat{M}_A^n(t)|^2\big]<\infty$,
from which the result concerning $\hat{S}^n$ follows readily. Concerning $\hat{I}^n$, we need to establish both
 \begin{align*}
  \sup_n\E\bigg[\sup_{0\le t\le T}\big(\hat{I}^n_0(t)\big)^2\bigg]<\infty,\qandq 
   \sup_n\E\bigg[\sup_{0\le t\le T}\big(\hat{I}^n_1(t)\big)^2\bigg]<\infty\,. 
  \end{align*}
Let us first consider the second term. We use the decomposition 
$\hat{I}^n_1=\widetilde{I}_1^n+\big[\hat{I}^n_1-\widetilde{I}_1^n\big]$. Concerning $\widetilde{I}_1^n$, we exploit the fact that $\widetilde{I}_1^n=\widetilde{M}_A^n-\widetilde{R}_1^n$, which is a difference of two martingales, to each of which we can apply Doob's inequality, which yields that 
$\sup_n\E\Big[\sup_{0\le t\le T}\big(\widetilde{I}^n_1(t)\big)^2\Big]<\infty$. The difference 
$\hat{I}^n_1-\widetilde{I}_1^n$ is easy to treat. Indeed, 
 \begin{align*}
 \hat{I}^n_1(t)-\widetilde{I}_1^n(t)=\int_0^t\int_t^\infty\int_0^\infty\rho_n(s,u)\widetilde{M}(ds,du),
 \end{align*}
 where 
 \begin{align*}
  \rho_n(s,u)&=\frac{\varphi_n(s,u)-\tilde{\varphi}_n(s,u)}{\sqrt{n}}\\
  &=n^{-1/2}{\bf1}(n\lambda\bar{S}(s^-)\bar{I}(s^-)\wedge\bar{S}^n(s^-)\bar{I}^n(s^-)<u\le n\lambda\bar{S}(s^-)\bar{I}(s^-)\vee\bar{S}^n(s^-)\bar{I}^n(s^-)). 
  \end{align*}
  As a consequence,
  \begin{align*} 
  \int_0^\infty \rho_n(s,u)du&=\sqrt{n}\lambda|\bar{S}(s^-)\bar{I}(s^-)-\bar{S}^n(s^-)\bar{I}^n(s^-)|\\
  &\le\lambda(|\hat{S}^n(s^-)|+|\hat{I}^n(s^-)|)\,.
  \end{align*}
Now we shall upper bound the absolute value of the integral with respect to the compensated PRM by the sum of two positive terms, the integral w.r.t. the PRM, and the integral w.r.t. the mean measure. Next in each one we upper bound  by replacing the second integral from $t$ to $\infty$ by the same integral from $0$ to $\infty$.
Finally, we shall upper bound the second moment of the sup on $t$ of 
$\hat{I}^n_1(t)-\widetilde{I}_1^n(t)$ by expectations of integrals involving the square of $|\hat{S}^n(s)|+|\hat{I}^n(s)|$, so thanks to step 1, we are done.

It remains to show that $\sup_n\E\Big[\sup_{0\le t\le T}\big(\hat{I}^n_0(t)\big)^2\Big]<\infty$.
Again we have the decomposition
 \[ \hat{I}^n_0(t)=\tilde{I}^n_0(t)+\hat{I}^n_0(t)-\tilde{I}^n_0(t)\,.\]
 The result concerning the second term follows easily from Assumption \ref{AS-SIR-2}, since 
 $\sup_{t\ge0} |\hat{I}^n_1(t)-\tilde{I}^n_1(t)|\le |\hat{I}^n(0)|$.
   It remains to consider $\tilde{I}^n_0$. Assuming for simplicity that $n\bar{I}(0)$ is an integer, we have
 \begin{align*}
  \frac{1}{\sqrt{\bar{I}(0)}}\tilde{I}^n_0(t)&=\frac{1}{\sqrt{n \bar{I}(0)}}\sum_{j=1}^{n\bar{I}(0)}\left(\bone(\eta^0_j>t)-F_0^c(t)\right)\\
  &=-\frac{1}{\sqrt{n \bar{I}(0)}}\sum_{j=1}^{n\bar{I}(0)}\left(\bone(\eta^0_j\le t)-F_0(t)\right)\\
  &=-\mathbb{F}_n(t).
  \end{align*}
  But from the well--known Dvoretsky--Kiefer--Wolfowitz inequality (with Massart's optimal constant, see \cite{Massart90}), we have
  \begin{align*}
  \P\bigg(\sup_{t\ge0}|\mathbb{F}_n(t)|>x\bigg)&\le 2\exp(-2x^2),\\
  \E\left[\sup_{t\ge0}|\mathbb{F}_n(t)|^2\right]&\le2\int_0^\infty \exp(-2x)dx=1,
  \end{align*}
  so that $\sup_n\E\Big[\sup_{0\le t\le T}\big(\tilde{I}^n_0(t)\big)^2\Big]\le \bar{I}(0) $. 
  By the representation of $\hat{I}^n(t)$ in \eqref{hatIn-rep-SIR}, we can apply Gronwall's inequality to conclude the claim. 
  The same kind of argument yields the estimate for $\hat{R}^n$.
\end{proof}

We next show that the differences of the processes $\hat{M}^n_A,  \hat{R}_1^n, \hat{I}_1^n$ with their corresponding $ \widetilde{M}^n_A,  \widetilde{R}_1^n,  \widetilde{I}_1^n$ are asymptotically negligible, stated in the next Lemma.
\begin{lemma} \label{lem-hat-tilde-diff-SIR}
Under Assumption \ref{AS-SIR-2}, 
$$
(\hat{M}^n_A - \widetilde{M}^n_A, \hat{R}_1^n - \widetilde{R}_1^n, \hat{I}_1^n - \widetilde{I}_1^n) \RA 0 \qinq D^3 \qasq n \to \infty. 
$$
\end{lemma}
\begin{proof}
It suffices to prove the convergence of each coordinate separately. We focus on the convergence $ \hat{R}_1^n - \widetilde{R}_1^n\RA 0$, since the convergence $\hat{M}^n_A - \widetilde{M}^n_A$ follows similarly, and then the convergence $ \hat{I}_1^n - \widetilde{I}_1^n\RA 0$ follows by the facts that $
   \hat{M}^n_A(t) =   \hat{I}^n_1(t)  +  \hat{R}^n_1(t)$ and    $
   \widetilde{M}^n_A(t) =   \widetilde{I}^n_1(t)  +  \widetilde{R}^n_1(t),
   $ for each $t\ge 0$.

Let $\widetilde{\Xi}^n:= \hat{R}_1^n - \widetilde{R}_1^n$. 
It is easy to see that for each $t\ge 0$,
$
\E[\widetilde\Xi^n(t)]  = 0,
$
and 
$$
\E\left[\widetilde\Xi^n(t)^2\right] = \lambda \int_0^t F(t-s) \E\left[  |\bar{S}^n(s) \bar{I}^n(s) -  \bar{S}(s) \bar{I}(s)| \right] ds \to 0 \qasq n \to \infty,
$$
where the convergence holds by Theorem \ref{thm-FLLN-SIR} and the dominated convergence theorem.  
Then it suffices to show that the sequence $\{\widetilde\Xi^n: n \ge 1\}$ is tight. 
Note that $\widetilde\Xi^n$ can be written as $\widetilde\Xi^n(t) =\widetilde \Xi_1^n(t)- \widetilde\Xi_2^n(t)$, where 
\begin{align*}
 \widetilde\Xi_1^n(t) &:=  \frac{1}{\sqrt{n}} \int_0^t \int_{0}^{t} \int_{n\lambda ( \bar{S}^n(s^-) \bar{I}^n(s^-) \wedge \bar{S}(s) \bar{I}(s))}^{n\lambda ( \bar{S}^n(s^-) \bar{I}^n(s^-) \vee \bar{S}(s) \bar{I}(s))} \text{sign}(\bar{S}^n(s^-) \bar{I}^n(s^-)- \bar{S}(s) \bar{I}(s)) M_1(ds,dz,du),  \non\\
\widetilde\Xi_2^n(t) &:= \lambda \sqrt{n} \int_0^t F(t-s)  \big(\bar{S}^n(s) \bar{I}^n(s) - \bar{S}(s) \bar{I}(s)\big) ds . 
\end{align*}
Both processes $ \widetilde\Xi_1^n(t) $ and $ \widetilde\Xi_2^n(t) $ are differences of two processes, each increasing in $t$, that is,
\begin{align}
\widetilde \Xi_1^n(t) & =  \frac{1}{\sqrt{n}} \int_0^t \int_{0}^{t} \int_{n\lambda ( \bar{S}^n(s^-) \bar{I}^n(s^-) \wedge \bar{S}(s) \bar{I}(s))}^{n\lambda ( \bar{S}^n(s^-) \bar{I}^n(s^-) \vee \bar{S}(s) \bar{I}(s))} \bone(\bar{S}^n(s^-) \bar{I}^n(s^-)- \bar{S}(s) \bar{I}(s)>0) M_1(ds,dz,du)  \non\\
 & \quad -  \frac{1}{\sqrt{n}} \int_0^t \int_{0}^{t} \int_{n\lambda ( \bar{S}^n(s^-) \bar{I}^n(s^-) \wedge \bar{S}(s) \bar{I}(s))}^{n\lambda ( \bar{S}^n(s^-) \bar{I}^n(s^-) \vee \bar{S}(s) \bar{I}(s))} \bone(\bar{S}^n(s^-) \bar{I}^n(s^-)- \bar{S}(s) \bar{I}(s) < 0)  M_1(ds,dz,du),  \non
\end{align}
and
\begin{align*}
\widetilde\Xi_2^n(t) = \lambda \sqrt{n} \int_0^t F(t-s)  \big(\bar{S}^n(s) \bar{I}^n(s) - \bar{S}(s) \bar{I}(s)\big)^+  ds 
 - \lambda \sqrt{n} \int_0^t F(t-s)  \big(\bar{S}^n(s) \bar{I}^n(s) - \bar{S}(s) \bar{I}(s)\big)^{-} ds . 
\end{align*}
Define $\Xi^n_1$ and $\Xi^n_2$ by
$$
\Xi^n_1(t) :=  \frac{1}{\sqrt{n}} \int_0^t \int_{0}^{t} \int_{n\lambda ( \bar{S}^n(s^-) \bar{I}^n(s^-) \wedge \bar{S}(s) \bar{I}(s))}^{n\lambda ( \bar{S}^n(s^-) \bar{I}^n(s^-) \vee \bar{S}(s) \bar{I}(s))}M_1(ds,dz,du), 
$$
and
$$
\Xi^n_2(t) :=\lambda \sqrt{n} \int_0^t F(t-s)  \big|\bar{S}^n(s) \bar{I}^n(s) - \bar{S}(s) \bar{I}(s)\big| ds. 
$$
Since the integrand in the integral $\widetilde \Xi_1^n(t)$ (resp. $\widetilde\Xi_2^n(t)$) is nonnegative and bounded by that in $\Xi_1^n(t)$ (resp. $\Xi_2^n(t)$),  tightness of $\Xi_1^n(t) $ and $\Xi_2^n(t) $  implies 
 tightness of the four components in the above expressions of $ \widetilde\Xi_1^n(t) $ and $ \widetilde\Xi_2^n(t)$. 
By the increasing property of $\Xi_1^n(t) $ and $\Xi_2^n(t) $, we only need to verify the following (see the Corollary on page 83 in \cite{billingsley1999convergence} or the use of \eqref{supVn-decomp} in the proof of Lemma \ref{lem-Vn-SIR-conv}): 
for any $\ep>0$, and $i=1,2$, 
\begin{equation}\label{Xi-n-tight-p}
\limsup_{n\to\infty} \frac{1}{\delta}\P\big( \big|\Xi_i^n(t+\delta) - \Xi_i^n(t) \big| \ge \ep \big) \to 0\qasq \delta \to 0.
\end{equation}
For  $ \Xi_2^n(t) $, we have
\begin{align*} 
 \Xi_2^n(t+\delta) - \Xi_2^n(t) 
&=  \lambda \int_t^{t+\delta} F(t+\delta-s) \Delta^n(s) ds + \lambda\int_0^t (F(t+\delta-s) - F(t-s)) \Delta^n(s) ds  \\
&= \Xi_{2,1}^n(t,\delta)+\Xi_{2,2}^n(t,\delta),
\end{align*}
where
\begin{equation}\label{eqn-Delta-n}
\Delta^n(s) :=  \sqrt{n} \left|\bar{S}^n(s) \bar{I}^n(s) - \bar{S}(s)\bar{I}(s)\right| =  |\hat{S}^n(s)\bar{I}^n(s) + \bar{S}(s) \hat{I}^n(s)| \le |\hat{S}^n(s)|  +  | \hat{I}^n(s)|. 
\end{equation}
We have 
\begin{align*}
\E[(\Xi_{2,1}^n(t,\delta))^2]&\le\lambda^2\delta^2\sup_{0\le t\le T}\E[(\Delta^n(t))^2]\\
&\le C\lambda^2\delta^2,
\end{align*}
thanks to Lemma \ref{lem-hatSn-t-2-bound}.  Hence
\begin{equation*}
\limsup_{n\to\infty} \frac{1}{\delta}\P\big( \big|\Xi_{2,1}^n(t,\delta) \big| \ge \ep \big)\le C\frac{\delta}{\ep^2} \to 0\qasq \delta \to 0.
\end{equation*}
Next
\begin{align*}
\Xi_{2,2}^n(t,\delta))&\le\lambda\sup_{0\le t\le T}\Delta^n(t) \int_0^t (F(t+\delta-s) - F(t-s))  ds \\
&\le\lambda\delta\sup_{0\le t\le T}\Delta^n(t),
\end{align*} 
where the second inequality follows from the same argument in \eqref{deltaVn-bound-p3} for the integral, 
and
\begin{equation}\label{Xi22}
\frac{1}{\delta}\P\big( \big|\Xi_{2,2}^n(t,\delta) \big| \ge \ep \big)\le \frac{\lambda^2\delta}{\ep^2}\E\left[\sup_{0\le t\le T}[\Delta^n(t)|^2\right].
\end{equation}
So the wished result follows from Lemma \ref{lem-hatSn-t-2-bound}, and
  \eqref{Xi-n-tight-p} holds for  $\Xi^n_2(t)$.

\smallskip

For the process $\Xi^n_1(t)$, we have
\begin{align}
& \E\big[|\Xi_1^n(t+\delta) - \Xi_1^n(t)|^2\big]  \non\\ 
&=  \E\Bigg[ \Bigg( \frac{1}{\sqrt{n}} \int_t^{t+\delta} \int_{0}^{t+\delta} \int_{n\lambda ( \bar{S}^n(s^-) \bar{I}^n(s^-) \wedge \bar{S}(s) \bar{I}(s))}^{n\lambda ( \bar{S}^n(s^-) \bar{I}^n(s^-) \vee \bar{S}(s) \bar{I}(s))}  M_1(ds,dz,du)  \non\\
& \quad + \frac{1}{\sqrt{n}} \int_0^{t} \int_{t}^{t+\delta} \int_{n\lambda ( \bar{S}^n(s^-) \bar{I}^n(s^-) \wedge \bar{S}(s) \bar{I}(s))}^{n\lambda ( \bar{S}^n(s^-) \bar{I}^n(s^-) \vee \bar{S}(s) \bar{I}(s))}  M_1(ds,dz,du)  \Bigg)^2  \Bigg]  \non\\
& \le 
2 \E\Bigg[ \Bigg( \frac{1}{\sqrt{n}} \int_t^{t+\delta} \int_{0}^{t+\delta} \int_{n\lambda ( \bar{S}^n(s^-) \bar{I}^n(s^-) \wedge \bar{S}(s) \bar{I}(s))}^{n\lambda ( \bar{S}^n(s^-) \bar{I}^n(s^-) \vee \bar{S}(s) \bar{I}(s))}  M_1(ds,dz,du)   \Bigg)^2  \Bigg]  \non\\
 &\quad + 2
  \E\Bigg[ \Bigg(  \frac{1}{\sqrt{n}} \int_0^{t} \int_{t}^{t+\delta} \int_{n\lambda ( \bar{S}^n(s^-) \bar{I}^n(s^-) \wedge \bar{S}(s) \bar{I}(s))}^{n\lambda ( \bar{S}^n(s^-) \bar{I}^n(s^-) \vee \bar{S}(s) \bar{I}(s))}  M_1(ds,dz,du)  \Bigg)^2  \Bigg]  \non\\
  &=: B^n_1 +B^n_2.  \non
\end{align}
Note that we can write
\begin{align}
&  \frac{1}{\sqrt{n}} \int_t^{t+\delta} \int_{0}^{t+\delta} \int_{n\lambda ( \bar{S}^n(s^-) \bar{I}^n(s^-) \wedge \bar{S}(s) \bar{I}(s))}^{n\lambda ( \bar{S}^n(s^-) \bar{I}^n(s^-) \vee \bar{S}(s) \bar{I}(s))}  M_1(ds,dz,du) \non\\
& =  \frac{1}{\sqrt{n}} \int_t^{t+\delta} \int_{0}^{t+\delta} \int_{n\lambda ( \bar{S}^n(s^-) \bar{I}^n(s^-) \wedge \bar{S}(s) \bar{I}(s))}^{n\lambda ( \bar{S}^n(s^-) \bar{I}^n(s^-) \vee \bar{S}(s) \bar{I}(s))}  \widetilde{M}(ds,dz,du)  + \lambda  \int_t^{t+\delta}  F(t+\delta -s) \Delta^n(s) d s. \non 
\end{align}
Thus, we have the following bound
\begin{align}\label{eqn-An-bound}
B^n_1 &\le 2 \E \left[ \left( \frac{1}{\sqrt{n}} \int_t^{t+\delta} \int_{0}^{t+\delta} \int_{n\lambda ( \bar{S}^n(s^-) \bar{I}^n(s^-) \wedge \bar{S}(s) \bar{I}(s))}^{n\lambda ( \bar{S}^n(s^-) \bar{I}^n(s^-) \vee \bar{S}(s) \bar{I}(s))}  \widetilde{M}(ds,dz,du)\right)^2\right]  \non \\
& \quad + 2 \E\left[\left(  \lambda \int_t^{t+\delta}  F(t+\delta -s) \Delta^n(s)  d s  \right)^2\right] \non\\
& \le 2 \lambda  \int_t^{t+\delta}  F(t+\delta -s) \E \left[ |\bar{S}^n(s) \bar{I}^n(s) - \bar{S}(s) \bar{I}(s) |\right] ds  + 2 \lambda^2 \delta^2 \sup_{s \in [0,T]} \E[|\Delta^n(s)|^2]. 
 \end{align}
Similarly, we have
\begin{align}\label{eqn-Bn-bound}
B^n_2 & \le 2 \E \left[ \Bigg(  \frac{1}{\sqrt{n}} \int_0^{t} \int_{t}^{t+\delta} \int_{n\lambda ( \bar{S}^n(s^-) \bar{I}^n(s^-) \wedge \bar{S}(s) \bar{I}(s))}^{n\lambda ( \bar{S}^n(s^-) \bar{I}^n(s^-) \vee \bar{S}(s) \bar{I}(s))} d \widetilde{M}(ds,dz,du)  \Bigg)^2 \right]  \non\\
& \quad + 2 \E \left[ \left(\lambda \int_0^t (F(t+\delta -s) - F(t-s)) \Delta^n(s) ds \right)^2\right]  \non\\
& \le 2  \lambda \int_0^t (F(t+\delta -s) - F(t-s)) \E \left[ |\bar{S}^n(s) \bar{I}^n(s) - \bar{S}(s) \bar{I}(s) |\right] ds + 2  \E \left[ \Xi_{2,2}^n(t,\delta)^2\right] . 
\end{align}
It is straightforward that the first terms on the right hand sides of \eqref{eqn-An-bound} and \eqref{eqn-Bn-bound} converge to zero as $n\to\infty$ since 
$\E \left[ |\bar{S}^n(s) \bar{I}^n(s) - \bar{S}(s) \bar{I}(s) |\right] \to 0$ as $n\to \infty$ by Theorem \ref{thm-FLLN-SIR}, and by the dominated convergence theorem. 
Thus,  by \eqref{Xi22}, we have shown \eqref{Xi-n-tight-p} for  $\Xi^n_1(t)$.  This completes the proof. 
\end{proof}

Let 
$$\sG^{A}_t := \sigma \left\{\widetilde{M}([0,u]\times\RR_+^2):   0 \le u \le t \right\}, \quad t \ge 0, 
$$
and 
$$\sG^{R}_t := \sigma \left\{\widetilde{M}([0,u]\times[0,u]\times\RR_+):   0 \le u \le t \right\}, \quad t \ge 0. 
$$
Then $\widetilde{M}_A^n$ is a $\{\sG^{A}_t: t \ge0\}$-martingale with quadratic variation
 $$
\langle\widetilde{M}^n_A  \rangle(t) = \lambda \int_0^t \bar{S}(s) \bar{I}(s) ds, \quad t\ge 0,
$$
and  $\widetilde{R}_1^n$ is a 
 $\{\sG^{R}_t: t \ge0\}$-martingale, with quadratic variation
$$
\langle\widetilde{R}^n_1  \rangle(t) = \lambda \int_0^t F(t-s) \bar{S}(s) \bar{I}(s) ds, \quad t\ge 0. 
$$
Note that  we do not have a martingale property for $\widetilde{I}^n_1$. 
It is important to observe that the joint process $(\widetilde{M}_A^n, \widetilde{R}^n_1)$ is not a martingale with respect to a common filtration, and therefore we cannot prove the joint convergence of them using FCLT of martingales. However, they play the role of establishing tightness of the processes $\{\hat{M}^n_A\}$, $\{\hat{I}_1^n\}$, and $\{ \hat{R}^n_1\}$. Moreover, while $\{\hat{M}^n_A\}$ is a $\sF^n_t$ martingale,  $\{ \hat{R}^n_1\}$ is not a martingale, the point being that the intensity $\lambda n \bar{S}^n(t) \bar{I}^n(t)$ is not $\sG^R$--adapted. In fact, for the sake of establishing tightness, one can exploit the martingale property of $\{\hat{M}^n_A\}$, so that the introduction of $\widetilde{M}_A^n$ is not necessary. And since the tightness of $\hat{I}^n_A$ follows from those of both $\hat{M}_A^n$ and $\hat{R}^n_1$, only $ \widetilde{R}^n_1$ really needs to be introduced for proving tightness. However, in the proof of Lemma \ref{lem-hatMARI1-conv}, we shall now need the full strength of Lemma \ref{lem-hat-tilde-diff-SIR}.

\begin{lemma}\label{lem-hatMARI1-conv}
Under Assumption \ref{AS-SIR-2}, 
\begin{equation*} 
( \hat{M}_A^n, \hat{I}^n_1, \hat{R}_1^n) \RA (\hat{M}_A,  \hat{I}_1, \hat{R}_1) \qinq D^3 \qasq n \to\infty, 
\end{equation*}
where $ (\hat{M}_A,  \hat{I}_1, \hat{R}_1)$ are given in Theorem~\ref{thm-FCLT-SIR}. 
\end{lemma}

\begin{proof} In view of Lemma \ref{lem-hat-tilde-diff-SIR}, all we need to show is that
\begin{equation}  \label{eqn-conv-wt-MA-I1-R1}
(\widetilde{M}_A^n, \widetilde{I}_1^n, \widetilde{R}_1^n)  \RA (\hat{M}_A, \hat{I}_1,\hat{R}_1)  \qinq D^3 \qasq n \to\infty\,.
\end{equation}
Exploiting the martingale property of both $\widetilde{M}^n_A$ and $ \widetilde{R}^n_1$, we can show that each of these two processes is tight in $D$. 
 In fact, by the FCLT for square-integral martingales (see, e.g., Theorem 1.4 in Chapter 7 of \cite{ethier-kurtz}),  we have $\widetilde{M}_A^n \RA \hat{M}_A$ in $D$ as $n\to \infty$, where 
$$
 \hat{M}_A(t) = B_A\left(\lambda \int_0^{t} \bar{S}(s) \bar{I}(s) ds\right), \quad t\ge 0,
$$
and $\widetilde{R}_1^n \RA \hat{R}_1$ in $D$ as $n\to \infty$, where 
$$\hat{R}_1(t) = B_R \left(\lambda \int_0^t F(t-s) \bar{S}(s) \bar{I}(s) ds\right), \quad t\ge 0,
$$
where $B_A$ and $B_R$ are a standard Brownian motions. 
Note that we do not obtain joint convergence as discussed above, which we do not need for this lemma. 
 It is then clear that the difference $\widetilde{I}^n_1(t)  =   \widetilde{M}^n_A(t) -\widetilde{R}^n_1(t)$ is also tight.
 Thus, by Lemma \ref{lem-hat-tilde-diff-SIR},  the sequences $\{\hat{M}^n_A\}$, $\{\hat{I}_1^n\}$, and $\{ \hat{R}^n_1\}$ are tight. 
 Therefore, to prove \eqref{eqn-conv-wt-MA-I1-R1}, 
it remains to show  (i) convergence of finite dimensional distributions of $(\widetilde{M}_A^n, \widetilde{I}_1^n, \widetilde{R}_1^n)$ to those of $(\hat{M}_A, \hat{I}_1,\hat{R}_1)$ and (ii) the limits $(\hat{M}_A, \hat{I}_1,\hat{R}_1)$ are continuous. 

To prove the convergence of finite dimensional distributions of $(\widetilde{M}_A^n, \widetilde{I}_1^n, \widetilde{R}_1^n)$ to those of $(\hat{M}_A, \hat{I}_1,\hat{R}_1)$, by the independence of the restrictions of a PRM to disjoint subsets, 
 it suffices to show that for $0 \le t' \le t$ and $0 \le a \le b < \infty$, 
\begin{align}\label{eqn-char-SIR}
& \lim_{n\to\infty} \E\left[ \exp \left(i \frac{\vartheta}{\sqrt{n}} \int_{t'}^t \int_a^b \int_0^\infty  \widetilde{\varphi}_n(s,u) d \widetilde{M}(ds, dz, du)\right)\right]  \non\\
& = \exp \left( -\frac{\vartheta^2}{2} \lambda \int_{t'}^t (F(b-s) - F(a-s)) \bar{S}(s) \bar{I}(s) ds \right) . 
\end{align}
Recall that for a compensated PRM $\bar{N}$ with mean measure $\nu$ and a deterministic function $\phi$, we have 
\begin{equation}\label{eqn-PRM-char}
\E\left[\exp(i\vartheta \bar{N}(\phi))\right] = e^{-i \vartheta \nu(\phi)} \exp\left( \nu(e^{i\vartheta \phi} -1) \right),
\end{equation}
where $\nu(\phi) := \int \phi d \nu$. 
As a consequence, the left hand side of \eqref{eqn-char-SIR} is equal to 
\begin{align}
& \exp\left( - i \frac{\vartheta}{\sqrt{n}} \int_{t'}^t (F(b-s) - F(a-s)) \lambda n \bar{S}(s) \bar{I}(s) d s  \right) \non\\
& \times \exp\left( (e^{i \vartheta/\sqrt{n}}-1) \int_{t'}^t  (F(b-s) - F(a-s)) \lambda n \bar{S}(s) \bar{I}(s) d s\right). \non
\end{align}
Then the claim  \eqref{eqn-char-SIR} is immediate by applying Taylor expansion. 

Given the consistent finite dimensional distributions of $\hat{R}_1$, to show that the limit process $\hat{R}_1$ has a continuous version, it suffices to show that
\begin{equation} \label{E-hatR1-4th}
\E\left[( \hat{R}_1(t+\delta) -  \hat{R}_1(t)))^4 \right] \le c \delta^{2}. 
\end{equation}
This is immediate since as a consequence of  \eqref{eqn-char-SIR}, 
\begin{align*}
&\E\left[( \hat{R}_1(t+\delta) -  \hat{R}_1(t)))^4 \right] = 3 \left(\E\left[( \hat{R}_1(t+\delta) -  \hat{R}_1(t)))^2 \right]\right)^2 \\
& = 3 \left( \lambda\int_t^{t+\delta} F(t+\delta -s) \bar{S}(s) \bar{I}(s) ds + \lambda \int_0^t (F(t+\delta -s) - F(t-s)) \bar{S}(s) \bar{I}(s) ds  \right)^2  \\
& \le 6 \lambda^2 \delta^2 + 6 \lambda^2 \left( \int_0^t (F(t+\delta -s) - F(t-s)) \bar{S}(s) \bar{I}(s) ds\right)^2\\
&\le 6 \lambda^2 \delta^2 + 6 \lambda^2\left( \int_0^t (F(t+\delta -s) - F(t-s))  ds\right)^2\\
&\le 12\delta^2.  
\end{align*}
 Here the two equalities are by the Gaussian property of the limit $\hat{R}_1$ and direct calculations from its covariance function. The first inequality follows from the simple bound $(x+y)^2 \le 2x^2 + 2y^2$ and the first term term bounded by $\lambda \delta$. Next $\bar{S}(s) \bar{I}(s)\le1$, and the remaining bound follows from the computation leading to \eqref{deltaVn-bound-p3}.
The same property holds analogously for the processes $\hat{M}_A$ and $\hat{I}_1$. This completes the proof. \end{proof}

\smallskip
\noindent \emph{Completing the proof of Theorem \ref{thm-FCLT-SIR}.}
By Lemmas \ref{lem-hatInitials-conv-SIR} and \ref{lem-hatMARI1-conv}, 
 we first obtain the joint convergence 
  \begin{equation*}
 (\hat{I}^n(0) F_0^c(\cdot), \hat{I}^n(0) F_0(\cdot),  \hat{I}^n_0, \hat{R}^n_0,   \hat{M}_A^n,  \hat{I}^n_1, \hat{R}^n_1)\RA  \left(\hat{I}(0) F^c_0(\cdot), \hat{I}(0) F_0(\cdot), \hat{I}_0, \hat{R}_0, \hat{M}_A, \hat{I}_1, \hat{R}_1 \right) 
  \end{equation*}
 in $D^7$ as $n\to\infty$. 
 Since the limit processes $\hat{I}_0, \hat{R}_0, \hat{M}_A, \hat{I}_1, \hat{R}_1$ are continuous, we have the convergence: 
 $$
  (-\hat{M}_A^n, \hat{I}^n(0) F_0^c(\cdot)  + \hat{I}_0^n  + \hat{I}_1^n,  \hat{I}^n(0) F_0(\cdot)  + \hat{R}_0^n  + \hat{R}_1^n)   \RA ( -\hat{M}_A, \hat{I}(0) F_0^c(\cdot)  + \hat{I}_0 + \hat{I}_1,   \hat{I}(0) F_0^c(\cdot)  + \hat{R}_0 + \hat{R}_1),  
 $$
 in $D^3$ as $n\to\infty$. 
It follows from \eqref{hatIn-rep-SIR}, \eqref{hatRn-rep-SIR}, Theorem \ref{thm-FLLN-SIR}, Lemma \ref{lem-hatInitials-conv-SIR}, \ref{lem-hatSn-t-2-bound} and \ref{lem-hatMARI1-conv} that $(\hat{I}^n, \hat{R}^n)$ is tight in $D^2$, and any limit of a converging subsequence satisfies \eqref{SIR-Ihat} and \eqref{SIR-Rhat}, where we may replace $\hat{S}$ by $-\hat{I}-\hat{R}$, since $\hat{S}^n=-\hat{I}^n - \hat{R}^n$ for all $n$. From Lemma \ref{lem-Gamma-cont}, this characterizes uniquely  the limit, hence the whole sequence converges, and finally \eqref{eqn-FCLT-conv-SIR}, \eqref{SIR-Shat} follow readily from the above, and again the fact that $\hat{S}^n=-\hat{I}^n - \hat{R}^n$ for all $n$.
\hfill $\Box$

\section{Proof of the FLLN for the SEIR model} \label{sec-SEIR-FLLN-proof}

In this section we prove Theorem~\ref{thm-FLLN-SEIR}. 
The expressions and claims in \eqref{barAn-rep}--\eqref{barSn-conv1} hold by the same arguments, which we assume from now on. 
By slightly modifying the argument as for the process $\bar{I}^n$ in the SIR model, we  obtain that  
\begin{equation}
\bar{E}^n(\cdot)  \RA  \bar{E}(0) G_0^c(\cdot) +   \int_0^{\cdot}G^c(\cdot-s) d\bar{A}(s) \qinq D \non
\end{equation}
in probability as $n \to\infty$.
Recall $I^n(t)$ in \eqref{SEIR-In}. Define 
\begin{align}
\bar{I}^n_{0,1}(t) &:= \frac{1}{n}\sum_{j=1}^{I^n(0)} \bone(\eta_j^0 > t), \quad
\bar{I}^n_{0,2}(t) := \frac{1}{n} \sum_{j=1}^{E^n(0)} \bone (\xi_j^0 \le t) \bone(\xi_j^0+ \eta_j >t ), \non\\
\bar{I}^n_{1}(t) &:= \frac{1}{n}\sum_{i=1}^{A^n(t)} \bone(\tau_i^n+ \xi_i \le t) \bone(\tau_i^n+ \xi_i + \eta_i >t). \non 
\end{align}
By the FLLN of empirical processes, and by Assumption \ref{AS-SEIR-1}, 
 we have  
 \begin{align}\label{barIn-0102-conv}
 (\bar{I}^n_{0,1}, \bar{I}^n_{0,2}) \to (\bar{I}_{0,1}, \bar{I}_{0,2}) \qinq D^2
 \end{align}
 in probability as $n\to\infty$, where $\bar{I}_{0,1}:=\bar{I}(0) G^c_0(\cdot)$ and $\bar{I}_{0,2}:= \bar{E}(0) \Psi_0(\cdot).$

For the study of the process $\bar{I}^n_1$, we first consider 
\begin{align}
\breve{I}^n_1(t) := \E[\bar{I}^n_1(t) |\sF^n_t]  
&=  \frac{1}{n}\sum_{i=1}^{A^n(t)} \Psi(t-\tau_i^n) =      \int_0^t \Psi(t-s) d \bar{A}^n(s)  =  \bar{A}^n(t) - \int_0^t \bar{A}^n(s) d \Psi(t-s). \non 
\end{align}

Applying the continuous mapping theorem  to the map $x\in D \to  x-\int_0^{\cdot}  x(s)d\Psi(\cdot -s) \in D$, we obtain  
\begin{equation}\label{tildeIn-1-conv}
\breve{I}^n_1 \to \bar{I}_1  \qinq D
\end{equation}
 in probability as $n\to \infty$, where
\begin{equation}\label{barI-1-def}
 \bar{I}_1(t) :=\bar{A}(t) - \int_0^t \bar{A}(s) d \Psi(t-s) =   \int_0^t  \Psi(t-s) d \bar{A}(s),\quad t \ge 0.
\end{equation}
We now consider the difference
$$
V^n(t):= \bar{I}^n_1(t)-\breve{I}^n_1(t) = \frac{1}{n}\sum_{i=1}^{A^n(t)} \kappa^n_i(t),
$$
where
\begin{align}
\kappa^n_i(t) 
& =  \bone(\tau_i^n+ \xi_i \le t) \bone(\tau_i^n+ \xi_i + \eta_i >t) - \Psi(t-\tau^n_i).  \non
\end{align}
We next show the following lemma.
\begin{lemma}
For any $\ep>0$, 
\begin{equation}\label{Vn-dffi-0-SEIR}
\P\left(\sup_{t\in [0,T]} |V^n(t)| > \ep \right) \to 0 \qasq n \to \infty. 
\end{equation}
\end{lemma}
\begin{proof}
We partition $[0,T]$ into intervals of length $\delta>0$, and have the bound for $\sup_{t\in [0,T]}|V^n(t)|$ as in \eqref{supVn-decomp}. 

First, we have
\begin{align*}
\E[\kappa^n_i(t) |\sF^n_t] = 0, \quad \forall \,  i;\quad
 \E[\kappa^n_i(t)\kappa^n_j(t)  |\sF^n_t] = 0, \quad \forall \, i\neq j.
\end{align*}
Thus
\begin{align}
&\E[V^n(t)^2|\sF^n_t] = \frac{1}{n^2}\sum_{i=1}^{A^n(t)} \E[\kappa^n_i(t)^2 |\sF^n_t] \non\\
&=  \frac{1}{n^2}\sum_{i=1}^{A^n(t)} \Psi(t-\tau^n_i) (1- \Psi(t-\tau^n_i))  = \frac{1}{n} \int_0^t \Psi(t-s) \left( 1-\Psi(t-s)\right) d \bar{A}^n(s)  \non\\
& = \frac{1}{n^{3/2}} \int_0^t \Psi(t-s) \left( 1-\Psi(t-s)\right) d\hat{M}_A^n(s)  +  \frac{1}{n} \int_0^t \Psi(t-s) \left( 1-\Psi(t-s)\right) d \bar{\Lambda}^n(s) \non\\
& \le \frac{1}{n^{3/2}} \int_0^t \Psi(t-s) \left( 1-\Psi(t-s)\right) d \hat{M}_A^n(s) + \frac{\lambda t}{n}, \non
\end{align}
where  the inequality follows from \eqref{barSIn-bound} and  \eqref{barLambda-n-Lip}. Thus
\begin{equation}\label{Vn-t-conv-SEIR}
\E\big[V^n(t)^2\big]\le \frac{\lambda t}{n},\quad \P(|V^n(t)|>\ep )  \le \frac{\lambda t}{\ep^2n}. 
\end{equation}

Next we have 
\begin{align}
&  |V^n(t+u) - V^n(u)|   \non \\
&=  \left|  \frac{1}{n}\sum_{i=1}^{A^n(t+u)} \kappa^n_i(t+u) -   \frac{1}{n}\sum_{i=1}^{A^n(t)} \kappa^n_i(t)  \right| \non\\
& =  \left| \frac{1}{n}\sum_{i=1}^{A^n(t)} (\kappa^n_i(t+u) - \kappa^n_i(t))  +   \frac{1}{n}\sum_{i=A^n(t)}^{{A^n(t+u)}} \kappa^n_i(t+u) \right| \non\\
&\le    \left| \frac{1}{n}\sum_{i=1}^{A^n(t)} ( \bone(\tau_i^n+ \xi_i \le t+u) \bone(\tau_i^n+ \xi_i + \eta_i >t+u)  -  \bone(\tau_i^n+ \xi_i \le t) \bone(\tau_i^n+ \xi_i + \eta_i >t) ) \right| \non\\
& \quad + \left|  \int_0^{t} \left( \Psi(t+u-s) - \Psi(t-s) \right)   d\bar{A}^n(s) \right|   +   \frac{1}{n}\sum_{i=A^n(t)}^{{A^n(t+u)}}| \kappa^n_i(t+u)|  \non \\
& \le  \frac{1}{n}\sum_{i=1}^{A^n(t)}  \bone(\tau_i^n+ \xi_i \le t+u) ( \bone(\tau_i^n+ \xi_i + \eta_i >t) -\bone(\tau_i^n+ \xi_i + \eta_i >t+u) )  \non\\
& \quad +  \frac{1}{n}\sum_{i=1}^{A^n(t)} ( \bone(\tau_i^n+ \xi_i \le t+u) -  \bone(\tau_i^n+ \xi_i \le t))  \bone(\tau_i^n+ \xi_i + \eta_i >t)   \non\\
& \quad + \int_0^t \left( \int_0^{t-s+u} (F^c(t-s-v|v) - F^c(t+u-s-v|v))d G(v) \right) d \bar{A}^n(s)  \non\\
& \quad +  \int_0^{t} \left( \int_{t-s}^{t-s+u} F^c(t-s-v|v)d G(v) \right) d \bar{A}^n(s)   +   \frac{1}{n}\sum_{i=A^n(t)}^{{A^n(t+u)}}| \kappa^n_i(t+u)|. \non
\end{align}
Observing that the first four terms on the right hand side are all increasing in $u$, and that $| \kappa^n_i(t)| \le 1$ for all $t, i, n$, we obtain that 
\begin{align} \label{Vn-conv-SEIR-p1}
& \sup_{u\in [0,\delta]} |V^n(t+u) - V^n(u)|   \non \\
& \le  \frac{1}{n}\sum_{i=1}^{A^n(t)}  \bone(\tau_i^n+ \xi_i \le t+\delta) ( \bone(\tau_i^n+ \xi_i + \eta_i >t) -\bone(\tau_i^n+ \xi_i + \eta_i >t+\delta) )  \non\\
& \quad +  \frac{1}{n}\sum_{i=1}^{A^n(t)} ( \bone(\tau_i^n+ \xi_i \le t+\delta) -  \bone(\tau_i^n+ \xi_i \le t))  \bone(\tau_i^n+ \xi_i + \eta_i >t)   \non\\
& \quad + \int_0^t \left( \int_0^{t-s+\delta} (F^c(t-s-v|v) - F^c(t+\delta-s-v|v))d G(v) \right) d \bar{A}^n(s)  \non\\
& \quad +  \int_0^{t} \left( \int_{t-s}^{t-s+\delta} F^c(t-s-v|v)d G(v) \right) d \bar{A}^n(s)    + ( \bar{A}^n(t+\delta) -\bar{A}^n(t)). 
\end{align}
Thus, for any $\ep>0$, 
\begin{align} \label{Vn-conv-SEIR-p2}
&\P \left( \sup_{u\in [0,\delta]} |V^n(t+u) - V^n(u)|  > \ep \right)  \\
& \le \P \left(\frac{1}{n}\sum_{i=1}^{A^n(t)}  \bone(\tau_i^n+ \xi_i \le t+\delta) \bone(t< \tau_i^n+ \xi_i + \eta_i \le t+\delta)  >\ep/5\right)  \non\\
& \quad + \P \left( \frac{1}{n}\sum_{i=1}^{A^n(t)}  \bone(t < \tau_i^n+ \xi_i \le t+\delta) \bone(\tau_i^n+ \xi_i + \eta_i >t)  >\ep/5\right) \non\\
& \quad +\P \left( \int_0^t \left( \int_0^{t-s+\delta} (F^c(t-s-v|v) - F^c(t+\delta-s-v|v))d G(v) \right) d \bar{A}^n(s) >\ep/5\right) \non\\
& \quad +  \P \left(\int_0^{t} \left( \int_{t-s}^{t-s+\delta} F^c(t-s-v|v)d G(v) \right) d \bar{A}^n(s) >\ep/5\right)  +\P \left( ( \bar{A}^n(t+\delta) -\bar{A}^n(t))>\ep/5\right). \non
\end{align}
We need the following definition to treat the first two terms on the right hand side of \eqref{Vn-conv-SEIR-p2}. 
\begin{definition} \label{def-PRM-SEIR-1}
 Define a PRM  $M(ds,dy,dz,du)$ on $[0,T]\times \RR_+ \times \RR_+ \times  \RR_+$ with  mean measure $\nu(ds, dy, dz, du)=ds H(dy, dz) du $. Denote the compensated PRM by $\overline{M}(ds,dy,dz,du)$. 
 \end{definition}
 
 For the first term on the right hand side of \eqref{Vn-conv-SEIR-p2}, we have
 \begin{align}\label{Vn-conv-SEIR-p3}
 & \E\left[ \left(\frac{1}{n}\sum_{i=1}^{A^n(t)}  \bone(\tau_i^n+ \xi_i \le t+\delta) \bone(t<\tau_i^n+ \xi_i + \eta_i \le t+\delta) \right)^2 \right]  \non\\
 & = \E\left[ \left(\frac{1}{n}\int_0^t  \int_0^{t+\delta-s} \int_{t-s-y}^{t+\delta-s-y} \int_0^{n\lambda \bar{S}^n(s^-) \bar{I}^n(s^-)}  M(ds, dy, dz, du)\right)^2 \right]  \non\\
 &\le 2\E\left[ \left(\frac{1}{n}\int_0^t  \int_0^{t+\delta-s} \int_{t-s-y}^{t+\delta-s-y} \int_0^{n\lambda \bar{S}^n(s^-) \bar{I}^n(s^-)}  \overline{M}(ds, dy, dz, du)\right)^2 \right]  \non\\
 & \quad + 2 \E\left[ \left( \int_0^t \left( \int_0^{t-s+\delta} (F^c(t-s-v|v) - F^c(t+\delta-s-v|v))d G(v) \right) d \bar{\Lambda}^n(s)\right)^2 \right]  \non\\
 & = \frac{2}{n}\E\left[  \int_0^t \left( \int_0^{t-s+\delta} (F^c(t-s-v|v) - F^c(t+\delta-s-v|v))d G(v) \right) d \bar{\Lambda}^n(s) \right]  \non\\
 & \quad + 2 \E\left[ \left( \int_0^t \left( \int_0^{t-s+\delta} (F^c(t-s-v|v) - F^c(t+\delta-s-v|v))d G(v) \right) d \bar{\Lambda}^n(s)\right)^2 \right]  \non\\
 & \le \frac{2}{n} \lambda \int_0^t \left( \int_0^{t-s+\delta} (F^c(t-s-v|v) - F^c(t+\delta-s-v|v))d G(v) \right) d s \non\\
 & \quad + 2  \left( \lambda \int_0^t \left( \int_0^{t-s+\delta} (F^c(t-s-v|v) - F^c(t+\delta-s-v|v))d G(v) \right) d s\right)^2  \non\\
 & = \frac{2}{n} \lambda \int_0^{t+\delta}  \left( \int_0^{t-v+\delta}   (F^c(t-s-v|v) - F^c(t+\delta-s-v|v))  ds \right) d G(v) \non\\
 & \quad + 2  \left( \lambda  \int_0^{t+\delta}  \left( \int_0^{t-v+\delta}   (F^c(t-s-v|v) - F^c(t+\delta-s-v|v))  ds \right) d G(v) \right)^2. 
 \end{align}
 Here the second inequality uses  \eqref{barLambda-n-Lip}. 
 The first term on the right hand side of \eqref{Vn-conv-SEIR-p3} converges to zero as $n\to\infty$. 
 It is easily seen, by the same argument as that leading to \eqref{deltaVn-bound-p3}, that
 \[ \int_0^{t-v+\delta}(F^c(t-s-v|v)-F^c(t+\delta-s-v|v))ds\le \delta\,.\]
 Consequently, 
 \begin{equation}\label{Vn-conv-SEIR-p4}
 \frac{1}{\delta}\left(   \int_0^{t+\delta}  \left( \int_0^{t-v+\delta}   (F^c(t-s-v|v) - F^c(t+\delta-s-v|v))  ds \right) d G(v) \right)^2\le\delta\to 0 \qasq \delta \to 0. 
 \end{equation}

 Similarly, for the second term on the right hand side of \eqref{Vn-conv-SEIR-p2}, we have
 \begin{align}\label{Vn-conv-SEIR-p5}
 & \E\left[ \left( \frac{1}{n}\sum_{i=1}^{A^n(t)} \bone(t <\tau_i^n+ \xi_i \le t+\delta)   \bone(\tau_i^n+ \xi_i + \eta_i >t) \right)^2 \right] \non\\
 & = \E\left[ \left( \frac{1}{n} \int_0^t \int_{t-s}^{t+\delta-s} \int_{t-s-y}^\infty \int_0^{n\lambda \bar{S}^n(s^-) \bar{I}^n(s^-)}  M(ds, dy, dz, du)  \right)^2 \right] \non\\
 & \le 2 \E\left[ \left( \frac{1}{n} \int_0^t \int_{t-s}^{t+\delta-s} \int_{t-s-y}^\infty \int_0^{n\lambda \bar{S}^n(s^-) \bar{I}^n(s^-)}  \overline{M}(ds, dy, dz, du)  \right)^2 \right] \non\\
 & \quad + 2  \E\left[ \left( \int_0^{t} \left( \int_{t-s}^{t-s+\delta} F^c(t-s-v|v)d G(v) \right) d \bar{\Lambda}^n(s) \right)^2 \right] \non\\
 & = \frac{2}{n}  \E\left[ \int_0^{t} \left( \int_{t-s}^{t-s+\delta} F^c(t-s-v|v)d G(v) \right) d \bar{\Lambda}^n(s) \right] \non\\
 & \quad + 2  \E\left[ \left( \int_0^{t} \left( \int_{t-s}^{t-s+\delta} F^c(t-s-v|v)d G(v) \right) d \bar{\Lambda}^n(s) \right)^2 \right] \non\\
 & \le  \frac{2}{n} \lambda \int_0^{t} \left( \int_{t-s}^{t-s+\delta} F^c(t-s-v|v)d G(v) \right) d s \non\\
 & \quad + 2  \left( \lambda \int_0^{t} \left( \int_{t-s}^{t-s+\delta} F^c(t-s-v|v)d G(v) \right) d s \right)^2. 
 \end{align}
 Again, here the second inequality uses  \eqref{barLambda-n-Lip}. The first term on the right hand side of \eqref{Vn-conv-SEIR-p5} converges to zero as $n\to\infty$. 
We have
 \begin{align}\label{Vn-conv-SEIR-p6}
 \frac{1}{\delta} \left( \lambda \int_0^{t} \left( \int_{t-s}^{t-s+\delta} F^c(t-s-v|v)d G(v) \right) d s \right)^2
 &\le\frac{\lambda^2}{\delta}\left(\int_0^t (G(t-s+\delta)-G(t-s))ds\right)^2 \non \\
 &\le\lambda^2\delta \to 0 \qasq \delta \to 0. 
 \end{align}
 where the second inequality follows from the argument used for establishing \eqref{deltaVn-bound-p3}.
For the third term on the right hand side of \eqref{Vn-conv-SEIR-p2}, by \eqref{barAn-rep}, we have
\begin{align}\label{Vn-conv-SEIR-p7}
& \E\left[ \left( \int_0^t \left( \int_0^{t-s+\delta} (F^c(t-s-v|v) - F^c(t+\delta-s-v|v))d G(v) \right) d \bar{A}^n(s)\right)^2 \right]  \non\\
& \le 2 \E\left[ \left( \frac{1}{\sqrt{n}}\int_0^t \left( \int_0^{t-s+\delta} (F^c(t-s-v|v) - F^c(t+\delta-s-v|v))d G(v) \right) d \hat{M}_A^n(s)\right)^2 \right]  \non\\
& \quad + 2 \E\left[ \left( \int_0^t \left( \int_0^{t-s+\delta} (F^c(t-s-v|v) - F^c(t+\delta-s-v|v))d G(v) \right) d \bar{\Lambda}^n(s) \right)^2 \right]. 
\end{align}
Then by \eqref{hatMnA-conv} the first term converges to zero as $n\to\infty$, and the second term can be treated similarly as the second term in \eqref{Vn-conv-SEIR-p3}. 
The fourth term in  \eqref{Vn-conv-SEIR-p2} can be treated similarly. The last term in  \eqref{Vn-conv-SEIR-p2} 
 is the same as in \eqref{deltaVn-bound-p4}. 
 Therefore, by combining the above arguments and  \eqref{Vn-conv-SEIR-p2}-- \eqref{Vn-conv-SEIR-p7}, we obtain 
  \begin{equation}
\lim_{\delta \to 0} \limsup_{n\to\infty}  \,   \left[\frac{T}{\delta}\right] \,\sup_{0\le t\le T}   \P\left(\sup_{u\in [0,\delta]}|V^n(t+u) - V^n(t)|  \ge \ep \right) =0.  \non
\end{equation}
Then by \eqref{supVn-decomp-ineq} and \eqref{Vn-t-conv-SEIR}, we conclude that   \eqref{Vn-dffi-0-SEIR} holds. 
 \end{proof}

 By \eqref{tildeIn-1-conv} and  \eqref{Vn-dffi-0-SEIR}, we have 
$
\bar{I}^n_1 \to \bar{I}_1$ in $ D$ in probability as $ n\to\infty. 
$
Combining this with the convergences of $(\bar{I}^n_{0,1},\bar{I}^n_{0,2})$ in \eqref{barIn-0102-conv}, by independence of $(\bar{I}^n_{0,1},\bar{I}^n_{0,2})$ and $\bar{I}^n_1$, we have 
$
\bar{I}^n = \bar{I}^n_{0,1} +  \bar{I}^n_{0,2} +  \bar{I}^n_{1} \to  \bar{I} =  \bar{I}_{0,1} +  \bar{I}_{0,2} +  \bar{I}_{1}$ in $ D$ in probability as $n\to\infty. 
$

Similar to the SIR model, we can show 
 the joint convergence $(\bar{S}^n, \breve{I}^n) \to (\bar{S}, \bar{I})$ in $D^2$ in probability as $ n\to\infty$. 
Thus, using a similar argument as in the SIR model, we have shown that the limits $(\bar{S},\bar{I})$ of  $(\bar{S}^n,\bar{I}^n)$ satisfy the integral equations \eqref{SEIR-barS} and \eqref{SEIR-barI}. Similarly to the SIR model, these two equations have a unique solution. Once the solutions of  $(\bar{S},\bar{I})$ are uniquely determined, the other limits $\bar{A}, \bar{E}, \bar{L}, \bar{R}$ are also uniquely determined by the corresponding integral equations. This proves the convergence in probability.  Therefore the proof of Theorem \ref{thm-FLLN-SEIR} is complete.

\section{Proof of the FCLT for the SEIR model}  \label{sec-SEIR-FCLT-proof}

In this section we prove Theorem~\ref{thm-FCLT-SEIR}, for the 
diffusion-scaled processes $(\hat{S}^n, \hat{E}^n,  \hat{I}^n, \hat{R}^n)$ defined in \eqref{SEIR-diff-def}.
Similarly to the SIR model, under Assumption~\ref{AS-SEIR-2}, we have 
 $(\bar{I}^n(0), \bar{E}^n(0)) \RA (\bar{I}(0),  \bar{E}(0))\in \RR^2_{+}$  as $n\to\infty$, and thus the FLLN Theorem~\ref{thm-FLLN-SEIR} holds, which will be taken as given in the proof below. Recall the martingale  $\hat{M}_A^n$ defined in \eqref{hatMn-def}. 
 
 We have the following representation of the diffusion-scaled processes. 
We have the same representation of  $\hat{S}^n$ in \eqref{hatSn-rep-SIR} for the SIR model. 
For the ease of exposition, we repeat the following expression for the process $\hat{S}^n$: 
  \begin{align}
  \hat{S}^n(t) 
  & = -\hat{I}^n(0)-\hat{M}_A^n(t) -   \lambda  \int_0^t \left( \hat{S}^n(s) \bar{I}^n(s) + \bar{S}(s) \hat{I}^n(s) \right) ds. \non
 \end{align} 
For the process $\hat{E}^n$, 
\begin{align}
\hat{E}^n(t)  &=  \hat{E}^n(0) G_0^c(t)  + \hat{E}_0^n(t)  + \hat{E}_1^n(t)  +  \lambda  \int_0^t G^c(t-s) \left( \hat{S}^n(s) \bar{I}^n(s) + \bar{S}(s) \hat{I}^n(s) \right) ds, \non
\end{align}
where 
\begin{align*}
\hat{E}_0^n(t) &:= \frac{1}{\sqrt{n}} \sum_{j=1}^{n\bar{E}^n(0)} \big( \bone(\xi^0_j > t) -  G_0^c(t)\big), \\
\hat{E}_1^n(t) &:=\frac{1}{\sqrt{n}}   \sum_{i=1}^{n\bar{A}^n(t)} \bone(\tau^n_i + \xi_i >t) - \sqrt{n} \lambda \int_0^t G^c(t-s) \bar{S}^n(s) \bar{I}^n(s) ds.  \end{align*}
For the process $\hat{I}^n$, 
 \begin{align}\label{hatIn-SEIR-rep}
 \hat{I}^n(t) & = \hat{I}^n(0) F^c_0(t) +  \hat{E}^n(0) \Psi_0(t) +  \hat{I}^n_{0,1}(t)  + \hat{I}^n_{0,2}(t) +  \hat{I}^n_{1}(t) \non\\
 & \qquad    + \lambda \int_0^t \Psi(t-s)
\left ( \hat{S}^n(s) \bar{I}^n(s) + \bar{S}(s) \hat{I}^n(s) \right) ds, 
 \end{align}
 where 
 \begin{equation}
  \hat{I}^n_{0,1}(t) = \frac{1}{\sqrt{n}} \sum_{j=1}^{I^n(0)} \big( \bone(\eta_j^0 > t)- F_0^c(t)\big), \quad
     \hat{I}^n_{0,2}(t) = \frac{1}{\sqrt{n}}  \sum_{j=1}^{E^n(0)} \left(\bone (\xi_j^0 \le t) \bone(\xi_j^0+ \eta_j >t ) -\Psi_0(t) \right), \non
 \end{equation}
 and
 \begin{align}
   \hat{I}^n_{1}(t)   &= \frac{1}{\sqrt{n}}  \sum_{i=1}^{A^n(t)}  \bone(\tau_i^n+ \xi_i \le t) \bone(\tau_i^n+ \xi_i + \eta_i >t)  -  \lambda \sqrt{n} \int_0^t \Psi(t-s)
    \bar{S}^n(s) \bar{I}^n(s) ds. \non
 \end{align}
 For the process $\hat{R}^n$, 
 \begin{align}
 \hat{R}^n(t) & = \hat{I}^n(0) F_0(t) +  \hat{E}^n(0) \Phi_0(t) +  \hat{R}^n_{0,1}(t)  + \hat{R}^n_{0,2}(t) +  \hat{R}^n_{1}(t)  \non\\
 & \qquad    + \lambda \int_0^t \Phi(t-s)
\left ( \hat{S}^n(s) \bar{I}^n(s) + \bar{S}(s) \hat{I}^n(s) \right) ds, \non
 \end{align}
 where 
 \begin{equation}
  \hat{R}^n_{0,1}(t) = \frac{1}{\sqrt{n}} \sum_{j=1}^{I^n(0)} \big( \bone(\eta_j^0 \le t)- F_0(t)\big),\quad   \hat{R}^n_{0,2}(t) = \frac{1}{\sqrt{n}}  \sum_{j=1}^{E^n(0)} \left( \bone(\xi_j^0+ \eta_j \le t ) - \Phi_0(t)  \right), \non
 \end{equation}
 and
 \begin{align}
   \hat{R}^n_{1}(t)   &= \frac{1}{\sqrt{n}}  \sum_{i=1}^{A^n(t)}   \bone(\tau_i^n+ \xi_i + \eta_i \le t)  -  \lambda \sqrt{n} \int_0^t  \Phi(t-s)
   \bar{S}^n(s) \bar{I}^n(s) ds. \non 
 \end{align}

To facilitate the proof, we also define the process $\hat{L}^n$ (recall that $L^n(t)=I^n(t)+R^n(t)-I^n(0)$) :
\begin{align}
\hat{L}^n(t) &:= \sqrt{n} \left( \bar{L}^n(t) - \bar{L}(t) \right) = \sqrt{n} \left( \bar{L}^n(t) - \left(  \bar{E}(0) G_0(t) +  \lambda \int_0^t G(t-s) \bar{S}(s) \bar{I}(s) ds\right)\right). \non 
\end{align}
It has the following representation: 
\begin{align}
\hat{L}^n(t)  &=  \hat{E}^n(0) G_0(t)  + \hat{L}_0^n(t)  + \hat{L}_1^n(t)   +  \lambda  \int_0^t G(t-s) \left( \hat{S}^n(s) \bar{I}^n(s) + \bar{S}(s) \hat{I}^n(s) \right) ds, \non
\end{align}
where 
\begin{align}
\hat{L}_0^n(t) &:= \frac{1}{\sqrt{n}} \sum_{j=1}^{n\bar{E}^n(0)} \big( \bone(\xi^0_j \le t) -  G_0(t)\big), \non\\
\hat{L}_1^n(t) &:=\frac{1}{\sqrt{n}}   \sum_{i=1}^{n\bar{A}^n(t)} \bone(\tau^n_i + \xi_i \le t) - \sqrt{n} \lambda \int_0^t G(t-s) \bar{S}^n(s) \bar{I}^n(s) ds.\non
\end{align}

 We have the following joint convergence for the initial quantities similar to Lemma \ref{lem-hatInitials-conv-SIR} for the SIR model. Its proof is omitted for brevity. 
 \begin{lemma}\label{lem-hatInitials-conv-SEIR}
 Under Assumption \ref{AS-SEIR-2}, 
  \begin{align}
&  \Big(\hat{E}^n(0) G_0^c(\cdot),  \hat{E}^n_0, \hat{E}^n(0) G_0(\cdot),  \hat{L}^n_0,  \hat{I}^n(0) F^c_0(\cdot), \hat{E}^n(0) \Psi_0(\cdot),    \hat{I}^n_{0,1}, \hat{I}^n_{0,2},  \hat{I}^n(0) F_0(\cdot), \hat{E}^n(0) \Phi_0(\cdot), 
 \hat{R}^n_{0,1}, \hat{R}^n_{0,2} \Big) \non\\
 & 
 \RA
 \Big(\hat{E}(0) G_0^c(\cdot),  \hat{E}_0, \hat{E}(0) G_0(\cdot),  \hat{L}_0,  \hat{I}(0) F^c_0(\cdot), \hat{E}(0)\Psi_0(\cdot),  \hat{I}_{0,1}, \hat{I}_{0,2},  \hat{I}(0) F_0(\cdot), \hat{E}(0)\Phi_0(\cdot), 
 \hat{R}_{0,1}, \hat{R}_{0,2} \Big) \non 
 \end{align}
 in $D^{12}$ as $n\to\infty$, 
 where the limit processes $ \hat{E}_0$,  $\hat{I}_{0,1}$, $\hat{I}_{0,2}$,  $\hat{R}_{0,1}$ and $\hat{R}_{0,2}$ are given in Theorem \ref{thm-FCLT-SIR}, and $\hat{L}_0$ is  a  mean-zero Gaussian process with the covariance function
$$
\Cov(\hat{L}_0(t), \hat{L}_0(s)) = \bar{E}(0) (G_0(t\wedge s) - G_0(t) G_0(s)), \quad t, s \ge 0.
$$
In addition,
\begin{align}
\Cov( \hat{E}_{0}(t),  \hat{L}_{0}(t')) &= \bar{I}(0)\Big( (G_0(t') -F_0(t)) \bone(t'\ge t) - G_0^c(t) G_0(t') \Big), \non\\
\Cov( \hat{L}_{0}(t),  \hat{I}_{0,2}(t')) &= \bar{E}(0) \left( \int_t^{t'} \bone(t'\ge t) F_0(t'-s|s)dG_0(s) - G_0(t) \Psi_0(t') \right),  \non\\
\Cov( \hat{L}_{0}(t),  \hat{R}_{0,2}(t')) &= \bar{E}(0) \left(\int_t^{t'}F_0(t'-s|s) d G_0(s) - G_0(t)  \Phi_0(t')\right), \non
\end{align}
and $\hat{L}_0$ is independent with the other limit processes of the initial quantities. 
 If $G_0$ and $F_0$ are continuous, then these processes are continuous. 
 \end{lemma}

Recall the definition of PRM $M(ds,dy,dz,du)$ and its compensated PRM in Definition \ref{def-PRM-SEIR-1}. 
\begin{definition} \label{def-PRM-SEIR-2}
 Let $M_1(ds,dy,dz,du)$ be a PRM on $[0,T]\times \RR_+ \times \RR_+ \times  \RR_+$  with  mean measure $\tilde\nu(ds,dy,dz,du) = ds \tilde{H}_s(dy,dz) du $ such that the first marginal of $\tilde{H}_s$ is $\tilde{G}_s((a,b]) = G((a+s, b+s])$ and the conditional distribution 
 $\tilde{F}_{s}((a,b]|y) = F((a+s+y, b+s+y]|y)$. 
 Denote the compensated PRM by $\widetilde{M}(ds,dy, dz,du)$. 
 \end{definition}
 
We use again the notation  $\varphi_n(s,u) = \bone\left(u \le n \lambda \bar{S}^n(s^-)\bar{I}^n(s^-)\right)$. 
 We can rewrite 
 \begin{align*}
   \hat{I}^n_1(t) 
  &= \frac{1}{\sqrt{n}} \int_0^t \int_0^{t-s} \int_{t-s-y}^\infty  \int_0^\infty  \varphi_n(s,u)  \overline{M}(ds,dy,dz,du)  \non\\
  & = \frac{1}{\sqrt{n}} \int_0^t \int_0^{t} \int_t^\infty  \int_0^\infty  \varphi_n(s,u)  \widetilde{M}(ds,dy, dz,du), 
  \end{align*}
  and similarly for the other processes  $\hat{M}^n_A$, $\hat{E}^n_1$, $\hat{L}^n_1$,
 and $\hat{R}^n_1$ (with $\widetilde{M}$ for brevity) as 
 \begin{align} 
 \hat{M}^n_A(t)
 & = \frac{1}{\sqrt{n}} \int_0^t \int_{0}^{\infty} \int_0^\infty  \int_0^\infty  \varphi_n(s,u)  \widetilde{M}(ds,dy, dz,du),  \non\\
  \hat{E}^n_1(t) 
  &= \frac{1}{\sqrt{n}} \int_0^t \int_{t}^{\infty} \int_0^\infty  \int_0^\infty  \varphi_n(s,u)  \widetilde{M}(ds,dy, dz,du),  \non\\
    \hat{L}^n_1(t) 
    & = \frac{1}{\sqrt{n}} \int_0^t \int_0^{t} \int_0^\infty  \int_0^\infty  \varphi_n(s,u)  \widetilde{M}(ds,dy, dz,du),  \non\\
  \hat{R}^n_1(t) 
  &= \frac{1}{\sqrt{n}} \int_0^t \int_0^{t} \int_0^t  \int_0^\infty  \varphi_n(s,u)  \widetilde{M}(ds,dy, dz,du). \non
 \end{align}
 
Observe that 
 \begin{equation}\label{eqn-hatMEL}
  \hat{M}^n_A(t) =   \hat{E}^n_1(t)  +  \hat{L}^n_1(t), \quad t \ge 0,
 \end{equation}
 and
 \begin{equation} \label{eqn-hatEIR}
  \hat{L}^n_1(t)  =  \hat{I}^n_1(t)  +  \hat{R}^n_1(t), \quad t \ge 0. 
 \end{equation}
 We define the auxiliary processes $\widetilde{M}^n_A$, $\widetilde{E}^n_1$, $\widetilde{L}^n_1$,
 $\widetilde{I}^n_1$ and $\widetilde{R}^n_1$ by replacing $\varphi_n(s,u)$ by
  $$
 \widetilde\varphi_n(s,u) = \bone\left(u \le n \lambda \bar{S}(s)\bar{I}(s)\right), 
 $$
 in the corresponding processes using the compensated PRM $\widetilde{M}(ds,dy, dz,du)$. Then we have
 \begin{equation}\label{eqn-tildeMEL}
  \widetilde{M}^n_A(t) =   \widetilde{E}^n_1(t)  +  \widetilde{L}^n_1(t), \quad t \ge 0,
 \end{equation}
 and
 \begin{equation}\label{eqn-tildeEIR}
  \widetilde{L}^n_1(t)  =  \widetilde{I}^n_1(t)  +  \widetilde{R}^n_1(t), \quad t \ge 0. 
 \end{equation}

Similar to Lemma \ref{lem-hatSn-t-2-bound} for the SIR model, we have the following result. We omit its proof for brevity. 
\begin{lemma} \label{lem-hatSn-t-2-bound-SEIR}
\begin{align*}
 \sup_n \E\bigg[\sup_{t \in [0,T]}|\hat{S}^n(t)|^2\bigg] <\infty&, \quad \sup_n \E\bigg[\sup_{t \in [0,T]}|\hat{E}^n(t)|^2\bigg] <\infty, \\  \sup_n \E\bigg[ \sup_{t \in [0,T]}|\hat{I}^n(t)|^2\bigg] <\infty&, \quad   \sup_n \E\bigg[ \sup_{t \in [0,T]}|\hat{R}^n(t)|^2\bigg] <\infty. 
\end{align*}
\end{lemma}

\begin{proof}
The proof for the processes $\hat{S}^n$ and $\hat{E}^n$ follows from the same argument as those of $\hat{S}^n$ and $\hat{I}^n$ in the SIR model. By the representation of $\hat{I}^n$ in \eqref{hatIn-SEIR-rep}, we prove the upper bounds for the processes $\hat{I}^n_{0,1}$, $\hat{I}^n_{0,2}$, and $\hat{I}^n_{1}$, and then apply Gronwall's inequality after  taking the expectation of the square of the equation and using the Cauchy--Schwartz inequality. 
The same arguments for $\hat{I}^n_{0}$ and $\hat{I}^n_{1}$ in the SIR model can be used for the process $\hat{I}^n_{0,1}$ and 
$\hat{I}^n_{1}$, respectively, where we use  the difference $    \widetilde{I}^n_1(t)  =\widetilde{L}^n_1(t)  -  \widetilde{R}^n_1(t)$ 
as shown in \eqref{eqn-tildeEIR} with both $\widetilde{L}^n_1(t)$ and $ \widetilde{R}^n_1(t)$ being martingales. 
Now for the process  $\hat{I}^n_{0,2}$, 
we define 
$$
    \tilde{I}^n_{0,2}(t) = \frac{1}{\sqrt{n}}  \sum_{j=1}^{n \bar{E}(0)} \left(\bone (\xi_j^0 \le t) \bone(\xi_j^0+ \eta_j >t ) -\Psi_0(t) \right). 
$$
We can rewrite $  \tilde{I}^n_{0,2}(t) $ as
\begin{align*}
  \tilde{I}^n_{0,2}(t) =  \frac{1}{\sqrt{n}}  \sum_{j=1}^{n \bar{E}(0)} \left(\bone (\xi_j^0 \le t)  -G_0(t) \right) -  \frac{1}{\sqrt{n}}  \sum_{j=1}^{n \bar{E}(0)} \left( \bone(\xi_j^0+ \eta_j \le t) -\Phi_0(t) \right). 
\end{align*}
Then each term can be treated in the same way as $\tilde{I}^n_{0}$ in the proof of Lemma \ref{lem-hatSn-t-2-bound}, using the Dvoretsky--Kiefer--Wolfowitz inequality. 
The difference $\hat{I}^n_{0,2}(t)  - \tilde{I}^n_{0,2}(t) $ can be also expressed as two terms similarly as the above expression, involving $\bar{E}^n(0)$ and $\bar{E}(0)$, and then each term 
 can be treated similarly as   $\hat{I}^n_{0} - \tilde{I}^n_{0}$ in the SIR model in  the proof of Lemma \ref{lem-hatSn-t-2-bound}.  Thus we obtain the result for $\hat{I}^n(t)$. The process $\hat{R}^n(t)$ can be treated analogously.  
\end{proof}

Then, following an analogous argument as in the proof of Lemma \ref{lem-hat-tilde-diff-SIR}, we obtain the following. 

\begin{lemma} \label{lem-hat-tilde-diff-SEIR}
Under Assumption \ref{AS-SEIR-2}, 
\begin{align*}
(\hat{M}^n_A - \widetilde{M}^n_A, \hat{E}_1^n - \widetilde{E}_1^n, \hat{L}_1^n - \widetilde{L}_1^n, \hat{I}_1^n - \widetilde{I}_1^n, \hat{R}_1^n - \widetilde{R}_1^n) \RA 0 \qinq D^5 \qasq n \to \infty. 
\end{align*}
\end{lemma}

\begin{proof}
By the same argument as in the proof for the SIR model, we obtain the convergence  
$\hat{M}^n_A - \widetilde{M}^n_A \RA 0$, and $ \hat{L}^n_1 - \widetilde{L}^n_1\RA 0$, and thus, by \eqref{eqn-hatMEL} and \eqref{eqn-tildeMEL}, we have 
 $ \hat{E}^n_1 - \widetilde{E}^n_1\RA 0$. 
 We then show that  $ \hat{R}^n_1 - \widetilde{R}^n_1\RA 0$, which will imply 
  $ \hat{I}^n_1 - \widetilde{I}^n_1\RA 0$ by \eqref{eqn-hatEIR} and \eqref{eqn-tildeEIR}. 
 On the other hand, the proof of $ \hat{R}^n_1 - \widetilde{R}^n_1\RA 0$ follows essentially the same argument as that in the SIR model, if we replace the infectious periods by the sum of the exposing and infectious periods. In the analysis we simply replace the distribution function $F$ by the convolution of $F$ and $G$. In particular,  the difference process $\Xi^n= \hat{R}_1^n - \widetilde{R}_1^n$, has
$
\E[\Xi_1^n(t)]  = 0,
$
and 
$$
\E\left[\Xi^n(t)^2\right] =  \int_0^t 
 \Phi(t-s) \E\left[  |\bar{S}^n(s) \bar{I}^n(s) -  \bar{S}(s) \bar{I}(s)| \right] ds,$$
for each $t\ge 0$. 
To show that the sequence $\{\Xi^n: n \ge 1\}$ is tight,  as in the proof of the SIR model, 
it suffices to show the tightness of the processes $ \Xi_1^n(t) $ and $ \Xi_2^n(t) $: 
\begin{align}
 \Xi_1^n(t) &=  \frac{1}{\sqrt{n}} \int_0^t \int_{0}^{t} \int_0^t \int_{n\lambda ( \bar{S}^n(s^-) \bar{I}^n(s^-) \wedge \bar{S}(s) \bar{I}(s))}^{n\lambda ( \bar{S}^n(s^-) \bar{I}^n(s^-) \vee \bar{S}(s) \bar{I}(s))}  M_1(ds,dy,dz,du),  \non\\
\Xi_2^n(t) &= \lambda \sqrt{n}
\int_0^t   \Phi(t-s) 
 \big|\bar{S}^n(s) \bar{I}^n(s) - \bar{S}(s) \bar{I}(s)\big| ds. \non  
\end{align}
It suffices to show that \eqref{Xi-n-tight-p} holds for each process. 
Both processes $ \Xi_1^n(t) $ and $ \Xi_2^n(t) $ are increasing in $t$. 
The proof then follows step by step and it requires the condition:
\begin{equation}\label{3stars}
\limsup_{n\to\infty} \frac{1}{\delta}  \E\left[ \left(\int_0^t 
(\Phi(t+\delta -s) - \Phi(t-s))  \Delta^n(s) ds \right)^2 \right] \to 0  
\end{equation}
as $\delta \to 0$. 
We observe that 
\begin{align}
 & \Phi(t+\delta -s) - \Phi(t-s)  \non\\
 &= \int_0^{t+\delta-s} F(t+\delta-s-u|u)dG(u)  -  \int_0^{t-s} F(t-s-u|u)dG(u)\non\\
& = \int_{t-s}^{t+\delta-s} F(t+\delta-s-u|u)dG(u)  +  \int_0^{t-s} (F(t+\delta-s-u|u)- F(t-s-u|u))dG(u). \non
\end{align}
Thus, we have
\begin{align}
&\E\left[ \left(\int_0^t 
(\Phi(t+\delta -s) - \Phi(t-s))  \Delta^n(s) ds \right)^2 \right] \non\\
& \le 2\E\left[ \left(\int_0^t 
 \int_{t-s}^{t+\delta-s} F(t+\delta-s-u|u)dG(u) \Delta^n(s) ds \right)^2 \right] \non\\
& \quad + 2\E\left[ \left(\int_0^t 
\int_0^{t-s} (F(t+\delta-s-u|u)- F(t-s-u|u))dG(u)  \Delta^n(s) ds \right)^2 \right]. \non 
\end{align}  
The first term can be bounded by
$$
2 \E\left[ \left(\int_0^t  (G(t+\delta-s) -G(t-s))\Delta^n(s) ds \right)^2 \right]
$$
which can be dealt with in the same way as was done for the SIR model. 
Concerning the second term, by interchanging the order of integration and using Jensen's inequality, we have
\begin{align}
& \E\left[ \left(\int_0^t 
\int_0^{t-s} (F(t+\delta-s-u|u)- F(t-s-u|u)) \Delta^n(s) ds dG(u)   \right)^2 \right] \non\\
& \le \E\left[ \int_0^t \left(
\int_0^{t-u} (F(t+\delta-s-u|u)- F(t-s-u|u)) \Delta^n(s) ds \right)^2 dG(u)   \right].  \non
\end{align}
Exploiting Lemma \ref{lem-hatSn-t-2-bound-SEIR}, we can show that this term is at most of the order of  $o(\delta)$ as in the SIR model. This completes the proof.
\end{proof}

Let 
$$\sG^{A}_t := \sigma \left\{\widetilde{M}([0,u]\times\RR_+^3):   0 \le u \le t \right\}, \quad t \ge 0, 
$$
$$\sG^{L}_t := \sigma \left\{\widetilde{M}([0,u]\times[0,u]\times\RR_+^2):   0 \le u \le t \right\}, \quad t \ge 0, 
$$
and 
$$\sG^{R}_t := \sigma \left\{\widetilde{M}([0,u]\times[0,u]\times[0,u]\times\RR_+):   0 \le u \le t \right\}, \quad t \ge 0. 
$$
It is clear that $\widetilde{M}_A^n$ is a $\{\sG^{A,n}_t: t \ge0\}$-martingale with quadratic variation
 $$
\langle\widetilde{M}^n_A  \rangle(t) = \lambda \int_0^t \bar{S}(s) \bar{I}(s) ds, \quad t\ge 0,
$$
 $\widetilde{L}_1^n$ is a $\{\sG^{L,n}_t: t \ge0\}$-martingale with quadratic variation
 $$
\langle\widetilde{L}^n_1  \rangle(t) = \lambda \int_0^t G(t-s) \bar{S}(s) \bar{I}(s) ds, \quad t\ge 0,
$$
and  $\widetilde{R}_1^n$ is a 
 $\{\sG^{R,n}_t: t \ge0\}$-martingale with quadratic variation
$$
\langle\widetilde{R}^n_1  \rangle(t) = \lambda \int_0^t \Phi(t-s) \bar{S}(s) \bar{I}(s) ds, \quad t\ge 0. 
$$
Note that  we do not have a martingale property for $\widetilde{E}^n$ nor $\widetilde{I}^n$, and
like in the SIR model, it is important to observe that the joint process $(\widetilde{M}_A^n, \widetilde{L}_A^n, \widetilde{R}^n_1)$ is not a martingale with respect to a common filtration, and we only use their individual martingale property to conclude their tightness.

\begin{lemma}\label{lem-hatMARI1-conv-SEIR}
Under Assumption \ref{AS-SEIR-2}, 
\begin{equation} 
( \hat{M}_A^n, \hat{E}^n_1, \hat{L}^n_1, \hat{I}^n_1, \hat{R}_1^n) \RA (\hat{M}_1,  \hat{E}_1, \hat{L}_1, \hat{I}^n,\hat{R}_1) \qinq D^5 \qasq n \to\infty, \non
\end{equation}
where $ (\hat{M}_A, \hat{E}_1, \hat{I}_1, \hat{R}_1)$ are given in Theorem~\ref{thm-FCLT-SIR}, and  $\hat{L}_1$ is a continuous Gaussian process with covariance function: for $t, t'\ge 0$, 
$$
\Cov(\hat{L}_1(t), \hat{L}_1(t')) =  \lambda \int_0^{t\wedge t'} G(t\vee t'-s) \bar{S}(s) \bar{I}(s) ds, 
$$
and it has covariance functions with the other processes: for $t, t'\ge 0$, 
\begin{align}
\Cov(\hat{M}_A(t), \hat{L}_1(t')) & = \lambda \int_0^{t\wedge t'} G(t'-s) \bar{S}(s) \bar{I}(s) ds,\non\\
\Cov(\hat{E}_1(t), \hat{L}_1(t')) & = \lambda \int_0^{t\wedge t'}  (G(t'-s) - G(t-s))\bone(t'\ge t)  \bar{S}(s) \bar{I}(s) ds,\non\\
\Cov(\hat{L}_1(t), \hat{I}_1(t')) & = \lambda \int_0^{t\wedge t'}  (G(t-s) - \Psi(t'-s))\bone(t'\ge t)  \bar{S}(s) \bar{I}(s) ds,\non\\
\Cov(\hat{L}_1(t), \hat{I}_R(t')) & = \lambda \int_0^{t\wedge t'}(G(t-s) - \Phi(t'-s))\bone(t'\ge t) \bar{S}(s) \bar{I}(s) ds. \non 
\end{align}

\end{lemma}

\begin{proof} In view of Lemma \ref{lem-hat-tilde-diff-SEIR}, it suffices to prove that
\begin{equation} \label{wt-MARI1-conv}
(\widetilde{M}_A^n,  \widetilde{E}_1^n, \widetilde{L}_1^n, \widetilde{I}_1^n, \widetilde{R}^n_1)  \RA (\hat{M}_A,  \hat{E}_1, \hat{L}_1, \hat{I}_1,\hat{R}_1)  \qinq D^5 \qasq n \to\infty.
\end{equation}
Using the martingale property of   $\widetilde{M}^n_A$,  $\widetilde{L}^n_1$ and $ \widetilde{R}^n_1$, we establish tightness of each of these processes in $D$. 
Moreover each of the possible limit being continuous, the differences $\widetilde{I}^n_1(t)  =
   \widetilde{L}^n_1(t) -\widetilde{R}^n_1(t)$, and  $\widetilde{E}^n_1(t)  =
   \widetilde{M}^n_A(t) -\widetilde{L}^n_1(t)$ are tight. 
  Lemma \ref{lem-hat-tilde-diff-SEIR} now implies that $\{\hat{M}^n_A\}$,  $\{\hat{E}^n_A\}$, , $\{\hat{I}_1^n\}$, and $\{ \hat{R}^n_1\}$ are tight. 
 We next show \eqref{wt-MARI1-conv} by proving  (i) convergence of finite dimensional distributions of  $(\widetilde{M}_A^n,  \widetilde{E}_1^n, \widetilde{L}_1^n, \widetilde{I}_1^n, \widetilde{R}^n_1) $ and (ii) the limits are continuous. 

To prove the convergence of finite dimensional distributions, by the independence of the restrictions of a PRM to disjoint subsets, 
 it suffices to show that for $0 \le t' \le t$, $0 \le a \le b < \infty$ and $0 \le c \le d < \infty$, 
\begin{align}\label{eqn-char-SEIR}
& \lim_{n\to\infty} \E\left[ \exp \left(i \frac{\vartheta}{\sqrt{n}} \int_{t'}^t \int_a^b \int_c^d\int_0^\infty  \widetilde{\varphi}^n(s)  \widetilde{M}(ds, dy, dz, du)\right)\right]  \non\\
& = \exp \left( -\frac{\vartheta^2}{2} \lambda \int_{t'}^t \left( \int_a^b \int_c^d \tilde{H}_s(dy,dz) \right) \bar{S}(s) \bar{I}(s) ds \right) \,,
\end{align}
where 
$$
 \int_a^b \int_c^d \tilde{H}_s(dy,dz) = \int_{a-s}^{b-s} (F(d-y-s|y) - F(c-y-s|y)) G(dy). 
$$
By \eqref{eqn-PRM-char},  the left hand side of \eqref{eqn-char-SEIR} is equal to 
\begin{align}
& \exp\left( - i \frac{\vartheta}{\sqrt{n}} \int_{t'}^t \left( \int_a^b \int_c^d \tilde{H}_s(dy,dz) \right)  \lambda n \bar{S}(s) \bar{I}(s) d s  \right) \non\\
& \times \exp\left( (e^{i \vartheta/\sqrt{n}}-1) \int_{t'}^t   \left( \int_a^b \int_c^d \tilde{H}_s(dy,dz) \right) \lambda n \bar{S}(s) \bar{I}(s) d s\right). \non
\end{align}
Then the claim in \eqref{eqn-char-SIR} is immediate by applying Taylor expansion. 

We next show that there exists a continuous version of the limit processes $\hat{M}_A$, $\hat{E}_1$, $\hat{I}_1$ and $\hat{R}_1$ in $C$. Taking  $\hat{R}_1$ as an example, we need to show \eqref{E-hatR1-4th} holds. 
By  \eqref{eqn-char-SEIR},  we have 
\begin{align}
&\E\left[( \hat{R}_1(t+\delta) -  \hat{R}_1(t)))^4 \right] = 3 \left(E\left[( \hat{R}_1(t+\delta) -  \hat{R}_1(t)))^2 \right]\right)^2 \non\\
& = 3 \left( \lambda\int_t^{t+\delta} \Phi(t+\delta -s) \bar{S}(s) \bar{I}(s) ds + \lambda \int_0^t (\Phi(t+\delta -s) - \Phi(t-s)) \bar{S}(s) \bar{I}(s) ds  \right)^2   \non\\
& \le 6 \lambda \delta^2 + 6 \lambda \left( \int_0^t (\Phi(t+\delta -s) - \Phi(t-s)) \bar{S}(s) \bar{I}(s) ds\right)^2. \non
\end{align}
This implies that \eqref{E-hatR1-4th} holds, see the computations for the proof of \eqref{3stars} above. 
This completes the proof. \end{proof}

\smallskip
\noindent \emph{Completing the proof of Theorem \ref{thm-FCLT-SIR}.}
By Lemmas \ref{lem-hatInitials-conv-SEIR} and \ref{lem-hatMARI1-conv-SEIR}, 
 we first obtain the joint convergence 
 \begin{align}
 &\Big(-\hat{I}^n(0)-\hat{M}_A^n,  \hat{E}^n(0) G_0^c(\cdot)  + \hat{E}_0^n  + \hat{E}_1^n, 
  \hat{I}^n(0) F^c_0(\cdot) +  \hat{E}^n(0) \Psi_0(\cdot)   +  \hat{I}^n_{0,1} + \hat{I}^n_{0,2} +  \hat{I}^n_{1}, \non\\
  & \qquad  \hat{I}^n(0) F_0(\cdot) +  \hat{E}^n(0) \Phi_0(\cdot) +  \hat{R}^n_{0,1}  + \hat{R}^n_{0,2} +  \hat{R}^n_{1} \Big) \non\\
  & \RA \Big(-\hat{I}(0)-\hat{M}_A,  \hat{E}(0) G_0^c(\cdot)  + \hat{E}_0  + \hat{E}_1, 
  \hat{I}(0) F^c_0(\cdot) +  \hat{E}(0) \Psi_0(\cdot)   +  \hat{I}_{0,1} + \hat{I}_{0,2} +  \hat{I}_{1}, \non\\
  & \qquad  \hat{I}(0) F_0(\cdot) +  \hat{E}(0) \Phi_0(\cdot) +  \hat{R}_{0,1}  + \hat{R}_{0,2} +  \hat{R}_{1} \Big) \non
 \end{align}
 in $D^4$ as $n\to\infty$. 
Then by Lemma \ref{lem-Gamma-cont} and  the continuous mapping theorem, we obtain \eqref{eqn-FCLT-conv-SEIR}.  \hfill $\Box$

As a consequence of the above proof, we also obtain the convergence $\hat{L}^n\RA \hat{L}$ in $D$ as $n\to\infty$, jointly with the processes in \eqref{eqn-FCLT-conv-SEIR}, where
\begin{align}
\hat{L}(t)  &=  \hat{E}(0) G_0(t)  + \hat{L}_0(t)  + \hat{L}_1(t)   +  \lambda  \int_0^t G(t-s) \left( \hat{S}(s) \bar{I}(s) + \bar{S}(s) \hat{I}(s) \right) ds,  \quad t\ge 0. \non 
\end{align}

\medskip

\section{Appendix}
\subsection{A system of two linear Volterra integral equations}

Define the mapping $\Gamma: (a, x,y,z) \to (\phi,\psi)$  by the integral equations:
\begin{align}\label{eqn-phi-psi}
\phi(t) &= a + x(t) + c \int_0^t (\phi(s) z(s) + w(s) \psi(s) )ds,  \non\\
\psi(t) &=  y(t) + c \int_0^t K(t-s)(\phi(s) z(s) + w(s) \psi(s) )ds, 
\end{align}
where $(a, x,y,z)  \in \RR\times D^3$, and $c>0$ and $w\in C$. (Here $c$ and $w$ are given and fixed.) 
We study the existence and uniqueness of its solution and the continuity property in the Skorohod $J_1$ topology.

\begin{lemma}\label{lem-Gamma-cont}
Assume that $K(0) =0$ and $K(\cdot)$ is measurable,  bounded and continuous, and let $c>0$ and $w\in C$ be given.  
There exists a unique solution $(\phi,\psi) \in D^2$ to the integral equations \eqref{eqn-phi-psi}. 
The mapping $\Gamma$ is continuous in the Skorohod topology, that is, if  $a^n\to a$ in $\RR$ and $(x^n,y^n, z^n) \to (x, y, z)$ in $D^3$ as $n\to\infty$ with $(x,z) \in C^2$ and $y\in D$, then $(\phi^n,\psi^n) \to (\phi,\psi)$ in $D^2$ as $n\to\infty$. 
In addition, if $y \in C$, then $(\phi,\psi) \in C^2$, and the mapping is continuous uniformly on compact sets in $[0,T]$. 
\end{lemma}

\begin{proof}

By Theorems 1.2 and 2.3 in Chapter II of \cite{miller1971},  if $x,y \in C$,  we have  existence and uniqueness of a solution $(\phi,\psi) \in C^2$ to the integral equations \eqref{eqn-phi-psi}. 
The proof can be easily extended to the case where $x,y \in D$ by applying the Schauder-Tychonoff fixed point theorem.

We next show the continuity in the Skorohod $J_1$ topology. Note that the functions in $D$ are necessarily bounded.  For the given $(x,z) \in C^2$ and $y\in D$, let the interval right end point $T$ be a continuity point of $y$. 
Since $(x,z) \in C^2$,  the convergence $(x^n,y^n, z^n) \to (x, y, z)$ in $D^3$
 in the product $J_1$ topology is equivalent to convergence $(x^n,y^n, z^n) \to (x, y, z)$ in $D([0,T], \RR^3)$  in the strong $J_1$ topology. 
 Then there exist increasing homeomorphisms $\lambda^n$ on $[0,T]$ such that $\|\lambda^n-e\|_T \to 0$, $\|x^n- x\circ \lambda^n\|_T \to 0$, $\|y^n- y\circ \lambda^n\|_T \to 0$, and $\|z^n- z\circ \lambda^n\|_T \to 0$,  as $n \to \infty$. Here $e(t) := t$ for all $t\ge 0$. 
 Moreover, it suffices to consider homeomorphisms $\lambda^n$ that are absolutely continuous with resect to the Lebesgue measure on $[0,T]$ having derivatives $\dot{\lambda}^n$ satisfying $\| \dot{\lambda}^n -1 \|_T \to 0$ as $n\to \infty$. Let $\sup_{t \in [0,T]}|K(t)| \le c_K$. 
 
 We have 
 \begin{align}
 & |\phi^n(t) - \phi(\lambda^n(t))|  \non \\
 & \le |a^n-a| + \|x^n- x\circ \lambda^n\|_T   + c \left| \int_0^t (\phi^n(s) z^n(s) + w(s) \psi^n(s) )d s - \int_0^{\lambda^n(t)} (\phi(s) z(s) + w(s) \psi(s) )d s \right| \non\\
 & \le |a^n-a| + \|x^n- x\circ \lambda^n\|_T  +   c \Bigg| \int_0^t (\phi^n(s) z^n(s) + w(s) \psi^n(s) )d s  \non\\
 & \qquad\qquad \qquad \qquad\qquad \qquad \qquad  - \int_0^{t} (\phi(\lambda^n(s)) z(\lambda^n(s)) + w(\lambda^n(s)) \psi(\lambda^n(s)) ) \dot{ \lambda}^n(s) d s \Bigg| \non\\
& \le |a^n-a| + \|x^n- x\circ \lambda^n\|_T  + c \| \dot{\lambda}^n -1 \|_T \int_0^T  |\phi(s) z(s) + w(s) \psi(s) | d s \non \\
 & \quad +  c\int_0^t \Big(  |\phi^n(s) - \phi(\lambda^n(s))| |z^n(s)| + |\phi(\lambda^n(s))| |z^n(s) - z(\lambda^n(s))| 
 \non\\
 & \qquad \qquad \qquad + |w(s) - w(\lambda^n(s))| |\psi^n(s)| + |w(\lambda^n(s))| |\psi^n(s) - \psi(\lambda^n(s))| \Big) ds \non 
 \end{align}
 and similarly,
 \begin{align}
 &  |\psi^n(t) - \psi(\lambda^n(t))|  \non\\
 & \le \|y^n- y\circ \lambda^n\|_T   + c \Bigg| \int_0^t  K(t-s)(\phi^n(s) z^n(s) + w(s) \psi^n(s) )ds \non\\
  & \qquad \qquad \qquad\qquad \qquad   - \int_0^{\lambda^n(t)} K(\lambda^n(t)-s) (\phi(s) z(s) + w(s) \psi(s) )d s \Bigg| \non \\
  & \le  \|y^n- y\circ \lambda^n\|_T  + c \times c_K \| \dot{\lambda}^n -1 \|_T \int_0^T  |\phi(s) z(s) + w(s) \psi(s) | d s  \non\\
  & \quad +  c \times c_K \int_0^t \Big(  |\phi^n(s) - \phi(\lambda^n(s))| |z^n(s)| + |\phi(\lambda^n(s))| |z^n(s) - z(\lambda^n(s))| 
 \non\\
 & \qquad \qquad \qquad \qquad  + |w(s) - w(\lambda^n(s))| |\psi^n(s)| + |w(\lambda^n(s))| |\psi^n(s) - \psi(\lambda^n(s))| \Big) ds  \non\\
 & \quad +  c \int_0^t \big|K(t-s)- K(\lambda^n(t) -\lambda^n(s))\big| (\phi^n(s) z^n(s) + w(s) \psi^n(s) )ds.  \non
 \end{align}
By first applying Gronwall's inequality and then using the convergence of $a^n\to a$ in $\RR$ and $(x^n,y^n, z^n) \to (x, y, z)$ in $D^3$, and $w\in C$,  we obtain 
$$
 \|\phi^n - \phi\circ\lambda^n\|_T +  \|\psi^n - \psi\circ\lambda^n\|_T \to 0 \qasq n \to \infty. 
$$
This completes the proof of the continuity property in the Skorohod $J_1$ topology. If $y\in C$, the continuity property is straightforward.
\end{proof}

\subsection{Proof of Proposition~\ref{prop-SIS-equiv-M}} \label{sec-proof-equivalence}

\begin{proof}[Proof of Proposition~\ref{prop-SIS-equiv-M}]
Recall that the unique solution of the linear differential equation: $x(t) = x(0) + a\int_0^t x(s)ds + y(t)$ with $y(0)=0$, is given by the formula $x(t) = e^{at} x(0) + \int_0^t a e^{a(t-s)}y(s)ds + y(t)$, for $t\ge0$, and if $y \in C^1$, we have $x(t) = e^{at} x(0) + \int_0^t e^{a(t-s)}\dot{y}(s)ds$. 

Let $X_1(t) = \hat{I}(0) e^{-\mu t} $. 
We have 
\begin{equation} \label{eqn-X1-SIS-M}
X_1(t) = -\mu \int_0^t X_1(s) ds +  \hat{I}(0). 
\end{equation}
Let 
$$X_2(t) = \lambda \int_0^t e^{-\mu(t-s)}  (1- 2 \bar{I}(s))   \hat{I}(s) ds. $$
We have 
\begin{equation} \label{eqn-X2-SIS-M}
 X_2(t) = - \mu\int_0^t X_2(s)ds  + \lambda \int_0^t   (1- 2 \bar{I}(s))   \hat{I}(s) ds. 
\end{equation}
 For $\hat{I}_0(t)$, its covariance is 
 $$
\Cov(\hat{I}_{0}(t), \hat{I}_{0}(t')) = \bar{I}(0) (e^{-\mu(t\vee t')} - e^{-\mu t} e^{-\mu t'}) , \quad t, t' \ge 0. 
$$
It is easy to verify that 
\begin{equation} \label{eqn-I0-SIS-M}
\hat{I}_0(t) = -\mu \int_0^{t} \hat{I}_0(s) ds +  W_0(t)
\end{equation}
where $W_0(t)  =\bar{I}(0)^{1/2}  B_0(1-e^{-\mu t})$ for  a standard Brownian motion $B_0$. 
We can represent $W_0(t) = \bar{I}(0)^{1/2} \int_0^t \sqrt{\mu e^{-\mu s}} d \tilde B_0(s)$ for another Brownian motion $B_0$, and thus write 
$$
\hat{I}_0(t) = \bar{I}(0)^{1/2}\int_0^t e^{-\mu(t-s)} \sqrt{\mu e^{-\mu s}} d \tilde{B}_0(s), \quad t \ge 0,
$$
which gives the same covariance as above by It{\^o}'s isometry property. 

For $\hat{I}_1$, its covariance is 
\begin{equation} \label{SIS-cov-hatI1-M}
\Cov(\hat{I}_{1}(t), \hat{I}_{1}(t')) =  \lambda \int_0^{t\wedge t'} e^{-\mu(t\vee t'-s)}  (1- \bar{I}(s))  \bar{I}(s) ds, \quad t, t'\ge 0.
\end{equation}
We next show that 
\begin{equation} \label{eqn-I1-SIS-M}
\hat{I}_1(t) = -\mu \int_0^t \hat{I}_1(s) ds + W_1(t)
\end{equation}
where $W_1(t)$ is a  continuous Gaussian process, independent of $W_0(t)$, with the covariance function 
$$
\Cov(W_1(t), W_1(t')) = \int_0^{t\wedge t'} \theta(r) dr
$$
where 
$$
\theta(r) :=  \lambda  (1- \bar{I}(r))  \bar{I}(r) +  \mu \bar{I}(r)  -   \bar{I}(0) \mu e^{-\mu r}. 
$$
We have
$$
\hat{I}_1(t) = -\mu \int_0^t e^{-\mu (t-s)} W_1(s) d s + W_1(t), \quad t\ge 0.
$$
We compute the covariance $\Cov(\hat{I}_1(t), \hat{I}_1(s)) $ using this expression: for $t>s$, 
\begin{align}
\Cov(\hat{I}_1(t), \hat{I}_1(s)) 
&= \E\left[ W_1(t) W_1(s) \right] -\mu \E\left[ W_1(t) \int_0^s e^{-\mu (s-r)}  W_1(r) dr \right]   \non\\
& \qquad \qquad
-\mu \E\left[ W_1(s) \int_0^t e^{-\mu (t-r)} W_1(r) d r \right]  \non\\
& \qquad \qquad  + \mu^2 \E\left[ \int_0^t  \left( \int_0^s e^{-\mu (t-r)}   e^{-\mu (s-r')} W_1(r) W_1(r') d r' \right) d r ) \right].\non
\end{align}
The first term is 
$$
\E\left[ W_1(t) W_1(s) \right]  = \int_0^s \theta(u) du.
$$
The second term is 
\begin{align}
 -\mu\int_0^s e^{-\mu (s-r)} \E\left[ W_1(t) W_1(r) \right]  dr  &= -\mu   \int_0^s e^{-\mu (s-r)} \left(\int_0^r \theta(u) du\right)  dr   = -\int_0^s (1-e^{-\mu(s-r)})  \theta(r)dr.  \non
\end{align}
The third term is
\begin{align}
 -\mu  \int_0^t e^{-\mu (t-r)} \E\left[ W_1(s)  W_1(r)  \right] d r  
&  = -\mu\int_0^s e^{-\mu (t-r)} \left(  \int_0^r \theta(u) du \right) dr - \mu \int_s^t e^{-\mu (t-r)} \left( \int_0^s  \theta(u) du \right) dr \non \\
& = - e^{-\mu (t-s)} \int_0^s  (1-e^{-\mu(s-r)}) \theta(r)dr - (1-e^{-\mu (t-s)}) \int_0^s \theta(u)du \non\\
& = - \int_0^s (1- e^{-\mu (t-r)})\theta(r) dr. \non
\end{align}
The fourth term is
\begin{align}
& \mu^2 \int_0^t\left( \int_0^s e^{-\mu (t-r)} e^{-\mu (s-r')} \E[W_1(r) W_1(r')] dr' \right) dr \non \\
& =  \mu^2 \int_s^t\left( \int_0^s e^{-\mu (t-r)} e^{-\mu (s-r')} \E[W_1(r) W_1(r')] dr' \right) dr \non \\
& \quad + \mu^2 \int_0^s\left( \int_0^s e^{-\mu (t-r)} e^{-\mu (s-r')} \E[W_1(r) W_1(r')] dr' \right) dr \non \\
& =  \mu^2 \int_s^t\left( \int_0^s e^{-\mu (t-r)} e^{-\mu (s-r')} \left( \int_0^{r'} \theta(u) du \right) dr' \right) dr \non \\
& \quad +2 \mu^2 \int_0^s\left( \int_0^r e^{-\mu (t-r)} e^{-\mu (s-r')} \left( \int_0^{r'} \theta(u) du \right) dr' \right) dr  \non \\
&= (1-e^{-\mu(t-s)}) \int_0^s (1-e^{-\mu(s-r)}) \theta(r)d r \non\\
& \quad + e^{-\mu (t-s)}  \int_0^s  ( 1- 2 e^{-\mu(s-r)} + e^{-2\mu (s-r)})\theta(r)  dr \non \\
& = \int_0^s  (1- e^{-\mu(s-r)} - e^{-\mu (t-r)} + e^{-\mu (t-r) - \mu (s-r)}) \theta(r) dr. \non
\end{align}
Combining the four terms, we obtain
\begin{align}
e^{-\mu (t-s)} \int_0^s e^{-2\mu (s-r)} \theta(r) dr. \non
\end{align}
Now we check that  this is equal to the covariance of $\hat{I}_1(t)$ in \eqref{SIS-cov-hatI1-M}. Taking the difference
between the last expression and the right--hand side of \eqref{SIS-cov-hatI1-M} with $t'=s<t$, we obtain
\begin{align} \label{SIS-equiv-diff}
& e^{-\mu (t-s)} \int_0^s e^{-2\mu (s-r)} \theta(r) dr -  \lambda \int_0^{s} e^{-\mu(t-r)}  (1- \bar{I}(r))  \bar{I}(r) dr  \non \\
& = e^{-\mu (t-s)} \int_0^s e^{-2\mu (s-r)} ( \lambda  (1- \bar{I}(r))  \bar{I}(r) +  \mu \bar{I}(r) ) dr -  \lambda e^{-\mu(t-s)}\int_0^{s} e^{-\mu(s-r)}  (1- \bar{I}(r))  \bar{I}(r) dr   \non\\
& \qquad -  e^{-\mu (t-s)} \int_0^s e^{-2\mu (s-r)} \bar{I}(0) \mu e^{-\mu r} dr. 
\end{align}

Observe that the fluid equation for $I(t)$ can be written as 
$$
\bar{I}'(t) = - \mu \bar{I}(t) + \lambda \bar{I}(t) (1-\bar{I}(t)), 
$$ 
and
$$
\bar{I}'(t) = -2 \mu \bar{I}(t) + \lambda \bar{I}(t) (1-\bar{I}(t)) + \mu \bar{I}(t).
$$ 
These two equations give the following representations of $\bar{I}(t)$: 
$$
\bar{I}(t) = \bar{I}(0) e^{-\mu s} + \lambda\int_0^s e^{-\mu (s-r)} \bar{I}(r) (1- \bar{I}(r)) dr, 
$$
and
$$
\bar{I}(t) = \bar{I}(0) e^{-2\mu s} + \int_0^s e^{-2\mu (s-r)} \left(\lambda\bar{I}(r) (1- \bar{I}(r)) + \mu \bar{I}(r)\right) dr. 
$$
Also notice that $ \int_0^s e^{-2\mu (s-r)} \mu e^{-\mu r} dr  = e^{-\mu s}- e^{-2\mu s}$. Using these equations, we verify that \eqref{SIS-equiv-diff} is equal to zero, and thus the equation for $\hat{I}_1$ in \eqref{eqn-I1-SIS-M} is established. 
Therefore, by combining \eqref{eqn-X1-SIS-M}, \eqref{eqn-X2-SIS-M} \eqref{eqn-I0-SIS-M} and \eqref{eqn-I1-SIS-M}, we obtain the equivalence of the non-Markovian and Markovian representations of $\hat{I}$ for the SIS model.
\end{proof}

\paragraph{\bf Acknowledgement}
This work was mostly done during G. Pang's visit at Aix--Marseille Universit{\'e}, whose hospitality was greatly appreciated. 
G. Pang was supported in part by  the US National Science Foundation grants DMS-1715875 and DMS-2108683, and Army Research Office grant W911NF-17-1-0019. 
The authors thank the reviewers for the helpful comments that  have improved the exposition of the paper.

\bibliographystyle{plain}

\bibliography{SIRSEIR}

\end{document}